\documentclass[a4paper,12pt]{article}
\usepackage{amsmath,amsthm,verbatim,amssymb,amsfonts,amscd, graphicx}
\usepackage[all]{xy}
\usepackage{xcolor}
\usepackage{geometry}
\usepackage{graphicx}
\usepackage{hyperref}
\usepackage{enumitem}
\usepackage{caption}
\usepackage[overload]{empheq}
\usepackage{tocloft}
\usepackage{indentfirst}
\definecolor{darkblue}{rgb}{0,0,0.5}
\definecolor{urlblue}{rgb}{0,0,0.7}
\hypersetup{
	colorlinks=true,
	citecolor=blue,
	linkcolor=darkblue,
	urlcolor=urlblue,
}

\newcommand{\RR}{\mathbb{R}}
\newcommand{\ZZ}{\mathbb{Z}}
\newcommand{\HH}{\mathbb{H}}
\newcommand{\NN}{\mathbb{N}}

\DeclareMathOperator{\tr}{tr}

\let\div\undefined

\DeclareMathOperator{\diam}{diam}
\DeclareMathOperator{\div}{div}

\DeclareMathOperator{\Ric}{Ric}

\DeclareMathOperator{\supp}{supp}
\DeclareMathOperator{\spt}{spt}
\DeclareMathOperator{\Lip}{Lip}
\DeclareMathOperator{\loc}{loc}

\newcommand{\D}{\nabla}
\newcommand{\p}{\partial}
\renewcommand{\H}{\mathcal{H}}
\renewcommand{\L}{\mathcal{L}}
\renewcommand{\th}{\theta}
\newcommand{\metric}[2]{\langle#1,#2\rangle}

\renewcommand{\bar}{\overline}
\renewcommand{\tilde}{\widetilde}
\renewcommand{\epsilon}{\varepsilon}
\renewcommand{\leq}{\leqslant}
\renewcommand{\geq}{\geqslant}

\newcommand{\TODO}[1]{\message{TODO Warning: {#1}}}
\newcommand{\updatetag}[2]{}

\newtheorem{theorem}{Theorem}[section]
\newtheorem{lemma}[theorem]{Lemma}
\newtheorem{cor}[theorem]{Corollary}
\newtheorem{prop}[theorem]{Proposition}
\newtheorem{defn}[theorem]{Definition}
\newtheorem{summaryofdef}[theorem]{Summary of Definitions\,/\,Theorems}

\newtheorem{question}[theorem]{Question}

\numberwithin{equation}{section}
\theoremstyle{definition}

\theoremstyle{definition}
\newtheorem{remark}[theorem]{Remark}
\theoremstyle{definition}
\newtheorem{example}[theorem]{Example}

\newcommand{\IMCF}[1]{\mathsf{IMCF}\pmb{(}{#1}\pmb{)}}
\newcommand{\IVP}[1]{\mathsf{IVP}\pmb{(}{#1}\pmb{)}}
\newcommand{\IMCFOO}[2]{\mathsf{IMCF}\pmb{(}{#1}\pmb{)}\texttt{+}\mathsf{OBS}\pmb{(}{#2}\pmb{)}}
\newcommand{\IMCFOOinthm}[2]{\mathsf{IMCF}\pmb{(}{#1}\pmb{)}\texttt{\emph{+}}\mathsf{OBS}\pmb{(}{#2}\pmb{)}}
\newcommand{\IVPOO}[2]{\mathsf{IVP}\pmb{(}{#1}\pmb{)}\texttt{+}\mathsf{OBS}\pmb{(}{#2}\pmb{)}}
\newcommand{\IVPOOinthm}[2]{\mathsf{IVP}\pmb{(}{#1}\pmb{)}\texttt{\emph{+}}\mathsf{OBS}\pmb{(}{#2}\pmb{)}}

\DeclareMathOperator{\BV}{BV}
\newcommand{\res}{\llcorner}

\newcommand{\uf}{\underline{f}}
\newcommand{\of}{\overline{f}}
\newcommand{\uu}{\underline{u}}
\newcommand{\ou}{\overline{u}}
\let\graph\undefined
\DeclareMathOperator{\graph}{graph}

\DeclareMathOperator{\ip}{ip}
\DeclareMathOperator{\sip}{sip}
\newcommand{\beuc}{\textbf{euc}}

\newcommand{\DSigma}{\D_\Sigma}
\newcommand{\bnu}{\boldsymbol{\nu}}

\newcommand{\usub}{{\underline{u}_\delta}}
\newcommand{\usup}{{\overline{u}_\delta}}
\newcommand{\ureg}{u_{\epsilon,\delta,\lambda}}

\newcommand{\bg}{\boldsymbol{g}}

\newcommand{\bh}{\boldsymbol{h}}
\newcommand{\br}{\boldsymbol{r}}
\newcommand{\bpr}{\boldsymbol{\p_r}}
\newcommand{\bOmega}{\boldsymbol{\Omega}}

\newcommand{\bPhi}{\boldsymbol{\Phi}}
\newcommand{\be}{\boldsymbol{e}}

\renewcommand{\P}[1]{{P\big(#1\big)}}
\newcommand{\Ps}[1]{{P(#1)}}
\newcommand{\tJ}{\tilde{J}}

\DeclareMathOperator*{\esssup}{ess\,sup}
\DeclareMathOperator*{\essinf}{ess\,inf}

\DeclareMathOperator{\arccosh}{arccosh}

\begin{document}
	
\title{\texorpdfstring{\vspace{-24pt}Inverse mean curvature flow with outer obstacle}{Inverse mean curvature flow with outer obstacle}}
\author{Kai Xu}
\date{\vspace{-12pt}}
\maketitle

\begin{abstract}
	We develop a new boundary condition for the weak inverse mean curvature flow, which gives canonical and non-trivial solutions in bounded domains. Roughly speaking, the boundary of the domain serves as an outer obstacle, and the evolving hypersurfaces are assumed to stick tangentially to the boundary upon contact. In smooth bounded domains, we prove an existence and uniqueness theorem for weak solutions, and establish $C^{1,\alpha}$ regularity of the level sets up to the obstacle. The proof combines various techniques, including elliptic regularization, blow-up analysis, and certain parabolic estimates. As an analytic application, we address the well-posedness problem for the usual weak inverse mean curvature flow, showing that the initial value problem always admits a unique maximal (or innermost) weak solution.
\end{abstract}
	
%

\tableofcontents
	
\section{Introduction}


The theme of this paper is to further investigate the analytic aspects of weak inverse mean curvature flow. We are focused on the study of an outer obstacle boundary condition, and are guided by primary questions on the existence of solutions.

\vspace{9pt}
\phantomsection
\addtocounter{subsection}{1}
\addcontentsline{toc}{subsection}{\arabic{section}.\arabic{subsection}\texorpdfstring{\quad}{} Background and motivations}
\textbf{Background and motivations.}

The classical form of the \textit{inverse mean curvature flow} (IMCF) is an evolution of hypersurfaces given by the equation
\begin{equation}\label{eq-intro:smooth_flow}
	\frac{\p\Sigma_t}{\p t}=\frac{\nu}{H},
\end{equation}
where $\{\Sigma_t\}$ is a smooth foliation of hypersurfaces in a Riemannian manifold $M$, and $\nu$, $H$ denote the outward unit normal and mean curvature of $\Sigma_t$. For analytic considerations, we usually transform \eqref{eq-intro:smooth_flow} into the level set equation (see \cite[Section 4.2.2]{Lee_rel})
\begin{equation}\label{eq-intro:level_set}
	\div\Big(\frac{\D u}{|\D u|}\Big)=|\D u|,
\end{equation}
by setting $u$ as the arrival time function, so that $\Sigma_t=\big\{x: u(x)=t\big\}$ for each $t$.

The long-time existence of \eqref{eq-intro:smooth_flow} is well-studied in the ambient space $\RR^n$: see for the case of compact initial hypersurfaces \cite{Gerhardt_1991, Huisken-Ilmanen_2008,Urbas_1990} and noncompact initial hypersurfaces \cite{Choi-Daskalopoulos_2021, Daskalopoulos-Huisken_2022}. In general, however, we cannot exclude singularity formation in IMCF due to the vanishing of mean curvature. Therefore, we shall consider \textit{weak solutions}. In \cite{Huisken-Ilmanen_2001}, Huisken and Ilmanen developed the notion of weak IMCF, which resolves the singularity issue. This was done by establishing suitable variational principles for the equation \eqref{eq-intro:level_set}, inspired by previous works \cite{Chen-Giga-Goto_1991, Evans-Spruck_1991} on the level set form of mean curvature flow.

A weak solution $u$ of the IMCF \eqref{eq-intro:level_set} is defined as a locally Lipschitz function that locally minimizes the functional
\begin{equation}\label{eq-intro:energy_v_interior}
	J_u^K(v):=\int_K\big(|\D v|+v|\D u|\big).
\end{equation}
Namely, $u$ is a weak solution if $J_u^K(u)\leq J_u^K(v)$ for all $v\in\Lip_{\loc}(\Omega)$ with $\{u\ne v\}\Subset K\Subset\Omega$. An equivalent definition is that each sub-level set $E_t:=\{u<t\}$ minimizes the perimeter-like energy
\begin{equation}\label{eq-intro:energy_E_interior}
	J_u^K(E):=\Ps{E;K}-\int_{E\cap K}|\D u|,
\end{equation}
in the sense that $J_u^K(E_t)\leq J_u^K(E)$ whenever $E\Delta E_t\Subset K\Subset\Omega$. Here we use $\Ps{E;K}$ to denote the perimeter of a set $E$ in $K$. A main feature of weak solutions is the jumping (or fattening) phenomenon, where the level set $\p E_t$ can ``jump'' to its outermost area-minimizing hull, leaving $u\equiv t$ in the skipped region. Weak solutions have nice regularity on manifolds: the current existence theories, namely \textit{elliptic regularization} and \textit{$p$-harmonic approximation}, provide locally Lipschitz solutions. Also, each level set $\p E_t$ of a weak solution is a priori $C^{1,\alpha}$ up to a codimension 8 singular set. (On the other hand, the main issue for existence is the behavior at infinity, see below.) See Section \ref{sec:prelim} for more background material on the smooth and weak IMCF.

The weak IMCF has been a powerful tool in studying scalar curvature problems in dimension three: we refer to the relevant works \cite{Bray-Miao_2008, Bray-Neves_2004, Huisken-Ilmanen_2001, Huisken-Koerber_2023, Lee-Neves_2013, Shi_2016, Xu_2023_sys2}. The applications are based on the well-known monotonicity of Hawking mass, initially discovered by Geroch \cite{Geroch_1973} for smooth flows and extended by Huisken-Ilmanen \cite{Huisken-Ilmanen_2001} to weak flows.



\vspace{3pt}

Let us turn to the main analytic focus of this paper. A motivating question is the well-posedness of the \textit{initial value problem}:
\begin{equation}\label{eq-intro:IVP}
	\left\{\begin{aligned}
		& \div\Big(\frac{\D u}{|\D u|}\Big)=|\D u|\qquad\text{weakly in}\ \ \Omega\setminus\bar{E_0}, \\
		& u|_{\p E_0}=0,\qquad u|_{\Omega\setminus E_0}\geq0.
	\end{aligned}\right. \tag{IVP}
\end{equation}
where $\Omega$ is a domain and $E_0\subset\Omega$ is an initial set. Let us consider two particular cases: (1) $\Omega=M$, and (2) $\Omega$ is a bounded domain with smooth boundary. To obtain a unique solution of \eqref{eq-intro:IVP}, a boundary condition must be assigned at infinity (when $\Omega=M$) or at $\p\Omega$ (when $\Omega$ is bounded).

\vspace{3pt}

(1) The case $\Omega=M$ was the major consideration in the literature. Commonly, one imposes the \textit{properness} condition
\begin{equation}\label{eq-intro:PROPER}
	\lim_{x\to\infty}u(x)=+\infty\ \ \text{uniformly}. \tag{Proper}
\end{equation}
This is equivalent to a more familiar form: $E_t\Subset M$ for all $t>0$. Properness is natural in geometric point of view, and is necessary for obtaining the monotonicity of Hawking mass. Regarding the solvability of \eqref{eq-intro:IVP}\texttt{+}\eqref{eq-intro:PROPER}, it was previously established that:

\begin{enumerate}[nosep]
	\item if there is a proper weak sub-solution of IMCF with bounded initial value \cite{Huisken-Ilmanen_2001}, or
	
	\item if $M$ admits a certain global isoperimetric inequality \cite{Xu_2023_proper}, a particular case of which is an Euclidean-type one $|\p E|\geq C|E|^{(n-1)/n}$, $\forall E\Subset M$,
\end{enumerate}

\noindent then \eqref{eq-intro:IVP}\texttt{+}\eqref{eq-intro:PROPER} has a unique solution for all bounded $E_0$. On the other hand, little was known about non-proper solutions in general. Note that there are abundant cases in which proper solutions do not exist: for instance, when $M$ has a cylindrical end (see \cite{Xu_2023_proper}; the flow must jump to infinity when it fully enters the end). Also, properness is not an expected condition when considering a noncompact initial value $E_0$, such as in \cite{Choi-Daskalopoulos_2021}. As no condition was known to work universally, it is natural to pose in general:
\vspace{-3pt}
\begin{equation}\label{eq-intro:Q_infty}
	\text{to what extent does \eqref{eq-intro:IVP} admit a unique solution?}\tag{$\text{Q}_1$}
\end{equation}

\begin{example}\label{ex-intro:radial_sym}
	To test for a suitable answer, we consider a capped cylinder $g=dr^2+f(r)^2d\th^2$, where
	\[f(r)=\left\{\begin{aligned}
		& \sin r\qquad(0\leq r\leq \pi/2), \\
		& 1\qquad(r>\pi/2).
	\end{aligned}\right.\]
	Let $E_0=\{r<\epsilon\}$ be an initial value with $\epsilon<\pi/2$. The unique solution that we seek in \eqref{eq-intro:Q_infty} should be invariant under rotations. Thus, we shall restrict our search within radial functions $u=u(r)$. It turns out that all radial solutions of \eqref{eq-intro:IVP} take the form
	\begin{equation}\label{eq-intro:classify}
		u(r)=\min\Big\{\log\big[f(r)/f(\epsilon)\big], t_0\Big\}\qquad\text{for some $t_0>0$.}
	\end{equation}
	See Lemma \ref{lemma-prelim:sphere_sym} for the proof. Among the solutions in \eqref{eq-intro:classify}, there is a maximal one $u_0(r)=\log\big[f(r)/f(\epsilon)\big]$: it represents a smooth IMCF motion until reaching the equator $r=\pi/2$, then followed by a jump over the whole cylinder $\{r>\pi/2\}$. The other solutions are truncations of $u_0$, representing sudden jumps to infinity at an earliler time $t_0$. In this example, the best candidate for \eqref{eq-intro:Q_infty} seems to be the \textit{maximal solution} $u_0$. 
\end{example}

Another case appears in Choi-Daskalopoulos \cite{Choi-Daskalopoulos_2021}, where maximal solutions are shown to exist for noncompact convex initial values in $\RR^n$. There one employs the phrase \textit{innermost solutions}, which is equivalent to \textit{maximal solutions} (a function being maximal means that its sub-level sets are innermost).

In Theorem \ref{thm-intro:max_sol}, we confirm in general the existence of maximal solutions of \eqref{eq-intro:IVP}, thus answering \eqref{eq-intro:Q_infty} in the way observed here.

\vspace{3pt}

(2) We turn to the case where $\Omega$ is smooth and bounded. A natural question is thus
\begin{equation}\label{eq-intro:Q_pOmega}
	\begin{aligned}
		&\hspace{14pt}\text{to find a boundary condition that gives rise to} \\
		&\hspace{142pt} \text{non-trivial unique solutions of \eqref{eq-intro:IVP}.}
	\end{aligned}\tag{$\text{Q}_2$}
\end{equation}
We observe that the classical Dirichlet and Neumann conditions both have limitations in this situation. The Neumann (or free boundary) condition
\begin{equation}\label{eq-intro:NEUMANN}
	\metric{\D u}{\nu_\Omega}=0\ \ \text{on}\ \ \p\Omega, \tag{Neumann}
\end{equation}
where $\nu_\Omega$ is the outer unit normal of $\Omega$, was studied by Marquardt \cite{Marquardt_2017} and Koerber \cite{Koerber_2020}, see also \cite{Lambert-Scheuer_2016, Marquardt_2013} for the smooth flow. Geometrically it means that each hypersurface in the flow stays perpendicular to $\p\Omega$. For weak solutions, the Neumann version of the energy \eqref{eq-intro:energy_E_interior} is
\begin{equation}\label{eq-intro:energy_Neumann}
	\hat J_u^K(E)=\Ps{E;\Omega\cap K}-\int_{E\cap\Omega\cap K}|\D u|,
\end{equation}
see \cite{Marquardt_2017}. However, note that \eqref{eq-intro:IVP}+\eqref{eq-intro:NEUMANN} only admits the trivial solution $u\equiv0$ in a bounded $\Omega$. This follows by \cite[Lemma 5.5]{Marquardt_2017}. Intuitively, since $\Ps{E;\Omega}$ only counts the interior perimeter, jumping over the whole $\Omega$ at $t=0$ is the energy-minimizing choice.

The classical Dirichlet condition $u|_{\p\Omega}=\varphi$ often does not admit solutions. Considering for example $\Omega=\{|x|<2\}\subset\RR^2$ and $E_0=\{|x|<1\}$, then all the weak solutions of \eqref{eq-intro:IVP} are bounded above by the smooth solution $\log|x|$, by Theorem \ref{thm-intro:existence}. Thus, the boundary value $\varphi$ cannot exceed $\log 2$. However, we shall notice later that the obstacle condition that we develop in this work is a certain relaxation of the Dirichlet condition.

\vspace{12pt}
\phantomsection
\addtocounter{subsection}{1}
\addcontentsline{toc}{subsection}{\arabic{section}.\arabic{subsection}\texorpdfstring{\quad}{} The outer obstacle condition; variational definitions; calibrations}
\textbf{The outer obstacle condition.}

In the present work, we consider the new boundary condition
\begin{equation}\label{eq-intro:OBS}
	\lim_{x\to\p\Omega}\frac{\D u(x)}{|\D u(x)|}=\nu_\Omega, \tag{OBS}
\end{equation}
in a sense to be made precise below. This states that each flowing hypersurface sticks tangentially to $\p\Omega$ after arrival (since $\D u/|\D u|$ is the outer unit normal of the flowing hypersurfaces). Both in the sense of energy comparison (see Definition \ref{def-intro:summary_defs} below), and in geometric sense, the boundary $\p\Omega$ plays the role of an \textit{outer obstacle}. Our main results include developing suitable weak formulations of IMCF\texttt{+}\eqref{eq-intro:OBS}, and establishing a first existence and regularity theorem in smooth bounded $\Omega$.

Expanding curvature flows with outer obstacle condition, to the author's knowledge, has not been extensively studied in the literature. The obstacle considered here is different from the one considered by Moser \cite{Moser_2008} (see ``\hyperlink{intro:relations}{further relations}'' below for more details).

We first present some examples:

\begin{example}[translating cycloids]\label{ex-intro:cycloid} {\ }
	
	It is observed by Drugan-Lee-Wheeler \cite{Drugan-Lee-Wheeler_2016} that cycloids are (the only) translating solitons of IMCF in $\RR^2$. With suitable positioning, let us view the cycloid $\gamma$ as sliding in a strip region $\Omega\subset\RR^2$, see Figure \ref{fig-intro:cycloid}. The curve $\gamma$ can be parametrized by angle using the formula
	\[\gamma(\th)=\frac14\big(1-\cos2\th,2\th-\sin2\th\big),\qquad\th\in(0,\pi).\]
	Note that $\gamma$ has two singularities at its endpoints, where it is asymptotic to a graph $y=\pm Cx^{3/2}$. This implies that the arrival time function $u$ generated by translating $\gamma$ satisfies \eqref{eq-intro:OBS}. By the asymptotic of $\gamma$, each sub-level set of $u$ is a globally $C^{1,1/2}$ curve (note that the sub-level set contains a translation of $\gamma$ and two horizontal radial lines).
\end{example}

\vspace{-12pt}

\begin{figure}[ht]
	\setlength{\abovecaptionskip}{-36pt}
	\setlength{\belowcaptionskip}{-6pt}
	\captionsetup{width=0.9\textwidth}
	\begin{center}
		\includegraphics[scale=0.55]{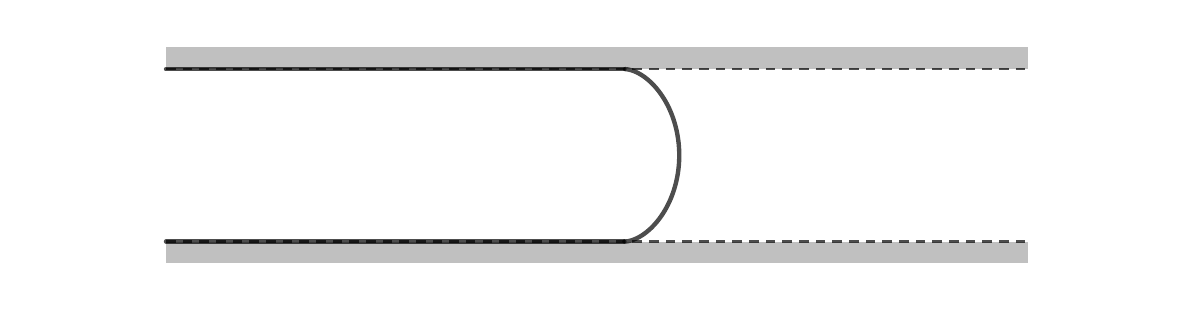}
	\end{center}
	\begin{picture}(0,0)
		\put(320,65){$\Omega$}
		\put(230,79){$\gamma_t$}
		\put(256,65){$\boldsymbol{\to}$}
		\put(253,52){$\boldsymbol{\to}$}
		\put(253,78){$\boldsymbol{\to}$}
	\end{picture}
	\caption{translating cycloids.}\label{fig-intro:cycloid}
\end{figure}

\begin{example}[expanding nephroid]\label{ex-intro:nephroid} {\ }
	
	Let $\gamma$ be a half of the nephroid curve, positioned in a half-plane as in Figure \ref{fig-intro:nephroids}. An explicit angle parametrization of $\gamma$ is given by
	\[\gamma(\th)=-\sin(\th/2)ie^{i\th}+\frac12\cos(\th/2)e^{i\th},\qquad\th\in(0,2\pi).\]
	Then the family of curves
	\[e^{3t/4}\gamma\qquad(t\in\RR)\]
	generate a smooth IMCF $u$ that satisfies \eqref{eq-intro:OBS} at the boundary $\{y=0\}$. Each sub-level $E_t$ is bounded in $\RR^2$ and has $C^{1,1/2}$ boundary. Moreover, note that $\lim_{x\to0}u(x)=-\infty$ and $\lim_{x\to\infty}u(x)=+\infty$. We refer to Appendix \ref{sec:ex} for general computations of homothetic solitons, along with discussions relating to blow-up\,/\,blow-down behaviors.
	\begin{figure}[h]
		\centering
		\includegraphics{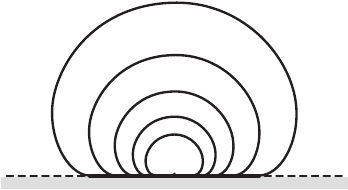}
		\begin{picture}(0,0)
			\put(-89,75){\rotatebox[origin=c]{90}{$\to$}}
			\put(-60,63){\rotatebox[origin=c]{50}{$\to$}}
			\put(-126,63){\rotatebox[origin=c]{130}{$\to$}}
		\end{picture}
		\caption{Expanding nephroids.}\label{fig-intro:nephroids}
	\end{figure}
	
	The nephroid soliton is an important example that shapes the sharpness of several main results. For instance, see the discussions around Theorem \ref{thm-intro:liouville} below.
\end{example}

\begin{example}
	Figure \ref{fig-intro:ellipse_obs} below qualitatively depicts a 2-dimensional flow, where:
	\begin{itemize}[nosep]
		\item the ambient domain $\Omega$ is the light gray region (it is a solid ellipse with a puncture),
		\item the obstacle $\p\Omega$ is the bold solid curve,
		\item the initial data $E_0$ is the dark gray region,
		\item the slim solid curves are level-sets in the flow,
		\item the shaded region is a jump that happens in the weak setting.
	\end{itemize}
	The puncture blocks the movement of the flowing curves, so the latter move forward in the left and right channels around the puncture. At a certain time $T$, the two branches of $\p E_T$ merge through a jump. See the left part of Figure \ref{fig-intro:ellipse_obs}: at the jump time we have $A+C=B+D$.
	\begin{figure}[h]
		\centering
		\includegraphics{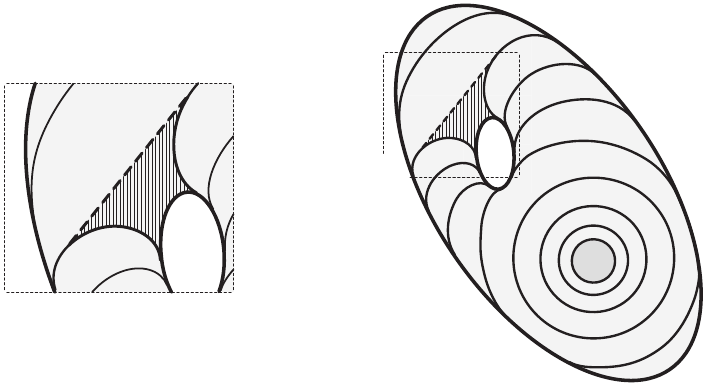}
		\begin{picture}(0,0)
			\put(-65,56){$E_0$}
			\put(-77,170){$\p E_T$}
			\put(-169,96){$\p E_T$}
			\put(-147,139){$\p E_T^+$}
			\put(-223,123){\rotatebox[origin=c]{180}{$\xrightarrow{\hspace{50pt}}$}}
			\put(-294,105){$A$}
			\put(-258,110){$B$}
			\put(-262,76){$C$}
			\put(-293,64){$D$}
		\end{picture}
		\caption{IMCF with obstacle in a punctured ellipse}\label{fig-intro:ellipse_obs}
	\end{figure}
\end{example}


\textbf{The variational definitions.}

The formal definition of weak IMCF with obstacle is given below, which is distilled from Definitions \ref{def-obs:energy_E}, \ref{def-obs:obstacle1}, \ref{def-obs:energy_u}, and Theorem \ref{thm-obs:obstacle2}. It is stated concisely here for ease of understanding; we refer the reader to Subsections \ref{subsec:obs_energy_E} and \ref{subsec:obs_diri} for the full version, and Remark \ref{rmk-obs:integrability} for the well-definedness of the energies. When $K\Subset\Omega$, note that the energies \eqref{eq-intro:energy_E} \eqref{eq-intro:energy_v} are the same as the interior ones \eqref{eq-intro:energy_v_interior} \eqref{eq-intro:energy_E_interior}.

\begin{summaryofdef}\label{def-intro:summary_defs} {\ }
	
	Let $\Omega$ be a locally Lipschitz domain in a smooth Riemannian $n$-manifold ($n\geq2$), and let $u\in\Lip_{\loc}(\Omega)$. We say that $u$ is a weak solution of IMCF in $\Omega$ that respects the outer obstacle $\p\Omega$, if either of the following equivalent conditions hold.
	
	\vspace{3pt}
	
	(i) Each sub-level set $E_t:=\{u<t\}$ locally minimizes the energy $E\mapsto\Ps{E}-\int_E|\D u|$. This means that: for any $E\subset\Omega$ and any domain $K$ satisfying $E\Delta E_t\Subset K\Subset M$, we have $\tJ_u^K(E_t)\leq\tJ_u^K(E)$, where
	\begin{equation}\label{eq-intro:energy_E}
		\tJ_u^K(E)=\Ps{E;K}-\int_{E\cap K}|\D u|.
	\end{equation}
	Here $\Ps{E;K}$ denotes the perimeter of $E$ in $K$.
	
	\vspace{3pt}
	
	(ii) $u$ locally minimizes the energy
	\begin{equation}\label{eq-intro:energy_v}
		v\mapsto\int_{\Omega}\big(|\D v|+v|\D u|\big)-\int_{\p\Omega}v
	\end{equation}
	among all functions $v\in\Lip_{\loc}(\Omega)$ such that $\{u\ne v\}\Subset M$, in a similar manner (see Definition \ref{def-obs:energy_u}, Theorem \ref{thm-obs:obstacle2}).
\end{summaryofdef}

In item (i), we view $E_t$ as a subset of $\Omega$, hence a subset of $M$ (we assume this convention for this paper). The difference set $E\Delta E_t$ is allowed to touch $\p\Omega$ (this can be compared with the interior case \eqref{eq-intro:energy_E}).

Additionally, note that the perimeter $\Ps{E;K}$ includes the boundary portion $\p E\cap\p\Omega\cap K$. This is also important for us: without including this portion, one obtains the free boundary formulation \eqref{eq-intro:energy_Neumann} discussed above.

\vspace{12pt}
\textbf{The role of calibration.}

The following calculation shows the compatibility between \eqref{eq-intro:energy_E} and the boundary condition \eqref{eq-intro:OBS}. Suppose $u\in C^\infty(\Omega)$ solves the IMCF in $\Omega$ and satisfies \eqref{eq-intro:OBS} at $\p\Omega$. Assume (an ideal scenario of) sufficiently strong regularity and convergence. Let $E_t$ be a sub-level set, and let $E\subset\Omega$ be a competitor with $E\Delta E_t\Subset M$. By divergence theorem, we have (where $\nu$ denotes the outer unit normal of the objects):
\begin{equation}\label{eq-intro:calib1}
	\begin{aligned}
		\Ps{E_t} &= \underbrace{\H^{n-1}(\p E_t\cap\Omega)}_{\text{where $\nu_{E_t}=\D u/|\D u|$}}+\underbrace{\H^{n-1}(\p E_t\cap\p\Omega)}_{\text{where $\nu_{E_t}=\nu_\Omega=\D u/|\D u|$}} \\
		&= \int_{\p E_t}\frac{\D u}{|\D u|}\cdot\nu_{E_t}
		= \int_{\p E}\frac{\D u}{|\D u|}\cdot\nu_E+\int_\Omega(\chi_{E_t}-\chi_E)\div\Big(\frac{\D u}{|\D u|}\Big) \\
		&\leq \Ps{E}+\int(\chi_{E_t}-\chi_E)|\D u|,
	\end{aligned}
\end{equation}
which is exactly $\tJ_u(E_t)\leq\tJ_u(E)$.

For weak solutions, the vector field $\D u/|\D u|$ is not everywhere defined, and the suitable replacement is an abstract calibration. We say that a weak solution $u$ of the IMCF is \textit{calibrated} by a vector field $\nu$, if
\begin{equation}\label{eq-intro:calibration}
	|\nu|\leq1,\ \ \metric{\D u}{\nu}=|\D u|\ \ \text{a.e.,\quad and}\quad \div(\nu)=|\D u|\ \ \text{weakly}.
\end{equation}
When $u$ is a smooth solution, $\D u/|\D u|$ is clearly a calibration of $u$. When $u$ is constant in a certain region, \eqref{eq-intro:calibration} reduces to $|\nu|\leq1$ and $\div(\nu)=0$, which corresponds to a usual calibration. Thus, we may view $\nu$ as providing extra information in the jumping regions. The formulation \eqref{eq-intro:calibration} first appeared in \cite[Section 3]{Huisken-Ilmanen_2001} briefly, and played major roles in several analytic studies of equations involving the 1-Laplacian, see \cite{Kawohl-Schuricht_2007, Latorre-Leon_2018, Mazon-Leon_2015} and references therein. We refer to Subsection \ref{subsec:calibrated} for a more details on \eqref{eq-intro:calibration}. The following criterion states that, in the weak setting, the condition \eqref{eq-intro:OBS} should be understood in terms of the calibration:

\begin{lemma}[identical with Lemma \ref{lemma-obs:bd_orthogonal}]\label{lemma-intro:asymp_tangent} {\ }
	
	Let $\Omega\subset M$ be a locally Lipschitz domain, and $u$ be a weak solution of IMCF in $\Omega$ which is calibrated by $\nu$. If $\nu\cdot\nu_\Omega=1$ on $\p\Omega$ in the trace sense, then the solution $u$ respects the outer obstacle $\p\Omega$.
\end{lemma}

\phantomsection
\addtocounter{subsection}{1}
\addcontentsline{toc}{subsection}{\arabic{section}.\arabic{subsection}\texorpdfstring{\quad}{} Main theorems: existence and regularity}
\textbf{Main theorems: existence and regularity.}

The following main theorem states the existence and regularity of solutions with smooth bounded obstacle, thus answering \eqref{eq-intro:Q_pOmega} raised above.

\begin{theorem}[concise version of Theorem \ref{thm-exist:main}]\label{thm-intro:existence} {\ }
	
	Let $\Omega\Subset M$ be a smooth domain, and let $E_0\Subset\Omega$ be a $C^{1,1}$ domain. Then there is a unique weak solution $u$ of the IMCF with initial value $E_0$ and outer obstacle $\p\Omega$, in the sense of Definition \ref{def-obs:ivp}. Regarding $u$ we have the following conclusions:
	
	(i) $u\in\Lip_{\loc}(\Omega)\cap\BV(\Omega)\cap C^{0,\alpha}(\bar\Omega)$,
	
	(ii) $u$ is calibrated by a vector field $\nu$ satisfying $\metric{\nu}{\nu_\Omega}\geq 1-Cd(x,\p\Omega)^\alpha$.
	
	(iii) For each $t>0$, the boundary of the sub-level set $E_t=\{u<t\}$ is a $C^{1,\alpha}$ hypersurface in some small neighborhood of $\p\Omega$.
	
	(iv) $u$ is maximal among all weak solutions of IMCF in $\Omega$ with initial value $E_0$ (the comparison solution need not respect the obstacle $\p\Omega$).
\end{theorem}

We remark on the statements of this theorem:
\begin{enumerate}[nosep]
	\item In item (ii), $\nu_\Omega$ is any smooth extension of the outer unit normal of $\Omega$. By Example \eqref{ex-intro:cycloid}, the regularity result of $u$ cannot exceed $C^{0,2/3}$, and the regularity of $E_t$ cannot exceed $C^{1,1/2}$. The H\"older exponent $\alpha$ in the theorem depends only on the dimension (but we do not know its optimal value).
	\item Item (iii) in particular implies that $E_t$ do not have singularity near $\p\Omega$. On the other hand, when $n\geq8$, minimal hypersurface singularities may occur in the interior.
	\item An intuitive understanding of item (iv) is that: the solution of \eqref{eq-intro:IVP}\texttt{+}\eqref{eq-intro:OBS} makes the minimal amount of ``jumps at $\p\Omega$'' relative to other solutions of \eqref{eq-intro:IVP}. Essentially, item (iv) is a consequence of Theorem \ref{thm-obs:auto_subsol}, which states that every interior weak solution of IMCF is a weak subsolution for the outer obstacle boundary condition.
	\item A direct corollary of (i) is that $u\in L^\infty(\Omega)$, i.e. the flow sweeps out $\Omega$ within finite time. This is closely related to the smoothness of $\p\Omega$. If $\Omega$ forms a cone at some point $x_0$, then a shrinking soliton could form near $x_0$, see Example \ref{ex-ex:hypocycloid}. In this case, the maximal solution would grow at the order of $O\big(|\log d(\cdot,x_0)|\big)$ near $x_0$.
\end{enumerate}

\vspace{3pt}

Applying Theorem \ref{thm-intro:existence}(iv) in an exhausting sequence of precompact domains, and then taking the limit, we can prove the following general existence theorem. This answers \eqref{eq-intro:Q_infty} raised above. 

\begin{theorem}[identical with Theorem \ref{thm-max:existence}]\label{thm-intro:max_sol} {\ }
	
	Let $M$ be a (possibly incomplete) manifold, and $E_0\subset M$ be a (possibly unbounded) $C^{1,1}$ domain. Then there exists, up to equivalence, a unique maximal weak solution of IMCF on $M$ with initial value $E_0$, in the sense of Definition \ref{def-prelim:ivp}.
\end{theorem}

Further properties of maximal solutions are summarized in Theorem \ref{thm-max:properties}. Building upon Theorem \ref{thm-intro:max_sol}, we can write down a neat proof of the main result of \cite{Xu_2023_proper}; see Theorem \ref{thm-exist:properness} for more details.


\vspace{12pt}
\phantomsection
\addtocounter{subsection}{1}
\addcontentsline{toc}{subsection}{\arabic{section}.\arabic{subsection}\texorpdfstring{\quad}{} Ideas of proof}
\textbf{Ideas in proving existence: blow-up and parabolic estimates}

We briefly describe the main ideas for proving Theorem \ref{thm-intro:existence}. Assume that $\Omega$ is a bounded smooth domain. Let us derive the core item (ii) of Theorem \ref{thm-intro:existence}. By Lemma \ref{lemma-intro:asymp_tangent}, this implies that the solution respects the obstacle $\p\Omega$.

We start with constructing a family of approximate solutions, where the obstacle $\p\Omega$ is replaced by a soft curvature well. Consider a family of smooth functions $\psi_\delta$, defined such that $\psi_\delta|_\Omega=0$ and $\psi_\delta(x)\to+\infty$ when $d(x,\Omega)\to\delta$. So $\psi_\delta$ is defined in the $\delta$-neighborhood of $\Omega$. It follows that
\[\lim_{\delta\to0}\psi_\delta(x)=\left\{\begin{aligned}
	& 0\qquad(x\in\Omega), \\
	& \infty\qquad(x\notin\Omega),
\end{aligned}\right.\]
and we recover a hard obstacle when $\delta\to0$. Then consider the weighted version of weak IMCF
\begin{equation}\label{eq-intro:soft_obs}
	\div\Big(e^{\psi_\delta}\frac{\D u_\delta}{|\D u_\delta|}\Big)=e^{\psi_\delta}|\D u_\delta|.
\end{equation}
The rapid growth of $\psi_\delta$ ensures the existence of a solution $u_\delta$ such that
\[\lim_{d(x,\Omega)\to\delta}u_\delta(x)=+\infty.\]
Then we take a subsequential limit $u=\lim_{\delta\to0}u_\delta$. Note that \eqref{eq-intro:soft_obs} is the usual IMCF inside $\Omega$, thus $u$ is an interior solution of \eqref{eq-intro:IVP} by a standard compactness theorem. The level sets of $u_\delta$ are bent drastically by $\psi_\delta$ when going outside $\Omega$, and when $\delta\to0$, a large amount of level sets pile up at $\p\Omega$.

We need to obtain a uniform control on the calibrations of $u_\delta$ (denoted by $\nu_\delta$), and pass it to the limit $\delta\to0$. For simplicity, assume for the moment that $u_\delta$ are smooth solutions of \eqref{eq-intro:soft_obs}, so  $\nu_\delta$ are the outer unit normals of the hypersurfaces. Several new ingredients are involved in the next step: a blow-up argument, a new Liouville theorem on the half-space, and finally, a new near-obstacle parabolic estimate.

The blow-up argument is aimed at showing the following:
\begin{enumerate}[label={(\roman*)}, nosep]
	\item the normal vector $\nu_\delta$ is uniformly converging to $\nu_\Omega$ \textit{on} $\p\Omega$,
	\item $\metric{\nu_\delta}{\p_r}\geq\frac12$ in a neighborhood of $\p\Omega$ that is independent of $\delta$.
\end{enumerate}
\noindent Here $\p_r=-\D d(\cdot,\p\Omega)$ is the outpointing radial vector orthogonal to $\p\Omega$. The two claims are proved through similar ideas, so let us explain (ii). Suppose, on the contrary, that as $\delta\to0$, there is a sequence of ``bad points'' $x_i\to\p\Omega$. Then we take suitably rescaled solutions near $x_i$, and obtain a ``bad'' blow-up limit. The limit turns out to be a weak solution of IMCF in the half space $\{x_n<0\}$, which respects the boundary obstacle. This eventually contradicts the following Liouville theorem, thus proving claim (ii):

\begin{theorem}[identical with Theorem \ref{thm-halfplane:liouville}]\label{thm-intro:liouville} {\ }
	
	Let $\Omega=\{x_n<0\}\subset\RR^n$, and $u$ is a weak solution of IMCF in $\Omega$ that respects the outer obstacle $\p\Omega$. Assume further that
	\begin{equation}
		|\D u(x)|\leq\frac{C}{|x_n|},\qquad\text{and}\qquad
		\inf_\Omega(u)>-\infty.
	\end{equation}
	Then $u$ is constant.
\end{theorem}

Here the condition $\inf(u)>-\infty$ is not removable, due to Example \ref{ex-intro:nephroid}.

\vspace{3pt}

Next, the parabolic estimate is used to derive the following: based on the results (i)(ii) above, the control on $\nu_\delta$ can be strengthened to
\begin{equation}\label{eq-intro:para1}
	\metric{\nu_\delta}{\p_r}\geq 1-Cd(x,\p\Omega)^\gamma-o_\delta(1),
\end{equation}
for some $\gamma=\gamma(n)\in(0,1)$. Roughly speaking, to derive \eqref{eq-intro:para1} we compute the evolution of the quantity $F=d(\cdot,\p\Omega)^{-\gamma}\big[1-\metric{\nu_\delta}{\p_r}\big]$, which yields
\begin{equation}\label{eq-intro:para2}
	\Big(\p_t-\frac{\Delta}{H^2}\Big)F\leq-\frac{c(n,\gamma)}{H^2}d(\cdot,\p\Omega)^{-\gamma-2}\big[1-\metric{\nu_\delta}{\p_r}\big]^2+\text{($\cdots$)}.
\end{equation}
The terms omitted here are caused by the curvature of $M$ and $\p\Omega$, and are of lower order than the negative main term. Then the estimate \eqref{eq-intro:para1} comes from the parabolic maximum principle. In view of the nephroid Example \ref{ex-intro:nephroid} (which does not satisfy \eqref{eq-intro:para1} for $t\ll0$), the uniform condition (ii) cannot be removed in this parabolic estimate.

For a more detailed demonstration of ideas (which involves more setup), see Subsection \ref{subsec:exist_strategy}. We remark here that the parabolic estimate is a necessary part in proving the existence theorem (not only an auxiliary technique to upgrade the regularity). See the proof of Lemma \ref{lemma-exist:graphical} for what it involves.

Many of the ideas above only work when $u_\delta$ are smooth solutions. Particularly, no analogue of parabolic estimate has been found for weak solutions. So we instead need to consider the elliptic regularized equations, originally introduced in \cite{Huisken-Ilmanen_2001}. For $\epsilon\ll1$, the $\epsilon$-regularized IMCF equation is
\begin{equation}\label{eq-intro:ellreg}
	\div\Big(\frac{\D u_\epsilon}{\sqrt{\epsilon^2+|\D u_\epsilon|^2}}\Big)=\sqrt{\epsilon^2+|\D u_\epsilon|^2},
\end{equation}
(with slight modification for the weighted case \eqref{eq-intro:soft_obs}), which is strictly elliptic and thus can be solved by standard PDE theories. Then as $\epsilon\to0$, the approximating solutions $u_\epsilon$ converge to a genuine weak IMCF. The equation \eqref{eq-intro:ellreg} has the geometric meaning that, if $u_\epsilon$ solves \eqref{eq-intro:ellreg} in $\Omega$, then the graph of $\epsilon^{-1}u_\epsilon$ forms a downward translating soliton of IMCF in the product $\Omega\times\RR$. So the parabolic estimates may be performed on this (smooth) soliton, then sent to the limit $\epsilon\to0$. We remark that this method was used by Heidusch to prove the $C^{1,1}$ regularity of level sets in proper solutions \cite{Heidusch_thesis}.

Finally, Theorem \ref{thm-intro:existence} is used to prove Theorem \ref{thm-intro:max_sol} in the following way. Consider an exhaustion of $E_0$ and $M$ by precompact domains $E_0^1\Subset E_0^2\Subset\cdots$ and $\Omega_1\Subset\Omega_2\Subset\cdots$. Applying Theorem \ref{thm-intro:existence} for each $E_0^i$ in $\Omega_i$, we obtain a localized solution $u_i$. Then the (descending) limit $u=\lim_{i\to\infty}u_i$ is the desired maximal solution. Indeed, for any other solution $v$ on $M$, we may compare in each $\Omega_i$ to conclude that $u_i\geq v$, then let $i\to\infty$ to derive $u\geq v$.

\vspace{12pt}
\phantomsection
\addtocounter{subsection}{1}
\addcontentsline{toc}{subsection}{\arabic{section}.\arabic{subsection}\texorpdfstring{\quad}{} Nonlinear potential theory; Dirichlet condition; further relations}
\textbf{Relation with nonlinear potential theory.}

A main analytic branch of the weak IMCF is the $p$-harmonic approximation scheme, initiated by Moser \cite{Moser_2007} and pursued by later works \cite{Agostiniani-Mantegazza-Mazzieri-Oronzio_2022, Fogagnolo-Mazzieri_2022, Kotschwar-Ni_2009, Mari-Rigoli-Setti_2022}. The core relation is the following. Suppose $v_p>0$ is a $p$-harmonic function, meaning that
\[\Delta_pv_p:=\div\big(|\D v_p|^{p-2}\D v_p\big)=0.\]
Moreover, suppose that for a sequence $p\to1$, the functions $u_p:=(1-p)\log v_p$ converges to $u$ in $C^0_{\loc}$. Then it can be shown that $u$ is a weak solution of the IMCF. Usually, one considers the \textit{$p$-capacitary potential} associated with an initial set $E_0$. Namely, one sets $v_p\in W^{1,p}_0(M\setminus E_0)$ as the $p$-harmonic function with $v_p|_{\p E_0}=1$. Under various conditions (for example, when $E_0$ is a $C^{1,1}$ domain in $\RR^n$ \cite{Moser_2007}), the resulting function $u$ is a weak solution of \eqref{eq-intro:IVP}.

When $\Omega$ is a smooth bounded domain, the $p$-capacitary potential carries a Dirichlet condition $v_p|_{\p\Omega}=0$. Note that each $v_p$ is the minimal positive $p$-harmonic function in $\Omega$ with $v_p|_{E_0}=1$. Therefore, Moser's transformation $u_p=(1-p)\log v_p$ yields a maximal solution of the equation
\[\left\{\begin{aligned}
	& \Delta_p u_p=|\D u_p|^p\qquad\text{in }\Omega\setminus E_0 \\
	& u_p=0\qquad\text{on }\p E_0
\end{aligned}\right.\]
in $\Omega$. As $p\to1$, one expects that the limit is also a maximal weak solution of the initial value problem for IMCF. In view of Theorem \ref{thm-intro:max_sol}, we thus pose the following question:

\begin{question}\label{qs-intro:p-harmonic}
	Let $\Omega$ be a bounded smooth domain, and $E_0\Subset\Omega$ be a $C^{1,1}$ domain. Let $v_p$ solve the boundary value problem
	\begin{equation}
		\left\{\begin{aligned}
			& \Delta_p v_p=0\qquad\text{in }\Omega\setminus\bar{E_0}, \\
			& v_p=1\qquad\text{on }\p E_0, \\
			& v_p=0\qquad\text{on }\p\Omega.
		\end{aligned}\right.
	\end{equation}
	Set $u_p=(1-p)\log v_p$. Does (a subsequence of) $u_p$ converge in $C^0_{\loc}(\Omega\setminus E_0)$ to the solution stated in Theorem \ref{thm-intro:existence}, as $p\to1$?
\end{question}

We hope to address this question in forthcoming works.

\vspace{12pt}
\hypertarget{intro:dirichlet}{\textbf{Comparison with Dirichlet condition.}}

The IMCF equation involves the 1-Laplacian $\Delta_1(u)=\div\big(\D u/|\D u|\big)$, which is the Euler-Lagrange operator of the energy $u\mapsto\int|\D u|$. The Dirichlet problem for 1-Laplacian type equations is known to often have outer obstacle interpretations. An instance is the following: for domains $E_0\Subset\Omega$, the \textit{1-capacitor} with Dirichlet condition is a solution of the following minimization problem:
\begin{align}
	\operatorname{Cap}_1(\Omega;E_0) &= \inf\Big\{\int_\Omega|\D u|: u\in\Lip(\bar\Omega),\ u|_{E_0}\geq1,\ u|_{\p\Omega}=0\Big\} \label{eq-intro:cap1_1}\\
	&= \inf\Big\{\|Du\|(\Omega)+\int_{\p\Omega}|u|: u\in\BV(\Omega),\ u|_{E_0}\geq1\Big\}. \label{eq-intro:cap1_2}
\end{align}
The former form \eqref{eq-intro:cap1_1} may not admit a minimizer, while the latter form \eqref{eq-intro:cap1_2} always does. A minimizer of \eqref{eq-intro:cap1_2} need not be zero on $\p\Omega$, but a penalty term $\int_{\p\Omega}u$ is introduced to control the boundary behavior. If $u$ is a minimizer of \eqref{eq-intro:cap1_2}, then every super-level set of $u$ solves the double obstacle problem
\[\inf\Big\{\Ps{E}: E_0\subset E\subset\Omega\Big\}.\]
Similar facts hold for the first Dirichlet eigenvalue of $\Delta_1$, see \cite{Kawohl-Friedman_2003, Kawohl-Schuricht_2007} and references therein.

\vspace{3pt}

Shifting our attention back to IMCF: \eqref{eq-intro:energy_v} may be viewed as a relaxation of the interior energy $\int|\D v|+v|\D u|$. The boundary term $-\int_{\p\Omega}v$ in \eqref{eq-intro:energy_v} is thus a penalty for $v$ ``not attaining $+\infty$ at $\p\Omega$''. In view of these thoughts, we may view the IMCF with outer obstacle as a relaxation of the Dirichlet problem
\[\left\{\begin{aligned}
	& \div\Big(\frac{\D u}{|\D u|}\Big)=|\D u|, \\
	& u|_{\p E_0}=0, \\
	& u|_{\p\Omega}=+\infty,
\end{aligned}\right.\]
This is consistent with the $p$-harmonic and soft obstacle approximations introduced above. In both schemes, note that the approximating solutions attain $+\infty$ at the boundary, but this property is lost when one takes the limit.

\vspace{12pt}
\hypertarget{intro:relations}{\textbf{Further relations with existing works.}}

In \cite{Moser_2008}, Moser studied a different obstacle problem for the weak IMCF. Given a barrier function $\phi$, one considers the problem
\begin{equation}\label{eq-intro:Moser_obstacle}
	\left\{\begin{aligned}
		& u\leq\Phi \quad\text{and}\quad \div\Big(\frac{\D u}{|\D u|}\Big)\geq|\D u|\quad \text{in}\ \ \Omega,\\
		& \div\Big(\frac{\D u}{|\D u|}\Big)=|\D u|\quad \text{in }\{u<\Phi\}.
	\end{aligned}\right.
\end{equation}
This problem is close to an \textit{inner obstacle} problem, corresponding to forced jumps indicated by sub-level sets of the function $\Phi$. Note that $u$ is a subsolution of IMCF, and it can be a strict subsolution in $\{u<\Phi\}$. A particular case of \eqref{eq-intro:Moser_obstacle} is
\[\phi(x)=\left\{\begin{aligned}
	& 0\qquad(x\in\bar{E_0}), \\
	& \!\!+\!\infty\qquad(x\notin\bar{E_0}),
\end{aligned}\right.\]
see \cite[p.2238]{Moser_2008}. In this case, \eqref{eq-intro:Moser_obstacle} becomes equivalent to \eqref{eq-intro:IVP} \cite{Moser_2008}.

\vspace{3pt}

The obstacle problem for the mean curvature flow has been studied in some depth. See, for example, \cite{Almeida-Chambolle-Novaga_2012, Logaritsch_thesis, Mercier-Novaga_2015} for the smooth flow, and \cite{Giga-Tran-Zhang_2019, Mercier_2014, Rupflin-Schnurer_2020} for the level set flow. The flow is based on viscosity characterizations in these works, and one often obtains $C^{1,1}$ regularity for the solution, see e.g. \cite{Mercier-Novaga_2015, Rupflin-Schnurer_2020}.

\vspace{12pt}
\textbf{Organization of the paper.}

Section \ref{sec:prelim} contains some preliminary materials and auxiliary results. This includes a review of the smooth and weak IMCF, and the definition of calibrated and weighted solutions, an introduction to elliptic regularization, and a few technical lemmas. In Section \ref{sec:obs}, we build the definitions and basic properties for the IMCF with outer obstacle, surrounding what is summarized in Definition \ref{def-intro:summary_defs}. In Section \ref{sec:halfplane} and \ref{sec:para} respectively, we prove Liouville theorems on the half space, and derive the main parabolic estimate. These results aid the proof of the main existence theorem. Finally, in Section \ref{sec:exist} we prove the main existence Theorem \ref{thm-intro:existence}, and in Section \ref{sec:max} we prove Theorem \ref{thm-intro:max_sol} and discuss on maximal solutions. Further outline of contents are included at the beginning of each Section \ref{sec:obs}\,--\,\ref{sec:exist}. In Section \ref{sec:max} we discuss on maximal solutions and prove Theorem \ref{thm-intro:max_sol}. In Appendix \ref{sec:ex} we collect several planar solitons and discuss the asymptotic limits of solutions.

\vspace{12pt}
\textbf{Acknowledgements}.

The author would like to thank Luca Benatti, Mattia Fogagnolo and Luciano Mari for inspiring discussions during the conference ``Recent advances in comparison geometry'' in Hangzhou, China. I'm also grateful to Hubert Bray, Gerhard Huisken, Demetre Kazaras and Chao Li for their interest in the topic of this paper. \TODO{update the acknowledgement before posting.}

\section{Preliminaries}\label{sec:prelim}

\subsection{Notations}\label{subsec:notations}

We use $M$ to denote a smooth manifold, which is always assumed to be connected, oriented, without boundary, and carries a smooth Riemannian metric $g$. The metric need not be complete unless particularly specified. Let $n=\dim(M)$, and assume $n\geq2$ (note: we do not assume $n\leq7$). We use $\Omega$ to denote a domain in $M$, and always assume that it is connected. We write $A\Subset B$ if the closure of $A$ is compact in $B$.

We say that a domain $\Omega$ is \textit{locally Lipschitz}, if for every $x\in\p\Omega$ there is a cylindrical coordinate chart centered at $x$, in which $\Omega$ is the sub-graph of a Lipschitz function. By Rademacher's theorem, the outer unit normal of $\Omega$ (denoted by $\nu_\Omega$ in this paper) exists for $\H^{n-1}$-almost every point on $\p\Omega$.


For a (two-sided) hypersurface $\Sigma$, we use $\nu_\Sigma$ to denote the outward unit normal, and $A, H$ to denote the second fundamental form and mean curvature. Our sign convention is that the mean curvature of $S^{n-1}\subset\RR^n$ is $n-1$.

For a function $u$, we fix the notations
\begin{equation}
	E_t(u):=\{u<t\},\qquad E_t^+(u):=\{u\leq t\}.
\end{equation}
When there is no ambiguity, we write $E_t, E_t^+$ for brevity.

Let $u$ be a function on $\Omega$. We call $u$ \textit{proper} in $\Omega$ if $E_t(u)\Subset\Omega$ for all $t\in\RR$.

\begin{remark}\label{rmk-prelim:Et_in_Omega}
	When $u$ is defined on a domain $\Omega\subset M$, the sets $E_t,E_t^+$ are viewed as subsets of $\Omega$, hence subsets of $M$.
\end{remark}

\begin{remark}\label{rmk-prelim:level_set_convergence}
	The following fact is useful: if $u_i,u$ are continuous functions such that $u_i\to u$ in $C^0_{\loc}$, and $E_t(u_i)$ converges to a set $E$ in $L^1_{\loc}$, then $E_t(u)\subset E\subset E_t^+(u)$ up to measure zero.
\end{remark}

For a function $u$ defined in a domain $\Omega$, we say that

\begin{enumerate}[label=$\boldsymbol{\cdot}$, topsep=0pt, itemsep=-0.5ex]
	\item $u\in\Lip_{\loc}(\Omega)$, if $u\in\Lip(K)$ for all $K\Subset\Omega$;
	\item $u\in\Lip_{\loc}(\bar\Omega)$, if $u\in\Lip(K\cap\Omega)$ for all $K\Subset M$. Namely, the regularity holds up to $\p\Omega$ but remains local in $M$.
	\item $u\in\BV_{\loc}(\bar\Omega)$, if $u\in\BV(K\cap\Omega)$ for all $K\Subset M$. (Note: the space $\BV(\Omega)$ is still in the traditional sense: $u\in\BV(\Omega)$ means $u\in L^1(\Omega)$ and $\|Du\|(\Omega)<\infty$.)
	\item $u\in\Lip_0(\Omega)$, if $u\in\Lip(\Omega)$ and $\spt(u)\Subset\Omega$.
\end{enumerate}

\noindent Given a set $E$ with locally finite perimeter, and a domain $\Omega$, we denote

\begin{enumerate}[label=$\cdot$, topsep=0pt, itemsep=-0.5ex]
	\item $\Ps{E;\Omega}$ and $\Ps{E}$: the perimeter of $E$ inside $\Omega$, and $\Ps{E}:=\Ps{E;M}$;
	\item $E^{(1)}$: the measure-theoretic interior of $E$;
	\item $\p^*E$ and $\nu_E$: the reduced boundary and outer unit normal of $E$;
	\item $|D\chi_E|$: the perimeter measure of $E$.
\end{enumerate}
We refer to Appendix \ref{sec:gmt} for more explanations on these notations, as well as geometric measure theory background.

The main symbols $\IMCF{\cdot}$, $\IVP{\cdot}$, $\IMCFOO{\cdot}{\cdot}$, $\IVPOO{\cdot}{\cdot}$, corresponding to various forms of weak IMCF, are introduced in Definition \ref{def-prelim:weak_sol}, \ref{def-prelim:ivp}, \ref{def-obs:obstacle1}, \ref{def-obs:ivp} respectively. The regular radius $\sigma(x;\Omega,g)$ arises in Definition \ref{def-prelim:sigma_x}.

\subsection{Smooth and interior weak solutions}

We say that a family of hypersurfaces $\{\Sigma_t\}$ is a \textit{smooth solution} of IMCF in $\Omega$, if
\begin{equation}\label{eq-prelim:smooth_imcf}
	\frac{\p\Sigma_t}{\p t}=\frac{\nu}{H}\qquad\text{in}\ \ \Omega,
\end{equation}
where $\nu,H$ are the outer unit normal and mean curvature of $\Sigma_t$. We say that a function $u\in C^\infty(\Omega)$ is a \textit{smooth solution} of IMCF in $\Omega$, if $u$ has nonvanishing gradient, and
\begin{equation}\label{eq-prelim:level_set}
	\div\Big(\frac{\D u}{|\D u|}\Big)=|\D u|\qquad\text{in}\ \ \Omega.
\end{equation}
The two equations are related in the following way: if $u$ is a parametrizing function for the hypersurfaces $\{\Sigma_t\}$, namely, if $\Sigma_t=\{u=t\}$ for each $t$, then we may calculate $H=\div\big(\frac{\D u}{|\D u|}\big)$ and $\metric{\p_t\Sigma_t}{\nu}=1/|\D u|$, and thus \eqref{eq-prelim:smooth_imcf} is equivalent to \eqref{eq-prelim:level_set}.

The following facts are not hard to see: (1) the family of circles $\Sigma_t=\big\{|x|=e^{t/(n-1)}\big\}$, or the function $u(x)=(n-1)\log|x|$, is a smooth solution of the IMCF in $\RR^n$. (2) If $\{\Sigma_t\}_{a\leq t\leq b}$ is a compact solution of IMCF, then we can compute
\[\frac{d|\Sigma_t|}{dt}=|\Sigma_t|,\]
hence $|\Sigma_t|=e^{t-s}|\Sigma_s|$ for all $a\leq s,t\leq b$. (3) If $u$ solves \eqref{eq-prelim:level_set} in $(\Omega,g)$, then $u$ solves \eqref{eq-prelim:level_set} in $(\Omega,\lambda g)$ for all $\lambda>0$. This observation tells the appropriate way to blow up or blow down a solution; see \cite{Choi-Hung_2023} and \cite{Choi-Daskalopoulos_2021} for works using this observation.

If $\Sigma_0\subset\RR^n$ is star-shaped, compact, and with $H>0$, then \cite{Gerhardt_1991, Urbas_1990} proved that the smooth flow exists for $t\in[0,\infty)$ and converges to round spheres as $t\to\infty$. If $\Sigma_0$ is noncompact and convex, then \cite{Choi-Daskalopoulos_2021} proved that \eqref{eq-prelim:smooth_imcf} exists until $\Sigma$ becomes a hyperplane (then $H\equiv0$ and the flow stops). The maximal time of existence depends on the asymptotic cone of $\Sigma$. In particular, if $\Sigma_0$ is a complete graph with superlinear growth, then we have long time existence (see also \cite{Daskalopoulos-Huisken_2022}).

\vspace{3pt}

We next review the definition of weak solutions, following \cite{Huisken-Ilmanen_2001}. We also refer the reader to \cite{Lee_rel} as a helpful introductory material.

\begin{defn}\label{def-prelim:energies}
	Let $K\Subset\Omega\subset M$ be domains. For $u,v\in\Lip_{\loc}(\Omega)$ we define the energy
	\begin{equation}\label{eq-prelim:energy_v}
		J_u^K(v):=\int_K\big(|\D v|+v|\D u|\big).
	\end{equation}
	For a set $E$ with locally finite perimeter in $\Omega$, we define the energy
	\begin{equation}\label{eq-prelim:energy_E}
		J_u^K(E):=\Ps{E;K}-\int_{E\cap K}|\D u|.
	\end{equation}
	We say that a set $E$ locally minimizes $J_u$ (resp. minimizes from outside, inside) in $\Omega$, if for all $F$ (resp. for all $F\supset E$, $F\subset E$) and domain $K$ satisfying $E\Delta F\Subset K\Subset\Omega$, we have $J_u^K(E)\leq J_u^K(F)$.
\end{defn}

\begin{defn}[interior weak solutions; cf. {\cite[p. 364-367]{Huisken-Ilmanen_2001}}]\label{def-prelim:weak_sol} {\ }
	
	We say that $u\in\Lip_{\loc}(\Omega)$ is a (sub-, super-) solution of $\IMCF{\Omega}$, if one of the following equivalent conditions holds.
	
	(1) For any $v\in\Lip_{\loc}(\Omega)$ (resp. for any $v\leq u$, $v\geq u$) and any domain $K$ satisfying $\{u\ne v\}\Subset K\Subset\Omega$, we have $J_u^K(u)\leq J_u^K(v)$.
	
	(2) For each $t\in\RR$, the sub-level set $E_t:=\{u<t\}$ locally minimizes $J_u$ (resp. locally minimizes from outside, inside) in $\Omega$, as in Definition \ref{def-prelim:energies}.
\end{defn}

Solutions satisfying Definition \ref{def-prelim:weak_sol} are often called interior solutions (as opposed to solutions with obstacles). When there is a need to clarify the background Riemannian metric, we will write $\IMCF{\Omega,g}$ to denote weak solutions of IMCF in the domain $\Omega$ with the metric $g$.

\begin{defn}[initial value problem; cf. {\cite[p. 367-368]{Huisken-Ilmanen_2001}}]\label{def-prelim:ivp} {\ }
	
	Given a domain $\Omega\subset M$ and a $C^{1,1}$ domain $E_0\subset\Omega$ such that $\p E_0\cap\p\Omega=\emptyset$. We say that $u\in\Lip_{\loc}(\Omega)$ is a (sub-, super-) solution of $\IVP{\Omega;E_0}$, if one of the following equivalent conditions holds.
	
	(1) $E_0=\{u<0\}$, and $u|_{\Omega\setminus\bar{E_0}}$ is a (sub-, super-) solution of $\IMCF{\Omega\!\setminus\!\bar{E_0}}$.
	
	(2) $E_0=\{u<0\}$, and for any $t>0$, any set $E\supset E_0$ (resp. any $E\supset E_t$, $E_0\subset E\subset E_t$) and domain $K$ with $E\Delta E_t\Subset K\Subset\Omega$, we have $J_u^K(E_t)\leq J_u^K(E)$.
	
	We say that two solutions $u_1,u_2$ are equivalent, if $u_1=u_2$ in $\Omega\setminus E_0$.
\end{defn}

Similarly, we will write $\IVP{\Omega,g;E_0}$ when we need to clarify the background metric.

Item (2) here is formulated slightly differently but equivalently as in \cite{Huisken-Ilmanen_2001}. Also, from this item we see that the initial value problem of IMCF has the form of an inner obstacle problem, which reflects an observation in \cite{Moser_2008}.

\vspace{3pt}

Suppose $u\in\Lip_{\loc}(\Omega)$ solves $\IMCF{\Omega}$. Then for each $t$ and any competitor set $E$ with $E\Delta E_t\Subset K\Subset\Omega$, we have by the minimization of \eqref{eq-prelim:energy_E}
\begin{equation}\label{eq-prelim:Lam_r0_min}
	\Ps{E_t;K}\leq\Ps{E;K}+\sup_K|\D u|\cdot|E\Delta E_t|.
\end{equation}
Therefore, each sub-level set $E_t$ is a $(\Lambda,r_0)$-perimeter minimizer in $K$, for all $K\Subset\Omega$, $r_0>0$: see Appendix \ref{sec:gmt} for the definition of $(\Lambda,r_0)$-perimeter minimizers. The classical results in geometric measure theory then imply that each $E_t$ is a $C^{1,\alpha}$ hypersurface ($\alpha<1/2$) except for a codimension 8 singularity. In dimension less than 8, the result of Heidusch \cite{Heidusch_thesis} strengthens this regularity to $C^{1,1}$, if the solution comes from the elliptic regularization process (see Subsection \ref{subsec:ellreg}).

\vspace{3pt}

We say that a set $E$ is locally outward minimizing in $\Omega$ (resp. strictly outward minimizing), if for all $F\supset E$ and domain $K$ satisfying $F\setminus E\Subset K\Subset \Omega$, $|F\setminus E|>0$, we have $\Ps{E;K}\leq\Ps{F;K}$ (resp. $\Ps{E;K}<\Ps{F;K}$).

\begin{lemma}[cf. {\cite[Property 1.4, Lemma 1.6]{Huisken-Ilmanen_2001}}]\label{lemma-prelim:out_min}
	Suppose $u$ solves $\IMCF{\Omega}$. Then:
	
	(i) each $E_t$ is locally outward minimizing in $\Omega$, provided that it is nonempty.
	
	(ii) each $E_t^+$ is strictly outward minimizing in $\Omega$, provided it is nonempty.
	
	Suppose $u$ solves $\IVP{\Omega;E_0}$ with $E_0\Subset\Omega$. Then for all $t\geq0$ we have:
	
	(iii) $\Ps{E_t}\leq e^t\Ps{E_0}$ if $E_t\Subset\Omega$.
	
	(iv) $\Ps{E_t}=e^t\Ps{E_0}$ if $E_t\Subset\Omega$ and $E_0$ is locally outward minimizing in $\Omega$.
\end{lemma}

Another conclusion is that $E_t^+$ is the ``strictly outward minimizing hull'' of $E_t$, whenever $E_t^+\setminus E_t\Subset\Omega$. A more general statement is included in Subsection \ref{subsec:obs_min}, thus we do not make repetition here.


The following lemma proves Example \ref{ex-intro:radial_sym}; the reader may verify that the functions in \eqref{eq-prelim:sphere_sym} are indeed solutions of $\IVP{\Omega;E_0}$. The lemma no longer holds if we drop the radial symmetry of $u$.

\begin{lemma}\label{lemma-prelim:sphere_sym}
	Let $0<r_0<R_0\leq\infty$, and $f\in C^\infty(0,R_0)$ be positive and non-decreasing. Set $\Omega=S^{n-1}\times(0,R_0)$ with the metric $g=dr^2+f(r)^2g_{S^{n-1}}$, and $E_0=\{r<r_0\}$. Then any radial solution of $\IVP{\Omega;E_0}$ takes the form
	\begin{equation}\label{eq-prelim:sphere_sym}
		u(r)=\min\Big\{(n-1)\log\big[f(r)/f(r_0)\big],t_0\Big\}
	\end{equation}
	for some $t_0\geq0$.
\end{lemma}
\begin{proof}
	First, note that each $E_t$ takes the form of $\big\{r<r(t)\big\}$ for some $r(t)\in[r_0,R_0]$. Indeed, if there are $a<b<c<d\leq R_0$ such that
	\[\{a<r<b\}\subset E_t,\quad \{c<r<d\}\subset E_t,\quad \{b\leq r\leq c\}\nsubseteq E_t,\]
	then we would have $J_u^K\big(E_t\cup\{b\leq r\leq c\}\big)<J_u^K(E_t)$, where $K=\{a<r<d\}$. This contradicts the energy minimizing of $u$. Now let
	\[t_0=\sup\big\{t\geq0: E_t\Subset\Omega\big\}.\]
	Since $f$ is non-decreasing, $E_0$ is locally outward minimizing in $\Omega$. By Lemma \ref{lemma-prelim:out_min}(iv), we have $\Ps{E_t}=e^t\Ps{E_0}$ for all $t<t_0$, hence $f(r(t))=e^{t/(n-1)}f(r_0)$. Next, we claim that $f(r')<e^{t/(n-1)}f(r_0)$ for all $r'<r(t)$. To see this, by the definition of $r(t)$, we note that there exists $t'<t$ with $\{r<r'\}\subset E_{t'}$. Hence $|S^{n-1}|f(r')^{n-1}\leq\Ps{E_{t'}}=e^{t'}\Ps{E_0}=|S^{n-1}|e^{t'}f(r_0)^{n-1}$, proving our claim. As a result of these conclusions, we have
	\[\big\{u(r)<t\big\}=E_t=\big\{r<r(t)\big\}=\big\{f(r)<f(r(t))\big\}=\big\{f(r)<e^{t/(n-1)}f(r_0)\big\}\]
	for all $t<t_0$. For $t>t_0$, the set $E_t$ must be equal to $\Omega$, since it is noncompact and takes the form $\big\{r<r(t)\big\}$. Based on these facts, it is elementary to verify \eqref{eq-prelim:sphere_sym}.
\end{proof}

\begin{theorem}[interior maximum principle, {\cite[Theorem 2.2]{Huisken-Ilmanen_2001}}]\label{thm-prelim:max_principle} {\ }
	
	Let $u,v\in\Lip_{\loc}(\Omega)$ be respectively a supersolution and subsolution of $\IMCF{\Omega}$, such that $\{u<v\}\Subset\Omega$. Then $u\geq v$.
\end{theorem}

\begin{theorem}[interior compactness]\label{thm-prelim:compactness}
	Let $\Omega\subset M$ be a domain, and $g$ be a fixed smooth Riemannian metric on $\Omega$. Given the following data:
	
	(1) $\Omega_i$ is a sequence of domains that locally uniformly converge to $\Omega$,
	
	(2) $g_i$ are smooth metrics on $\Omega_i$, which locally uniformly converge to $g$,
	
	(3) $u_i\in\Lip_{\loc}(\Omega_i)$ solves $\IMCF{\Omega_i,g_i}$, and $u_i\to u$ in $C^0_{\loc}$ for some $u$,
	
	(4) for each $K\Subset\Omega$ we have $\sup_K|\D u_i|_{g_i}\leq C(K)$ for all sufficiently large $i$.
	
	\noindent Then $u$ solves $\IMCF{\Omega,g}$. Moreover, if a sequence of sets $E_i$ locally minimizes $J_{u_i}$ in $(\Omega_i,g_i)$, and $E_i$ converge to a set $E$ in $L^1_{\loc}$, then $E$ locally minimizes $J_u$ in $(\Omega,g)$.
\end{theorem}
\begin{proof}
	To prove $u$ solves $\IMCF{\Omega;g}$, one argues as in \cite[Theorem 2.1]{Huisken-Ilmanen_2001}. The minimizing property of $E$ follows from a standard set replacing argument. See for example the proof of \cite[Theorem 21.14]{Maggi}. In this argument we need the fact that $\int_A|\D_g u|_g\,dV_g=\lim_{i\to\infty}\int_A|\D_{g_i}u_i|_{g_i}\,dV_{g_i}$ for all measurable $A\Subset\Omega$. This is proved as follows. For any $\varphi\in\Lip_0(\Omega)$ with $\varphi\geq0$, by (2)(3)(4) and lower semi-continuity we have
	\begin{equation}\label{eq-prelim:lower_semi_cont}
		\int_\Omega\varphi|\D_g u|_g\,dV_g\leq\liminf_{i\to\infty}\int_\Omega\varphi|\D_{g_i}u_i|_{g_i}\,dV_{g_i}.
	\end{equation}
	On the other hand, we have the energy comparison $J_{u_i}^K(u_i)\leq J_{u_i}^K\big(\varphi u+(1-\varphi)u_i\big)$, where $K$ is chosen with $\spt(\varphi)\Subset K\Subset\Omega\cap\Omega_i$ for large $i$. Expanding this inequality we have
	\[\int_\Omega\varphi|\D_{g_i}u_i|_{g_i}\,dV_{g_i}\leq\int_\Omega\varphi|\D_{g_i}u|_{g_i}\,dV_{g_i}+\int_\Omega|u-u_i|\,\big(\varphi|\D_{g_i}u_i|_{g_i}+|\D_{g_i}\varphi|_{g_i}\big)\,dV_{g_i}.\]
	Taking $i\to\infty$, by items (2)(3)(4) and \eqref{eq-prelim:lower_semi_cont}, we obtain exact continuity
	\begin{equation}\label{eq-prelim:limit_du}
		\int_\Omega\varphi|\D_g u|_g\,dV_g=\lim_{i\to\infty}\int_\Omega\varphi|\D_{g_i}u_i|_{g_i}\,dV_{g_i}.
	\end{equation}
	Combined with item (4), this implies our claim.
\end{proof}

\subsection{Weak solutions with weights}\label{subsec:weight}

Given a function $\psi\in C^\infty(\Omega)$, consider the following weighted version of IMCF:
\begin{equation}\label{eq-prelim:weighted_imcf}
	\div\Big(e^\psi\frac{\D u}{|\D u|}\Big)=e^\psi|\D u|.
\end{equation}
In the smooth regime, this corresponds to the curvature flow
\[\frac{\p\Sigma_t}{\p t}=\frac{\nu}{H+\p\psi/\p\nu}.\]

We call $u$ a subsolution (resp. supersolution) of \eqref{eq-prelim:weighted_imcf}, if the equality sign is replaced by ``\,$\geq$\,'' (resp. ``\,$\leq$\,'').

\begin{lemma}\label{lemma-prelim:def_weighted}
	Regarding a function $u\in\Lip_{\loc}(\Omega)$, the following are equivalent:
	
	(1) for all $v\in\Lip_{\loc}(\Omega)$ and every domain $K$ with $\{u\ne v\}\Subset K\Subset\Omega$, we have
	\begin{equation}\label{eq-weight:energy_min}
		\int_K e^\psi\big(|\D u|+u|\D u|\big)\leq\int_K e^\psi\big(|\D v|+v|\D u|\big).
	\end{equation}
	
	(2) Setting $\tilde\Omega=\Omega\times S^1$ with the warped product metric $\tilde g=g+e^{2\psi(x)}dz^2$, the function $\tilde u(x,z)=u(x)$ is a solution of $\IMCF{\tilde\Omega,\tilde g}$.
	
	(3) $u$ is a solution of $\IMCF{\Omega;g'}$, where $g'=e^{\frac{2\psi}{n-1}}g$.
	
	\vspace{1pt}
	
	We call $u$ a weak solution of \eqref{eq-prelim:weighted_imcf}, if any of the above conditions holds.
\end{lemma}

The characterization (3) allows the standard theory of IMCF to be extended to the weighted case. In particular, if $u$ weakly solves \eqref{eq-prelim:weighted_imcf}, then each $E_t$ locally minimizes
\begin{equation}\label{eq-weight:energy_E}
	E\mapsto\int_{\p^*E}e^\psi\,d\H^{n-1}-\int_Ee^\psi|\D u|
\end{equation}
in the same sense with Definition \ref{def-prelim:energies}. The notion of weak sub- and super-solutions can be defined similarly. The maximum principle (Theorem \ref{thm-prelim:max_principle}) holds for weighted solutions as well. When $\psi=0$ in a certain region, $u$ reduces to a usual weak solution of IMCF.

In the context of mean curvature flow, similar warping and conformal transformations are used to relate mean curvature flow to minimal surfaces, see \cite{Ilmanen_1994, Mari-deOliveiraSavas-Halilaj-deSena_2023, Smoczyk_2001}.

\begin{proof}[Proof of Lemma \ref{lemma-prelim:def_weighted}] {\ }
	
	(1)\,$\Rightarrow$\,(2). Given a function $\tilde v(x,t)$ and domain $\tilde K$ with $\{\tilde u\ne\tilde v\}\Subset\tilde K\Subset\Omega\times S^1$. By enlarging $\tilde K$, we may assume that $\tilde K=K\times S^1$ for some $K\Subset\Omega$. 
	We need to show
	\begin{equation}\label{eq-weight:aux1}
		\int_{\tilde K}\big(|\tilde\D\tilde u|+\tilde u|\tilde\D\tilde u|\big)\,dV_{\tilde g}\leq\int_{\tilde K}\Big(|\tilde\D\tilde v|+\tilde v|\tilde\D\tilde u|\Big)\,dV_{\tilde g},
	\end{equation}
	where $|\tilde\D\cdot|$ denotes the norm of gradient with respect to $\tilde g$. Using the facts $dV_{\tilde g}(x,z)=e^{\psi(x)}dV_g(x)dz$, $\tilde u(x,z)=u(x)$ and $|\D\tilde v(x,z)|\geq|\D_x\tilde v(x,z)|$, it is sufficient to show
	\[2\pi\int_{K}e^{\psi}\big(|\D u|+u|\D u|\big)dV_g\leq\int_{K\times S^1}e^{\psi(x)}\Big(|\D_x\tilde v(x,z)|+\tilde v(x,z)|\tilde\D\tilde u(x)|\Big)\,dV_g(x)dz.\]
	However, this follows from Fubini's theorem and \eqref{eq-weight:energy_min}.
	
	(2)\,$\Rightarrow$\,(1): It is now our hypothesis that \eqref{eq-weight:aux1} holds for any competitor $\tilde v$. Choosing $\tilde v(x,t)=v(x)$ and $\tilde K=K\times S^1$, we find that \eqref{eq-weight:aux1} implies item (1).
	
	(1)\,$\Leftrightarrow$\,(3): Note that $dV_{g'}=e^{\frac{n\psi}{n-1}}dV_g$ and $|\D_{g'}f|=e^{-\frac{\psi}{n-1}}|\D_g f|$ for all $f$. Therefore, $u$ being a solution of $\IMCF{\Omega,g'}$ is equivalent to
	\[\int_K\big(e^{-\frac{\psi}{n-1}}|\D_g u|+ue^{-\frac{\psi}{n-1}}|\D_g u|\big)e^{\frac{n\psi}{n-1}}dV_g
		\leq \int_K\big(e^{-\frac{\psi}{n-1}}|\D_g v|+ve^{-\frac{\psi}{n-1}}|\D_g u|\big)e^{\frac{n\psi}{n-1}}dV_g\]
	whenever $\{u\ne v\}\Subset K\Subset\Omega$. This inequality is identical with \eqref{eq-weight:energy_min}.
\end{proof}

\subsection{Calibrated weak solutions}\label{subsec:calibrated}

The notion of calibrated weak solutions of IMCF plays a central role in the main existence theorem. The following definition first appeared in \cite[Section 3]{Huisken-Ilmanen_2001}:

\begin{defn}\label{def-prelim:calibrated}
	Let $\Omega$ be a domain in a manifold $M$, and $u\in\Lip_{\loc}(\Omega)$. We say that $u$ is a calibrated solution of IMCF in $\Omega$, if there is a measurable vector field $\nu$ such that:
	
	(1) $\esssup_\Omega|\nu|\leq1$, and $\metric{\D u}{\nu}=|\D u|$ almost everywhere,
	
	(2) for all $\varphi\in\Lip_0(\Omega)$ we have
	\begin{equation}\label{eq-prelim:calibration}
		\int_\Omega\big(\D\varphi\cdot\nu+\varphi|\D u|\big)=0.
	\end{equation}
	
	\noindent Given a solution $u$ of $\IMCF{\Omega}$, we say that $u$ is calibrated by $\nu$, if (1)(2) are satisfied for the data $u$, $\nu$.
\end{defn}


A calibrated solution is always a weak solution satisfying Definition \ref{def-prelim:weak_sol}. Indeed, given $v\in\Lip_{\loc}(\Omega)$ and $K$ with $\{u\ne v\}\Subset K\Subset M$, the inequality $J_u^K(u)\leq J_u^K(v)$ follows by taking $\varphi=u-v$ in \eqref{eq-prelim:calibration}.

\begin{remark}\label{rmk-prelim:calibration_normal_vec}
	Suppose a solution $u$ of $\IMCF{\Omega}$ is calibrated by $\nu$. Then for almost every $t$, we have $\nu_{E_t}=\nu$ a.e. on $\p^*E_t\cap\Omega$. This follows from $\metric{\D u}{\nu}=|\D u|$.
\end{remark}

We prove the following compactness theorem for calibrated solutions. In item (6) below, we say that a sequence of vector fields $\nu_i$ converge to $\nu$ weakly in $L^1_{\loc}$, if for all $L^\infty$ vector field $X$ with $\supp(X)\Subset\Omega$, we have $\int_\Omega\metric{\nu}{X}=\lim_{i\to\infty}\int_\Omega\metric{\nu_i}{X}$. This notion of convergence is independent of the choice of background metric.

\begin{theorem}[compactness of calibrated solutions]\label{thm-prelim:cptness_calibrated}
	Let $\Omega\subset M$ be a domain, and $g$ be a fixed smooth Riemannian metric on $\Omega$. Given the following data:
	
	(1) $\Omega_i$ is a sequence of domains that locally uniformly converge to $\Omega$,
	
	(2) $g_i$ are smooth metrics on $\Omega_i$, which locally uniformly converge to $g$,
	
	(3) $u_i\in\Lip_{\loc}(\Omega_i)$ solve $\IMCF{\Omega_i,g_i}$ and are calibrated by $\nu_i$,
	
	(4) for each $K\Subset\Omega$ we have $\sup_K|\D u_i|_{g_i}\leq C(K)$ for all sufficiently large $i$,
	
	(5) $u_i$ locally uniformly converges to a function $u$ on $\Omega$,
	
	(6) $\nu_i$ converges to a vector field $\nu$ weakly in $L^1_{\loc}$.
	
	\noindent Then $u$ is solves $\IMCF{\Omega,g}$ and is calibrated by $\nu$.
\end{theorem}
\begin{proof}
	From items (2)(4)(5) we have $u\in\Lip_{\loc}(\Omega)$. From (2)(6) and $\esssup|v_i|_{g_i}\leq1$ we have $\esssup|v|_g\leq1$. We have argued in \eqref{eq-prelim:limit_du} that for any $\varphi\in\Lip_0(\Omega)$ with $\varphi\geq0$, it holds
	\begin{equation}\label{eq-prelim:limit_du_2}
		\int\varphi|\D_g u|_g\,dV_g=\lim_{i\to\infty}\int\varphi|\D_{g_i}u_i|_{g_i}\,dV_{g_i}.
	\end{equation}
	Clearly, this identity holds for all $\varphi\in\Lip_0(\Omega)$ as well.
	
	Suppose $\varphi\in\Lip_0(\Omega)$. It follows from items (2)(6) that
	\begin{equation}\label{eq-prelim:aux4}
		\int\metric{\nu}{\D_g\varphi}_g\,dV_g=\lim_{i\to\infty}\int\metric{\nu_i}{\D_{g_i}\varphi}_{g_i}\,dV_{g_i}.
	\end{equation}
	Recall that each pair $(u_i,\nu_i)$ satisfies
	\[\int\big(\metric{\nu_i}{\D_{g_i}\varphi}_{g_i}+\varphi|\D_{g_i}u_i|_{g_i}\big)\,dV_{g_i}=0.\]
	Taking $i\to\infty$ and combining with \eqref{eq-prelim:limit_du_2} \eqref{eq-prelim:aux4}, the identity \eqref{eq-prelim:calibration} is verified.
	
	Finally, we show $\metric{\nu}{\D_g u}_g=|\D_gu|_g$ a.e.. Taking $\varphi=u_i\psi$ in the calibration condition \eqref{eq-prelim:calibration} for $u_i$ (where $\psi\in\Lip_0(\Omega)$ and $i$ is sufficiently large), we obtain
	\begin{align}
		0 &= \int\metric{\nu_i}{u_i\D_{g_i}\psi+\psi\D_{g_i}u_i}_{g_i}\,dV_{g_i}+\int u_i\psi|\D_{g_i} u_i|_{g_i}\,dV_{g_i} \nonumber\\
		&= \int\metric{\nu_i}{u_i\D_{g_i}\psi}_{g_i}\,dV_{g_i}+\int(1+u_i)\psi|\D_{g_i}u_i|_{g_i}\,dV_{g_i}. \label{eq-prelim:aux5}
	\end{align}
	By items (2)(5)(6), the first term of \eqref{eq-prelim:aux5} converges to $\int\metric{\nu}{u\D_g\psi}_g\,dV_g$. By \eqref{eq-prelim:limit_du_2} and items (2)(4)(5), the second term of \eqref{eq-prelim:aux5} converges to $\int(1+u)\psi|\D_g u|_g\,dV_g$. Therefore
	\[0=\int\metric{\nu}{u\D_g\psi}_g\,dV_g+\int(1+u)\psi|\D_g u|_g\,dV_g.\]
	On the other hand, since we have already verified \eqref{eq-prelim:calibration} for the pair $(u,\nu)$, we may apply $\varphi=u\psi$ there to obtain
	\[0=\int\metric{\nu}{u\D_g\psi}_g+\metric{\nu}{\psi\D_g u}_g+u\psi|\D_g u|_g.\]
	In comparison, it holds
	\[\int\psi\metric{\nu}{\D_g u}_g=\int\psi|\D_g u|_g.\]
	Since $\psi$ is arbitrary, we conclude that $\metric{\nu}{\D_g u}_g=|\D_g u|_g$ a.e.. This proves that $u$ is a weak solution calibrated by $\nu$.
\end{proof}

\subsection{The elliptic regularization process}\label{subsec:ellreg}

In \cite{Huisken-Ilmanen_2001}, the initial value problem of weak IMCF is solved by means of elliptic regularization. In this subsection, we recall the analytic results obtained there. We assume the following setup: $E_0$ is a precompact $C^{1,1}$ domain in a complete manifold $M$, and $v$ is proper function on $M$, such that $\{v<0\}\Supset E_0$, and $v$ is a smooth subsolution of IMCF in $M\setminus\bar{\{v<0\}}$ with nonvanishing gradient there (namely, $v$ is a smooth proper subsolution with an initial value containing $E_0$).

For each $\epsilon>0$ and $L>2$, consider the region $F_L=\{v<L\}$ (which is smooth and precompact by our assumptions), and the boundary value problem

\begin{align}[left=\empheqlbrace]
	& \div\Big(\frac{\D u_{\epsilon,L}}{\sqrt{\epsilon^2+|\D u_{\epsilon,L}|^2}}\Big)=\sqrt{\epsilon^2+|\D u_{\epsilon,L}|^2}\qquad\text{on }F_L\setminus\bar{E_0}, \label{eq-ellreg:reg_eq}\\
	& u_{\epsilon,L}=0\qquad\text{on }\p E_0, \label{eq-ellreg:reg_eq2}\\
	& u_{\epsilon,L}=L-2\qquad\text{on }\p F_L. \label{eq-ellreg:reg_eq3}
\end{align}
Equation \eqref{eq-ellreg:reg_eq} has the following meaning: if $u_\epsilon$ is a solution of \eqref{eq-ellreg:reg_eq} in a domain $\Omega$, then the function
\[U_\epsilon(x,z)=u_\epsilon(x)-\epsilon z\]
is a smooth solution of IMCF in $\Omega\times\RR$. Equivalently, the family of hypersurfaces
\[\Sigma_t^\epsilon=\graph\Big(\epsilon^{-1}\big(u_\epsilon(x)-t\big)\Big)\]
is a downward translating soliton of IMCF (with translation speed $\epsilon^{-1}$).

\vspace{3pt}

The gradient estimate for the smooth IMCF follows from a parabolic estimate. The following is proved in \cite{Huisken-Ilmanen_2001}, see also \cite{Choi-Daskalopoulos_2021}.

\begin{defn}\label{def-prelim:sigma_x}
	Let $\Omega$ be a domain in $M$, and $g$ be a Riemannian metric on $M$. For $x\in\Omega$, we define $\sigma(x;\Omega,g)$ to be the supremum of radius $r$ such that
	
	(1) $B_g(x,r)\Subset\Omega$, and $\Ric_g>-1/(100nr^2)$ in $B_g(x,r)$,
	
	(2) the distance function $p=d(\cdot,x)^2$ is smooth and satisfies $\D^2_gp<3g$ in $B_g(x,r)$.
	
	\noindent When there is no ambiguity, we write $\sigma(x)$ and omit the dependence on $\Omega,g$.
\end{defn}

\begin{theorem}[gradient estimate for smooth solutions, {\cite[(3.6)]{Huisken-Ilmanen_2001}}]\label{thm-prelim:grad_est}\label{thm-ellreg:grad_est_smooth} {\ }
	
	Let $\{\Sigma_t\}_{a\leq t\leq b}$ be a family of smooth hypersurfaces that solves the IMCF in $\Omega$. Denote by $H_t$ the mean curvature of $\Sigma_t$. Suppose $x\in\Sigma_t$ and $r<\sigma(x;\Omega,g)$. Define the maximal mean curvature on the parabolic boundary:
	\[H_r=\max\Big\{\sup_{\Sigma_a\cap B(x,r)}H_a,\ \sup_{s\in[a,t]}\big(\sup_{\p\Sigma_s\cap B(x,r)}H_s\big)\Big\}\]
	Then we have
	\begin{equation}
		H_t(x)\leq\max\Big\{H_r,\,\frac{C(n)}r\Big\}.
	\end{equation}
\end{theorem}

\begin{cor}\label{cor-ellreg:grad_est}
	Let $u\in C^\infty(\Omega)$ solve the smooth IMCF $\div\big(\frac{\D u}{|\D u|}\big)=\D u$ in $\Omega$, with $|\D u|$ nonvanishing. Then
	\begin{equation}
		|\D u|(x)\leq\frac{C(n)}{\sigma(x;\Omega,g)},\qquad \forall\,x\in\Omega.
	\end{equation}
	In particular, when $\Omega$ is precompact, we have $|\D u|(x)\leq C(\Omega)\cdot d(x,\p\Omega)^{-1}$.
\end{cor}

Combining \cite[Lemma 3.4]{Huisken-Ilmanen_2001} and \cite[Lemma 3.5]{Huisken-Ilmanen_2001}, we have the following existence theorem for the regularized equation:

\begin{theorem}[approximate existence]\label{thm-ellreg:existence} {\ }
	
	Let $M$ be complete, and $E_0\Subset M$ be a $C^{1,1}$ domain. Suppose $v\in C^\infty(M)$ is proper, such that $\{v<0\}\Supset E_0$, and $v$ is a subsolution of IMCF in $M\setminus\bar{\{v<0\}}$ with nonvanishing gradient there. Then for each $L>2$ there is a small $\epsilon(L)>0$ such that \eqref{eq-ellreg:reg_eq}\,$\sim$\,\eqref{eq-ellreg:reg_eq3} admits a smooth solution $u_{\epsilon,L}$ for all $0<\epsilon\leq\epsilon(L)$. Moreover, we have the lower bound
	\begin{equation}\label{eq-ellreg:lower_bound}
		\left\{\begin{aligned}
			& u_{\epsilon,L}\geq-\epsilon\qquad\text{in }\bar{F_L}\setminus E_0, \\
			& u_{\epsilon,L}\geq v-2\qquad\text{in }\bar{F_L}\setminus\{v<0\},
		\end{aligned}\right.
	\end{equation}
	and the gradient estimate
	\begin{equation}\label{eq-ellreg:grad_reg_sol}
		\big|\D u_{\epsilon,L}(x)\big|\leq\max\Big\{\sup_{B(x,r)\cap\p E_0}H_+,\ \sup_{B(x,r)\cap\p F_L}|\D u_{\epsilon,L}|\Big\}+2\epsilon+\frac{C(n)}{r}
	\end{equation}
	for all $x\in\bar{F_L}\setminus E_0$ and $0<r\leq\sigma(x;M,g)$, where $H_+=\max\{H_{\p E_0},0\}$.
\end{theorem}

To recover a solution of $\IMCF{\Omega}$, we take sequences $L_i\to\infty$, $\epsilon_i\to0$, and take limit of the corresponding solutions $u_i=u_{\epsilon_i,L_i}$. The following theorem follows \cite[Theorem 3.1]{Huisken-Ilmanen_2001}, and is proved in more generality for the need of blow-up analysis.

\begin{theorem}[convergence to calibrated solutions]\label{thm-ellreg:convergence}
	Let $\Omega\subset M$ be a domain, and $g$ be a fixed smooth Riemannian metric on $\Omega$. Given the following data:
	
	(1) $\Omega_i$ is a sequence of domains that converge locally uniformly to $\Omega$,
	
	(2) $g_i$ are smooth metrics on $\Omega_i$, which converge locally smoothly to $g$,
	
	(3) $\{\epsilon_i>0\}$ is a sequence with $\epsilon_i\searrow0$, and $u_i\in C^\infty(\Omega_i)$ are solutions of the equations
	\[\div_{g_i}\Big(\frac{\D_{g_i} u_i}{\sqrt{\epsilon_i^2+|\D_{g_i} u_i|_{g_i}^2}}\Big)=\sqrt{\epsilon_i^2+|\D_{g_i} u_i|_{g_i}^2}.\]
	
	Then a subsequence of $u_i$ converges in $C^0_{\loc}(\Omega)$ to a function $u\in\Lip_{\loc}(\Omega)$, which solves $\IMCF{\Omega,g}$. We have the gradient estimate
	\begin{equation}\label{eq-prelim:grad_est_limit}
		|\D_g u(x)|\leq\frac{C(n)}{\sigma(x;\Omega,g)},\qquad\forall\,x\in\Omega.
	\end{equation}
	Moreover, on $\Omega$ and $\Omega\times\RR$ respectively, the vector fields
	\[\nu_i=\frac{\D_{g_i}u_i}{\big(\epsilon_i^2+|\D_{g_i} u_i|_{g_i}^2\big)^{1/2}}
		\qquad\text{and}\qquad
		\bnu_i=\frac{\D_{g_i}u_i-\epsilon_i\p_z}{\big(\epsilon_i^2+|\D_{g_i} u_i|_{g_i}^2\big)^{1/2}}\]
	converge to some $\nu,\bnu$ in the weak topology of $L^1_{\loc}(\Omega)$ resp. $L^1_{\loc}(\Omega\times\RR)$, such that $\nu$ is the projection of $\bnu$ on the $\Omega$ factor, and $\nu$ calibrates $u$ in the sense of Definition \ref{def-prelim:calibrated}.
\end{theorem}
\begin{proof}
	On $\Omega_i\times\RR$ with the product metric $g_i+dz^2$, the functions $U_i(x,z)=u_i(x)-\epsilon_iz$ are smooth solutions of the IMCF. By Theorem \ref{thm-prelim:grad_est} and Definition \ref{def-prelim:sigma_x} we have the estimate
	\begin{equation}\label{eq-ellreg:grad_ui}
		|\D_{g_i}u_i(x)|\leq|\D_{g_i+dz^2}U_i(x,0)|\leq\frac{C(n+1)}{\sigma\big((x,0);\Omega_i\times\RR,g_i+dz^2\big)}\leq\frac{C(n+1)}{\sigma(x;\Omega_i,g_i)}.
	\end{equation}
	By condition (2) of the theorem, for all $K\Subset\Omega\times\RR$ there exists $i_0$ such that
	\[\inf_{x\in K}\inf_{i\geq i_0}\sigma(x;\Omega_i\times\RR,g_i+dz^2)>0.\]
	By the Arzela-Ascoli theorem, there is a subsequence such that $u_i\to u$ locally uniformly for some $u\in\Lip_{\loc}(\Omega)$. Set $U(x,z):=u(x)$, which is clearly the $C^0_{\loc}$ limit of $U_i(x,z)$. Note that $\sigma(x;\Omega,g)\leq2\sigma(x;\Omega_i,g_i)$ for sufficiently large $i$. Thus \eqref{eq-ellreg:grad_ui} passes to the limit and gives
	\[|\D_g u(x)|\leq\frac{C'(n)}{\sigma(x;\Omega,g)}.\]
	
	Since $U_i$ are smooth solutions, they are calibrated by the vector fields
	\[\bnu_i:=\frac{\D_{g_i+dz^2}U_i}{|\D_{g_i+dz^2}U_i|}=\frac{\D_{g_i}u_i-\epsilon_i\p_z}{\big(\epsilon_i^2+|\D_{g_i} u_i|_{g_i}^2\big)^{1/2}}.\]
	By the Dunford-Pettis theorem and a diagonal argument, there is a subsequence such that $\bnu_i$ converges to some $\bnu$ weakly in $L^1_{\loc}(\Omega\times\RR)$. Now all the conditions of Theorem \ref{thm-prelim:cptness_calibrated} are met, and it follows that $U$ solves $\IMCF{\Omega\times\RR;g+dz^2}$ and is calibrated by $\bnu$.
	
	Note that $\bnu_i$ are invariant under vertical translation, and this property passes to the limit $\bnu$. Let $\nu$ be the projection of $\bnu$ on the $\Omega$ factor. It is easily seen that $\nu$ is the $L^1_{\loc}$ weak limit of $\nu_i$.
	
	Finally, we show that $u$ is a weak solution in $\Omega$ calibrated by $\nu$. Indeed, we have $|\nu|\leq|\bnu|\leq1$ and $\D u(x)\cdot\nu(x)=\D U(x,0)\cdot\bnu(x,0)=|\D U(x,0)|=|\D u(x)|$ almost everywhere. To verify the condition \eqref{eq-prelim:calibration}, we fix a cutoff function $\eta\in C^\infty(\RR)$ with $\eta|_{(-\infty,-1]}\equiv0$ and $\eta|_{[0,\infty)}\equiv1$. For $R>0$ we set $\rho_R(z)=\eta(z)\eta(R-z)$. Now for a fixed $\phi\in\Lip_{\loc}(\Omega)$, we test the calibration property of $U$ with the function $\phi(x)\rho_R(z)$ and find
	\[\begin{aligned}
		0 &= \int_{\Omega\times\RR}\big(\D_x\phi(x)\rho_R(z)+\phi(x)\rho'_R(z)\p_z\big)\cdot\bnu+\phi(x)\rho_R(z)|\D_xu(x)|\,dx\,dz \\
		&= R\int_\Omega\D\phi\cdot\nu+\phi|\D u| \\
		&\qquad +\int_{\Omega\times([-1,0]\times[R,R+1])}\big(\D_x\phi(x)\rho_R(z)+\phi(x)\rho'_R(z)\p_z\big)\cdot\bnu+\phi(x)\rho_R(z)|\D_xu(x)|\,dx\,dz.
	\end{aligned}\]
	Taking $R\to\infty$, we have $0=\int_\Omega\D\phi\cdot\nu+\phi|\D u|$. This verifies that $\nu$ calibrates $u$.
\end{proof}

In \cite{Huisken-Ilmanen_2001}, the elliptic regularization finally leads to the following existence theorem:

\begin{theorem}[{\cite[Theorem 3.1]{Huisken-Ilmanen_2001}}]
	Suppose $M$ is complete, and there exists a proper weak subsolution with a precompact initial value. Then with any $C^{1,1}$ initial data $E_0\Subset M$, there exists a unique proper solution of $\IVP{M;E_0}$.
\end{theorem}

\begin{remark}\label{rmk-ellreg:weighted}
	When a weighted IMCF is considered, with the interpretation of Lemma \ref{lemma-prelim:def_weighted}(3), the corresponding elliptic regularization takes the form
	\begin{equation}\label{eq-ellreg:aux3}
		\div\Big(e^\psi\frac{\D u}{\sqrt{\epsilon^2e^{2\psi/(n-1)}+|\D u|^2}}\Big)=e^\psi\sqrt{\epsilon^2e^{2\psi/(n-1)}+|\D u|^2}.
	\end{equation}
	To derive this, we note that the regularized equation writes
	\[\div_{g'}\Big(\frac{\D_{g'}u}{\sqrt{\epsilon^2+|\D_{g'}u|^2}}\Big)=\sqrt{\epsilon^2+|\D_{g'}u|^2},\qquad\text{where}\ \ g'=e^{\frac{2\psi}{n-1}}g.\]
	Since $\D_{g'}u=e^{-2\psi/(n-1)}\D_gu$ and $|\D_{g'}u|^2=e^{-2\psi/(n-1)}|\D_gu|^2$, this is equivalent to
	\[\div_{g'}\Big(e^{-\psi/(n-1)}\frac{\D u}{\sqrt{\epsilon^2e^{2\psi/(n-1)}+|\D u|^2}}\Big)=e^{-\psi/(n-1)}\sqrt{\epsilon^2e^{2\psi/(n-1)}+|\D u|^2}.\]
	Since $\div_{g'}(X)=\div_gX+\frac{n}{n-1}\metric{X}{\D_g\psi}_g$ for all $X$, from this we obtain \eqref{eq-ellreg:aux3}.
\end{remark}

\subsection{An excess inequality}

By Lemma \ref{lemma-prelim:out_min}(i), each $E_t$ in a weak solution is locally outward minimizing. On the other hand, the following useful lemma provides a quantitative inward-minimizing of $E_t$. Similar arguments played a major role in proving the main theorem of \cite{Xu_2023_proper}.

\begin{lemma}[excess inequality]\label{lemma-prelim:excess_ineq} {\ }
	
	Suppose $u\in\Lip_{\loc}(\Omega)$ is a supersolution of $\IMCF{\Omega}$, and $t\in\RR$. Let $F\Subset\Omega$ be a set with finite perimeter. Then for any domain $K$ with $F\Subset K\Subset\Omega$, we have
	\begin{equation}\label{eq-prelim:excess_ineq}
		\Ps{E_t;K}\leq\Ps{E_t\setminus F;K}+\int_{\inf_F(u)}^te^{t-s}\Ps{F;E_s}\,ds.
	\end{equation}
	In particular,
	\begin{equation}\label{eq-prelim:excess_ineq_2}
		\Ps{E_t;K}\leq\Ps{E_t\setminus F;K}+\big(e^{t-\inf_F(u)}-1\big)\Ps{F;E_t}.
	\end{equation}
\end{lemma}
\begin{proof}
	Denote $t_0=\inf_F(u)$. By a zero measure modification, we may assume $F=F^{(1)}$. From the energy comparison $J_u^K(E_t)\leq J_u^K(E_t\setminus F)$ and the coarea formula, we have
	\begin{equation}\label{eq-prelim:aux1}
		\begin{aligned}
			\Ps{E_t;K} &\leq \Ps{E_t\setminus F;K}+\int_{t_0}^t\H^{n-1}\big(\p^*E_s\cap E_t\cap F\big)\,ds \\
			&= \Ps{E_t\setminus F;K}+\int_{t_0}^t\H^{n-1}\big(\p^*E_s\cap F\big)\,ds
		\end{aligned}
	\end{equation}
	If $\H^{n-1}(\p^*E_t\cap\p^*F)=0$ (which holds for almost every $t$), then we apply the decomposition identities \eqref{eq-gmt:Federer_decomp} \eqref{eq-gmt:set_operation} and Lemma \ref{lemma-prelim:density_one} to obtain
	\begin{equation}\label{eq-prelim:aux2}
		\H^{n-1}(\p^*E_t\cap F)\leq\Ps{F;E_t}+\int_{t_0}^t\H^{n-1}(\p^*E_s\cap F)\,ds.
	\end{equation}
	By a standard Gronwall argument, \eqref{eq-prelim:aux2} implies
	\begin{equation}\label{eq-prelim:aux3}
		\H^{n-1}(\p^*E_t\cap F)\leq\Ps{F;E_t}+\int_{t_0}^t e^{t-s}\Ps{F;E_s}\,ds\qquad\text{for a.e. $t$}.
	\end{equation}
	The lemma follows by inserting \eqref{eq-prelim:aux3} into \eqref{eq-prelim:aux1}.
\end{proof}

\subsection{Further analytic properties of weak solutions}

We include several technical lemmas that are frequently used later.

In summary, if $\Omega$ is locally Lipschitz and $u$ is a solution of $\IMCF{\Omega}$, then $u$ has the following a priori regularity: $\Lip_{\loc}(\Omega)$ (given by the definition), $\BV_{\loc}(\bar\Omega)$ (by Lemma \ref{lemma-prelim:BV_loc}(i)), and $L^p_{\loc}(\bar\Omega)$ for all $p\geq1$ (by Lemma \ref{lemma-obs:integrable}). In fact, the proof of Lemma \ref{lemma-obs:integrable} implies a stronger fact: for all $K\Subset M$ there exists $c(K)>0$ such that $e^{c(K)|u|}\in L^1(\Omega\cap K)$.

\begin{lemma}\label{lemma-prelim:density_one}
	Suppose $u\in\Lip_{\loc}(\Omega)$ is a solution of $\IMCF{\Omega}$. Then for every $t$ we have $E_t=E_t^{(1)}$ (which is the set of points with density 1 with respect to $E_t$).
\end{lemma}
\begin{proof}
	Since $E_t$ is open, it is sufficient to show $E_t^{(1)}\cap\p E_t=\emptyset$. Suppose $x\in\p E_t$. Since $u$ is continuous, there exists a sequence of times $t_i\nearrow t$ and points $x_i\in\p^*E_{t_i}$ with $x_i\to x$. Since $E_t=\bigcup_i E_{t_i}$, by \eqref{eq-prelim:Lam_r0_min} and Lemma \ref{lemma-gmt:Lam_r0_convergence}(i) it follows that $x\notin E_t^{(1)}$.
\end{proof}

\begin{lemma}\label{lemma-obs:inner_approx}
	Suppose $\Omega_i, \Omega\subset M$ are sets with locally finite perimeter, such that $\chi_{\Omega_i}\to\chi_\Omega$ in $L^1_{\loc}$ and $|D\chi_{\Omega_i}|\rightharpoonup|D\chi_\Omega|$ weakly as measures. Then for any $K\Subset M$ with $\H^{n-1}(\p^*\Omega\cap\p K)=0$ and any set $A\subset\Omega$ with locally finite perimeter, we have
	\[\Ps{A;K}=\lim_{i\to\infty}\P{A\cap\Omega_i;K}.\]
\end{lemma}
\begin{proof}
	It is well-known that
	\begin{equation}\label{eq-obs:aux1}
		\Ps{A;K}\leq\liminf_{i\to\infty}\P{A\cap\Omega_i;K},
	\end{equation}
	and moreover
	\begin{equation}\label{eq-obs:aux2}
		\Ps{\Omega;K}\leq\liminf_{i\to\infty}\P{A\cup\Omega_i;K},\qquad\Ps{\Omega;K}=\lim_{i\to\infty}\Ps{\Omega_i;K}.
	\end{equation}
	To show the reverse direction of \eqref{eq-obs:aux1}, we note that
	\[\Ps{A;K}+\P{\Omega_i;K}\geq\P{A\cap\Omega_i;K}+\P{A\cup\Omega_i;K}.\]
	Taking $i\to\infty$ and using \eqref{eq-obs:aux2}, we obtain
	\[\Ps{A;K}+\P{\Omega;K}\geq\limsup_{i\to\infty}\P{A\cap\Omega_i;K}+\Ps{\Omega;K}. \qedhere\]
\end{proof}

\begin{lemma}[a priori global regularity]\label{lemma-prelim:BV_loc} {\ }
	
	Let $\Omega\subset M$ be a locally Lipschitz domain, and $u\in\Lip_{\loc}(\Omega)$.
	
	(i) If $u$ is a subsolution of $\IMCF{\Omega}$, then for all $K\Subset M$ it holds
	\begin{equation}\label{eq-prelim:subsol_per}
		\int_{\Omega\cap K}|\D u|\leq\Ps{\Omega\cap K}\quad\ \text{and}\quad\
		\Ps{E_t;K}\leq\Ps{\Omega\cap K}\ \ \ \forall\,t\in\RR.
	\end{equation}
	In particular, $u\in\BV_{\loc}(\bar\Omega)$ and each $E_t$ has locally finite perimeter in $M$.
	
	(ii) If $u$ is a supersolution of ${\IMCF\Omega}$, and for all $K\Subset M$ it holds $\inf_{\Omega\cap K}u\geq T(K)$ for some $T(K)$, then each $E_t$ has locally finite perimeter in $M$. More precisely, we have
	\begin{equation}\label{eq-prelim:supersol_per}
		\Ps{E_t;\Omega\cap K}\leq e^{t-T(K)}\Ps{\Omega\cap K},\quad
		\Ps{E_t;K}\leq 2e^{t-T(K)}\Ps{\Omega\cap K},
	\end{equation}
	for all domains $K\Subset M$. We also have
	\begin{equation}\label{eq-prelim:supersol_per2}
		\int_{E_t\cap K}|\D u|\leq e^{t-T(K)}\Ps{\Omega\cap K}.
	\end{equation}
	In particular, if $u\in L^\infty_{\loc}(\bar\Omega)$ then $u\in\BV_{\loc}(\bar\Omega)$.
\end{lemma}
\begin{proof}
	Fix $t\in\RR$. We find a sequence of locally Lipschitz domains $\Omega_1\Subset\Omega_2\Subset\cdots\subset\Omega$, such that $\H^{n-1}(\p^*E_t\cap\p^*\Omega_i)=0$ for each $i$, and $\Omega_i\to\Omega$ locally uniformly, and $|D\chi_{\Omega_i}|\rightharpoonup|D\chi_\Omega|$ weakly as measures. For $K$ in either statement of the lemma, we choose another smooth precompact domain $K'\Supset K$ with $\H^{n-1}(\p^*\Omega\cap\p K')=0$.
	
	(i) Suppose $u$ is a weak subsolution. For each $i$ we have $\inf_{\Omega_{i+1}\cap K'} u\geq T_i$ for some $T_i\in\RR$; thus $E_{T_i}\cap(\Omega_{i+1}\cap K')=\emptyset$. Using the subsolution property, we have
	\begin{equation}\label{eq-prelim:using_subsol}
		0=J_u^{\Omega_{i+1}\cap K'}(E_{T_i})\leq J_u^{\Omega_{i+1}\cap K'}(\Omega_i\cap K)=\Ps{\Omega_i\cap K}-\int_{\Omega_i\cap K}|\D u|.
	\end{equation}
	Letting $i\to\infty$ in \eqref{eq-prelim:using_subsol} and using Lemma \ref{lemma-obs:inner_approx} (there with the choice $A=\Omega\cap K$ and with $K$ replaced by $K'$), we obtain $\int_{\Omega\cap K}|\D u|\leq\Ps{\Omega\cap K}$. To prove the second statement of \eqref{eq-prelim:subsol_per}, we recall that $E_t$ is locally outward minimizing in $\Omega$, hence $\Ps{E_t;K'}\leq\Ps{E_t\cup(\Omega_i\cap K);K'}$ for each $i$. This implies $\Ps{E_t\cap \Omega_i\cap K;K'}\leq\Ps{\Omega_i\cap K;K'}$. Taking $i\to\infty$ and applying Lemma \ref{lemma-obs:inner_approx} to both sides, we conclude that $\Ps{E_t\cap K}\leq\Ps{\Omega\cap K}$. In particular, $\Ps{E_t;K}\leq\Ps{\Omega\cap K}$. This proves \eqref{eq-prelim:subsol_per}. To obtain $u\in\BV_{\loc}(\bar\Omega)$ we need $u\in L^1_{\loc}(\bar\Omega)$ as well, but this follows from $u\in\Lip_{\loc}(\Omega)$ and \eqref{eq-prelim:subsol_per} and that $\Omega$ is locally Lipschitz.
	
	(ii) Suppose $u$ is as in the statement. We first assume that $t$ satisfies $\H^{n-1}\big(\p^*E_t\cap\p K\big)=0$. From Lemma \ref{lemma-prelim:excess_ineq} we obtain
	\[\Ps{E_t;\Omega_{i+1}\cap K'}\leq\P{E_t\setminus(\Omega_i\cap K);\Omega_{i+1}\cap K'}+(e^{t-T(K)}-1)\P{\Omega_i\cap K;E_t}.\]
	Since $\H^{n-1}(\p^*E_t\cap\p^*K)=\H^{n-1}(\p^*E_t\cap\p^*\Omega_i)=0$ for all $i$, we can decompose the first two terms using \eqref{eq-gmt:Federer_decomp} \eqref{eq-gmt:set_operation} and Lemma \ref{lemma-prelim:density_one}, and obtain
	\[\begin{aligned}
		\P{E_t;\Omega_i\cap K} &\leq \P{E_t;(\Omega_i\cap K)^{(1)}}
		\leq \P{\Omega_i\cap K;E_t^{(1)}}+(e^{t-T(K)}-1)\Ps{\Omega_i\cap K} \\
		&\leq e^{t-T(K)}\Ps{\Omega_i\cap K}.
	\end{aligned}\]
	Taking $i\to\infty$, this implies the first inequality of \eqref{eq-prelim:supersol_per}. The case of all $t$ follows by lower semi-continuity. The second inequality in \eqref{eq-prelim:supersol_per} holds since $\Ps{E_t;K}=\Ps{E_t;\Omega\cap K}+\H^{n-1}\big(\p^*E_t\cap\p^*\Omega\cap K\big)$. Finally, \eqref{eq-prelim:supersol_per2} follows by the coarea formula.
\end{proof}

\section{Weak formulations with outer obstacle}\label{sec:obs}

In this section, we set up the general framework for IMCF with outer obstacle. We make precise the variational principles summarized in Definition \ref{def-intro:summary_defs} (see Subsections \ref{subsec:obs_energy_E}, \ref{subsec:obs_diri}), and address some basic implications (see Subsections \ref{subsec:obs_energy_E}, \ref{subsec:obs_min}). In Subsection \ref{subsec:obs_tangent} we prove Lemma \ref{lemma-intro:asymp_tangent}. Then, in Subsection \ref{subsec:obs_max_prin}, we prove a weak maximum principle.

\subsection{Formulation using sub-level sets}\label{subsec:obs_energy_E}

We use the following variational principle as the starting point of the theory. Recall that all the manifolds are assumed to be smooth, connected, oriented, without boundary.

\begin{defn}\label{def-obs:energy_E}
	Let $M$ be a manifold without boundary, and $\Omega\subset M$ be a locally Lipschitz domain. For a function $u\in\Lip_{\loc}(\Omega)$, a domain $K\Subset M$, and a set $E\subset\Omega$, we define the energy
	\begin{equation}\label{eq-obs:energy_E}
		\tJ_u^K(E):=\Ps{E;K}-\int_{E\cap K}|\D u|
	\end{equation}
	whenever the two terms are not both infinite.
	
	We say that a set $E\subset\Omega$ locally minimizes $\tJ_u$ (resp. minimizes from inside, outside), if for any $F\subset\Omega$ (resp. for any $F\subset\Omega$ with $F\subset E$, $F\supset E$) and any domain $K$ that satisfy $E\Delta F\Subset K\Subset M$, we have
	\begin{equation}\label{eq-obs:energy_comp}
		\tJ_u^K(E)\leq\tJ_u^K(F)
	\end{equation}
	whenever both energies are defined.
\end{defn}

\begin{defn}[outer obstacle I]\label{def-obs:obstacle1} {\ }
	
	Let $M$ be a manifold without boundary, and $\Omega$ be a locally Lipschitz domain in $M$. Given $u\in\Lip_{\loc}(\Omega)$. We say that $u$ is a (sub-, super-) solution of $\IMCFOOinthm{\Omega}{\p\Omega}$, if $E_t:=\{u<t\}$ locally minimizes $\tJ_u$ (resp. minimizes from outside, inside) for each $t\in\RR$.
\end{defn}

We emphasize once more that $E_t$ are viewed as subsets of $\Omega$ hence $M$, and its perimeter $\Ps{E;K}$ contains the area of $\p^*E\cap K\cap\Omega$ and $\p^*E\cap K\cap\p\Omega$. Note that:

(1) A solution of $\IMCFOO{\Omega}{\p\Omega}$ is clearly a solution of $\IMCF{\Omega}$.

(2) If $M=\Omega$, then $\IMCFOO{\Omega}{\p\Omega}$ is equivalent to $\IMCF{\Omega}$.

Given an interior solution $u$ (i.e. $u$ solves $\IMCF{\Omega}$), we say that $u$ \textit{respects the obstacle} $\p\Omega$ if $u$ actually solves $\IMCFOO{\Omega}{\p\Omega}$. When we need to clarify the background metric, we will write $\IMCFOO{\Omega,g}{\p\Omega}$.


\begin{remark}[integrability in the definition of $\tJ_u$]\label{rmk-obs:integrability} {\ }
	
	The energy \eqref{eq-obs:energy_E} is not a priori defined, since $u$ is only guaranteed interior regularity. However, if $u$ is a subsolution of $\IMCF{\Omega}$, then $\tJ_u^K(E_t)$ is always defined and is finite, due to Lemma \ref{lemma-prelim:BV_loc}(i). By the same lemma, $\tJ_u^K(E)$ is defined and finite provided $E$ has locally finite perimeter (which will always be the case in specific energy comparisons). Hence, there is no integrability issue for interior subsolutions.
	
	The case of supersolution is more complicated. If $u$ is a supersolution of $\IMCF{\Omega}$ with additionally
	\begin{equation}\label{eq-obs:aux8}
		\inf_{\Omega\cap K}u>-\infty\ \ \text{ for all }\ \ K\Subset M,
	\end{equation}
	then $\tJ_u^K(E_t)$ is always defined and is finite, by Lemma \ref{lemma-prelim:BV_loc}(ii). Since any competitor set $E$ is contained in $E_t$, it follows that $\int_{E\cap K}|\D u|\leq\int_{E_t\cap K}|\D u|$, thus $\tJ_u^K(E)$ is also defined with finite value. On the other hand, see Remark \ref{rmk-obs:abnormal_supersol} below that $\tJ_u=-\infty$ may occur for general supersolutions.
\end{remark}

\begin{remark}[a supersolution not belonging to $\BV_{\loc}(\bar\Omega)$]\label{rmk-obs:abnormal_supersol} {\ }
	
	Consider $\Omega=\{(x,y): -1<x<1\}\subset\RR^2$ and the function $u(x,y)=\tan(\pi x/2)$. We argue that $u$ is a supersolution of $\IMCFOO{\Omega}{\p\Omega}$. First, it is easy to see that $u$ is a supersolution of $\IMCF{\Omega}$. Let $F\subset E_t$ be a competitor with $E_t\setminus F\Subset\RR^2$. One of the following two cases must occur for $F$.
	
	(1) $E_t\Delta F\Subset\Omega$. In this case the energy comparison is entirely interior.
	
	(2) $\bar{E_t\Delta F}$ has nonempty intersection with $\{x=-1\}$. In this case $\tJ_u^K(E_t)=-\infty$ for all $K\Supset E_t\Delta F$, thus $\tJ_u^K(E_t)\leq\tJ_u^K(F)$ trivially holds.
	
	Therefore, $u$ is a supersolution satisfying Definition \ref{def-obs:obstacle1}.
	
	On the other hand, for any increasing function $f\in C^\infty(-1,1)$ with $\lim_{x\to-1}f(x)>-\infty$, the interior supersolution $u(x,y)=f(x)$ does not respect the obstacle $\p\Omega$. Informally speaking, such function is only a supersolution of $\IMCFOO{\Omega}{\{x=1\}}$.
\end{remark}

The following remark includes several useful observations regarding the definitions.

\begin{remark}[properties and relations]\label{rmk-obs:definition} {\ }
	
	\vspace{-6pt}
	\begin{enumerate}[label=(\roman*)]
		\item Note the differences between Definition \ref{def-obs:energy_E}, \ref{def-obs:obstacle1} and the interior formulation: (a) the comparison set $E$ must be contained in $\Omega$, but the difference set $E\Delta E_t$ is allowed to touch $\p\Omega$ (this characterizes an outer obstacle problem); (b) the boundary portion of perimeter $\p^*E\cap\p\Omega$ is contained in the energy $\tJ_u$. If we remove the boundary portion in item (b), then we obtain the energy
		\[\hat J_{0;u}^K(E)=\P{E;\Omega\cap K}-\int_{E\cap\Omega\cap K}|\D u|,\]
		which describes the weak IMCF with free boundary \cite{Marquardt_2017}. For a constant $\th\in(-1,1)$, we may vary the energy to be
		\[\hat J_{\th;u}^K(E)=\P{E;\Omega\cap K}+\th\cdot\H^{n-1}\big(\p^*E\cap\p^*\Omega\cap K\big)-\int_{E\cap\Omega\cap K}|\D u|.\]
		This describes weak solutions with capillary boundary condition, corresponding to the flow such that each hypersurface keeps the contact angle $\arccos(\th)$ with $\p\Omega$. No existence result is known for the weak IMCF with capillary conditions, except for the free boundary case $\th=0$ \cite{Koerber_2020, Marquardt_2017} and the obstacle case $\th=1$ considered by us.
		
		\item The following inequality is useful:
		\begin{equation}\label{eq-obs:cap_cup}
			\tJ_u^K(E\cap F)+\tJ_u^K(E\cup F)\leq\tJ_u^K(E)+\tJ_u^K(F),
		\end{equation}
		for all $E,F\subset\Omega$ with finite perimeter in $K$ and $u\in\BV(\Omega\cap K)$. It follows that $u$ is a solution of $\IMCFOO{\Omega}{\p\Omega}$ if it is both a subsolution and supersolution.
		
		\item Let $u$ be a solution of $\IMCFOO{\Omega}{\p\Omega}$. If each $E_t$ is precompact in $M$, then the energy comparison forces $\tJ_u^K(E_s)=\tJ_u^K(E_t)$ for all $s<t$. This implies
		\begin{equation}\label{eq-obs:per_sub_ode}
			\Ps{E_s}=\Ps{E_t}-\int_s^t\Ps{E_r;\Omega}\,dr\geq\Ps{E_t}-\int_s^t\Ps{E_r}\,dr.
		\end{equation}
		In particular, we have sub-exponential growth of area
		\begin{equation}\label{eq-obs:sub_exp_growth}
			\Ps{E_t}\leq e^{t-s}\Ps{E_s},\qquad\forall\,t>s,
		\end{equation}
		with strict inequality when $\H^{n-1}\big(\p^*E_s\cap\p^*\Omega\big)\ne0$. In Example \ref{ex-ex:epicycloid} of expanding epicycloids, for instance, it holds $\Ps{E_t}=e^{c(t-s)}\Ps{E_s}$ where $c\in(0,1)$ is the expanding rate defined there. From $\tJ_u^K(E_s)=\tJ_u^K(E_t)$ we also have $\Ps{E_s;\Omega}\geq\Ps{E_t;\Omega}-\int_s^t\Ps{E_r;\Omega}\,dr$, so the function $t\mapsto e^{-t}\Ps{E_t;\Omega}$ is also non-increasing.
		
		\item For $u$ a solution of $\IMCFOO{\Omega}{\p\Omega}$, the set $E_t^+$ also locally minimizes $\tJ_u$ for each $t\in\RR$. This follows from the fact $E_t^+=\bigcap_{s>t}E_s$ and the standard set-replacing argument (see for example \cite[Theorem 21.14]{Maggi}).
		
		\item When $\Omega\Subset M$, the constant function on $\Omega$ is \textit{not} a solution of $\IMCFOO{\Omega}{\p\Omega}$. Indeed, setting $u\equiv c$ we find that $\tJ_u^K(E_{c+1}(u))=\Ps{\Omega}>0=\tJ_u^K(\emptyset)$, for any $K\Supset\Omega$. For a noncompact $\Omega$, constant functions are solutions of $\IMCFOO{\Omega}{\p\Omega}$ if and only if $\Omega$ is locally inward perimeter-minimizing. An interesting example for this is $M=\HH^2\times\RR$ and $\Omega=\{-\pi/2<z<\pi/2\}$, where $z$ is the coordinate in the $\RR$-direction. This set is inward-minimizing since it is calibrated by the vector field $\nu=\sin(z)\p_z-\cos(z)\tanh(r/2)\p_r$.
		
		\item When $\Omega$ is precompact, there exists no solution $u$ of $\IMCFOO{\Omega}{\p\Omega}$ such that $\inf_\Omega(u)>-\infty$. Indeed, if $\inf_\Omega (u)=-T>-\infty$, then by \eqref{eq-obs:sub_exp_growth} we have $\Ps{E_t}\leq e^{t+T+1}\Ps{E_{-T-1}}=0$ for all $t$, which is impossible. On the other hand, Example \ref{ex-ex:bounded_domain} shows that there do exist solutions that are unbounded below. For these solutions $u$, it holds $\Ps{\Omega}=\int_\Omega|\D u|$, and in general, $\Ps{E_t}=\int_{E_t}|\D u|$ for all $t$. This follows by combining Lemma \ref{lemma-prelim:BV_loc}(i) and the comparison $\tJ_u(E_t)\leq\tJ_u(\emptyset)$.
		
		\item The validity of \eqref{eq-obs:energy_comp} is not affected by the specific choice of $K$, as long as $u\in\BV(\bar\Omega)$ and $K\Supset E_t\Delta E$. Therefore, we will usually choose $K$ at our convenience.
	\end{enumerate}
\end{remark}

\begin{defn}[initial value problem]\label{def-obs:ivp} {\ }
	
	Let $M$ be a manifold without boundary, and $\Omega$ be a locally Lipschitz domain in $M$, and $E_0\subset\Omega$ be a $C^{1,1}$ domain with $\p E_0\cap\p\Omega=\emptyset$. A function $u\in\Lip_{\loc}(\Omega)$ is called a (sub-, super-) solution of $\IVPOOinthm{\Omega;E_0}{\p\Omega}$ in $M$, if
	
	(i) $E_0=\{u<0\}$,
	
	(ii) $u|_{\Omega\setminus\bar{E_0}}$ is a (sub-, super-) solution of $\IMCFOOinthm{\Omega\!\setminus\!\bar{E_0}}{\p\Omega}$.
	
	We say that two solutions $u_1,u_2$ are equivalent, if $u_1=u_2$ on $\Omega\setminus E_0$.
\end{defn}

Similar to the interior case, we have an equivalent formulation stated below.

\begin{theorem}[initial value problem II]\label{thm-obs:ivp2} {\ }
	
	Let $M$, $\Omega$, $E_0$ be as in Definition \ref{def-obs:ivp}, and suppose $u\in\Lip_{\loc}(\Omega)$. Then $u$ is a solution of $\IVPOOinthm{\Omega;E_0}{\p\Omega}$ if and only if
	
	(i) $E_0=\{u<0\}$,
	
	(ii) for any $t>0$, any set $E$ with $E_0\subset E\subset\Omega$, and any domain $K$ with $E\Delta E_t\Subset K\Subset M$, we have $\tJ_u^K(E_t(u))\leq\tJ_u^K(E)$.
\end{theorem}
\begin{proof}
	The ``only if'' part: in the item (ii) here, if we additionally assume $\p E_t\cap\p E_0=\p E\cap\p E_0=\emptyset$, then $E\Delta E_t\Subset M\setminus\bar{E_0}$, and the result is directly implied by Definition \ref{def-obs:ivp}(ii). We always have $\p E_t\cap\p E_0=\emptyset$ by condition (i). The other assumption $\p E\cap\p E_0=\emptyset$ can be achieved by perturbing $E$ outward with a small increase of $\tJ_u^K(E)$.
	
	The ``if'' part: we only need to verify that for each $t\leq0$ and any $E\subset\Omega\setminus\bar{E_0}$ satisfying $E\Delta E_t\Subset K\Subset M\setminus\bar{E_0}$, it holds $\tJ_{u}^K(E_t)\leq\tJ_{u}^K(E)$. Since $E_t(u)\cap(\Omega\setminus\bar{E_0})=\emptyset$, this is equivalent to showing that $\tJ_u^K(E)\geq0$. To prove this, take $s>0$ such that $\H^{n-1}\big(\p^*E_s\cap\p^*E\big)=0$. Comparing $\tJ_u^K(E_s)\leq\tJ_u^K(E\cup E_s)$ by condition (ii), and decomposing the perimeters using \eqref{eq-gmt:Federer_decomp} \eqref{eq-gmt:set_operation}, we find
	\[\begin{aligned}
		\int_{(E\setminus E_s)\cap K}|\D u| &\leq \P{E\cup E_s;K}-\P{E_s;K} \\
		&= \H^{n-1}\big(\p^*E\cap E_s^{(0)}\cap K\big)-\H^{n-1}\big(\p^*E_s\cap E^{(0)}\cap K\big).
	\end{aligned}\]
	The right hand side is trivially no greater than $\Ps{E;K}$. Taking $s\to0$, the left hand side converges to $\int_{(E\setminus E_0^+)\cap K}|\D u|=\int_{E\cap K}|\D u|$. This shows $\tJ_u^K(E)\geq0$.
\end{proof}

The following observation is important in relating interior solutions to solutions with outer obstacles. See also Theorem \ref{thm-obs:max_principle}.

\begin{theorem}[automatic subsolution]\label{thm-obs:auto_subsol} {\ }
	
	Suppose $\Omega$ is a locally Lipschitz domain, and $u$ is a subsolution of $\IMCF{\Omega}$. Then $u$ is a subsolution of $\IMCFOOinthm{\Omega}{\p\Omega}$.
\end{theorem}
\begin{proof}
	Let a competitor set $E$ satisfy $E_t\subset E\subset\Omega$ and $E\setminus E_t\Subset M$. Choose $K\Supset E\setminus E_t$ such that $\H^{n-1}(\p^*\Omega\cap\p K)=0$. Choose $\Omega_1\Subset\Omega_2\Subset\cdots\Subset\Omega$ a sequence of locally Lipschitz domains, such that $\bigcup\Omega_i=\Omega$ and $|D\chi_{\Omega_i}|\rightharpoonup|D\chi_\Omega|$ weakly as measures. By the interior variational principle, we have
	\[\begin{aligned}
		J_u^{\Omega_{i+1}\cap K}(E_t)
		&\leq J_u^{\Omega_{i+1}\cap K}\big(E_t\cup(E\cap\Omega_i)\big) \\
		&\leq J_u^{\Omega_{i+1}\cap K}(E_t)
		+ J_u^{\Omega_{i+1}\cap K}(E\cap\Omega_i)
		- J_u^{\Omega_{i+1}\cap K}(E_t\cap E\cap\Omega_i) \\
		&= J_u^{\Omega_{i+1}\cap K}(E_t)+J_u^K(E\cap\Omega_i)-J_u^K(E_t\cap\Omega_i).
	\end{aligned}\]
	Thus
	\[J_u^K(E_t\cap\Omega_i)\leq J_u^K(E\cap\Omega_i).\]
	Applying Lemma \ref{lemma-obs:inner_approx} to both $E_t$ and $E$, and noticing $u\in\BV(K\cap\Omega)$ by Lemma \ref{lemma-prelim:BV_loc}(i), we have
	\[J_u^K(E_t\cap\Omega_i)\to\tJ_u^K(E_t),\quad J_u^K(E\cap\Omega_i)\to\tJ_u^K(E)\qquad\text{as }i\to\infty,\]
	as desired.
\end{proof}

We end this subsection with the following connectedness lemma.

\begin{lemma}\label{lemma-obs:connectedness}
	Suppose $E_0\Subset\Omega\Subset M$, and $E_0$ is connected, and $u$ is a solution of $\IVPOOinthm{\Omega;E_0}{\p\Omega}$. Then $\bar{E_t(u)}$ is connected for all $t>0$.
\end{lemma}
\begin{proof}
	Otherwise, there is a connected component of $\bar{E_t(u)}$, denoted by $S$, that does not intersect $\bar{E_0}$. Pick a domain $U$ with $S\Subset U\Subset M\setminus\bar{E_0}$ and $\bar{E_t(u)}\cap U=S$. Denote $\tilde u=u|_{\Omega\setminus\bar{E_0}}$ and $T=\inf_U(\tilde u)\geq0$. For each $s<t$, note that $E_s(\tilde u)\cap U\subset E_t(\tilde u)\cap U\Subset U$, hence we may compare $\tJ_{\tilde u}^U(E_s(\tilde u))\leq\tJ_{\tilde u}^U\big(E_s(\tilde u)\setminus U\big)=\tJ_{\tilde u}^U(\emptyset)=0$ and obtain
	\[\P{E_s(\tilde u);U}\leq\int_{E_s(\tilde u)\cap U}|\D u|\leq\int_T^s\P{E_{s'}(\tilde u);U}\,ds'.\]
	Then by Gronwall's inequality $\P{E_s(\tilde u);U}=0$ for all $s<t$, which is a contradiction.
\end{proof}

\subsection{Outward minimizing properties}\label{subsec:obs_min}

For a weak solution $u$ respecting an outer obstacle, the sub-level sets of $u$ satisfy certain outward minimizing properties subject to the outer obstacle $\p\Omega$. When $M=\Omega$, the conclusions here reduce to the interior case \cite[Property 1.4]{Huisken-Ilmanen_2001}.

\begin{defn}\label{def-obs:out_min_obs}
	Given a locally Lipschitz domain $\Omega$. We say that a set $E\subset\Omega$ is locally outward minimizing in $\bar\Omega$ (namely, outward minimizing subject to the outer obstacle $\p\Omega$), if for any competitor $F$ and domain $K$ satisfying $E\subset F\subset\Omega$ and $F\setminus E\Subset K\Subset M$, we have
	\begin{equation}\label{eq-obs:out_min_obs}
		\Ps{E;K}\leq\Ps{F;K}.
	\end{equation}
	We say that $E$ is strictly locally outward-minimizing, if \eqref{eq-obs:out_min_obs} is a strict inequality whenever $|F\setminus E|>0$. We will often drop the word ``locally'' for the ease of notations.
\end{defn}

\begin{defn}
	Given two sets $E,E'$ with $E\subset E'\subset\Omega$. We say that $E'$ is the strictly outward minimizing hull of $E$ (or minimizing hull, for brevity) in $\bar\Omega$, provided that:
	
	(i) $E'$ is strictly outward minimizing in $\bar\Omega$,
	
	(ii) if $E''$ is another strictly outward minimizing set in $\bar\Omega$, with $E\subset E''\subset\Omega$, then we have $E'\subset E''$ up to a set with zero measure.
\end{defn}

It can be verified that, if $E_1,E_2$ are both strictly outward minimizing, then $E_1\cap E_2$ is also strictly outward minimizing (we do not need $E_1\Delta E_2\Subset M$). Therefore, any set $E\subset\Omega$ has at most one minimizing hull, up to modifications of zero measure.

Similar to the interior case, we have the following statements.

\begin{theorem}[minimizing properties]\label{thm-obs:out_min} {\ }
	
	Let $u$ be a solution of $\IMCFOOinthm{\Omega}{\p\Omega}$. Then for all $t\in\RR$ it holds
	
	(i) $E_t$ is locally outward minimizing in $\bar\Omega$ whenever it is nonempty,
	
	(ii) $E_t^+$ is strictly locally outward minimizing in $\bar\Omega$ whenever it is nonempty,
	
	(iii) $E_t^+$ is the minimizing hull of $E_t$ in $\bar\Omega$, provided $E_t^+\setminus E_t\Subset M$.
\end{theorem}
\begin{proof}
	Items (i)(ii) follows from Definition \ref{def-obs:obstacle1} and Remark \ref{rmk-obs:definition}(iv), by arguing verbatim as in \cite[Property 1.4]{Huisken-Ilmanen_2001}. In the argument we need $\Omega$ to be connected, but this is our assumption throughout the paper. As item (iii) is stronger than the version in \cite{Huisken-Ilmanen_2001}, we present a proof. Suppose $E'\supset E_t$ is another strictly outward minimizing set in $\bar\Omega$. It follows that $E_t^+\cap E'$ is strictly outward minimizing in $\bar\Omega$. Choose $K$ such that $E_t^+\setminus E_t\Subset K\Subset M$. We may compare the energy $\tJ_u^K(E_t^+)\leq\tJ_u^K(E_t^+\cap E')$, see Remark \ref{rmk-obs:definition}(iv), and obtain $\Ps{E_t^+;K}\leq\Ps{E_t^+\cap E';K}$. This shows $|E_t^+\setminus E'|=0$ by the minimizing property of $E_t^+\cap E'$. Hence $E_t^+\subset E'$ up to a null set.
\end{proof}

For solutions of initial value problems, the initial time $t=0$ is not included. Hence we additionally state

\begin{theorem}\label{thm-obs:out_min_ivp}
	Let $M$, $\Omega$, $E_0$ be as in Definition \ref{def-obs:ivp}, and suppose that $u$ is a solution of $\IVPOOinthm{\Omega;E_0}{\p\Omega}$. Then for each $t\geq0$, $E_t^+$ is the minimizing hull of $E_t$ in $\bar\Omega$, provided that $E_t^+\setminus E_t\Subset M$.
\end{theorem}
\begin{proof}
	The case $t>0$ is already covered by Theorem \ref{thm-obs:out_min}, since $\p E_t\cap\p E_0=\emptyset$. Taking Theorem \ref{thm-obs:ivp2}(ii) and approximating $t\to0^+$, it follows that $\tJ_u^K(E_0^+)\leq\tJ_u^K(E)$ whenever $E_0\subset E\subset\Omega$ and $E\Delta E_0^+\Subset K\Subset M$. Then the same argument as in Theorem \ref{thm-obs:out_min} shows that $E_0^+$ is the strictly minimizing hull of $E_0$.
\end{proof}

\begin{cor}\label{cor-obs:sub_exp}
	Let $M,\Omega,E_0$ be as in Definition \ref{def-obs:ivp}, and suppose that $u$ is a solution of $\IVPOOinthm{\Omega;E_0}{\p\Omega}$. Then for any $0\leq s<t$ with $E_t\Subset M$, we have sub-exponential growth of area $\Ps{E_t}\leq e^{t-s}\Ps{E_s}$.
\end{cor}
\begin{proof}
	The fact $E_t\Subset M$ implies $E_0\Subset M$. In particular, it implies $E_0\Subset\Omega$ as we have assumed $\p E_0\cap\p\Omega=\emptyset$ in Definition \ref{def-obs:ivp}. Since $u\in\Lip_{\loc}(\Omega)$, for every $0<s<t$ we have $E_t\setminus E_s\Subset M\setminus\bar{E_0}$. Hence the energy comparison forces $\tJ_u^K(E_s)=\tJ_u^K(E_t)$ for any $K\Supset\Omega$. Arguing as in Remark \ref{rmk-obs:definition}(iii), this implies $\Ps{E_t}\leq e^{t-s}\Ps{E_s}$.
	
	Then we consider the case $s=0$. By Theorem \ref{thm-obs:ivp2} and the standard approximation argument, each $E_s^+$ ($s\geq0$) has the same minimizing property as described in Theorem \ref{thm-obs:ivp2}(ii). By mutual energy comparison, this implies $\Ps{E_s^+}\leq e^s\Ps{E_0^+}$ by arguing in the same manner. For almost every $s\leq t$ we have $E_s=E_s^+$, hence we obtain $\Ps{E_t}\leq e^t\Ps{E_0^+}$. Finally, by Theorem \ref{thm-obs:out_min_ivp} we have $\Ps{E_0^+}\leq\Ps{E_0}$.
\end{proof}

Note: the inequality $\Ps{E_t}\leq e^t\Ps{E_0}$ may still be strict even if $E_0$ is outward minimizing in $\bar\Omega$; see Remark \ref{rmk-obs:definition}(iii).

\begin{example}
	Consider $M=\RR^2$ with the metric $g=dr^2+f(r)^2d\th^2$, where $f(r)=\sin r$ inside $[0,3\pi/4]$, and $f'(r)<0$ inside $(3\pi/4,\pi)$, and $f(r)\equiv1/2$ in $[\pi,\infty)$. Let $B$ be a small geodesic ball centered at $r=\pi/2$, $\th=0$, with $|\p B|=2\pi\epsilon$. Consider the domain $\Omega=\{r<5\}\setminus B$ and the initial value $E_0=\{r<\pi/6\}$. Set $T=\log(1+2\epsilon)$, determined by $e^T|\p E_0|=|\p\Omega|$. Let $u$ be the function defined such that its level sets evolve by $1/H$ in $[0,T)$ and then jump over the whole $\Omega$ at time $T$. Such solution $u$ can be verified to respect the obstacle $\p\Omega$. Thus, the ball $B$ delays the jumping time in the obstacle setting.
\end{example}

\subsection{Formulation using 1-Dirichlet energy}\label{subsec:obs_diri}

Similar to the interior case, there is another variational principle in terms of a 1-Dirichlet type energy. The presence of obstacle results in a boundary term:
\begin{equation}\label{eq-obs:energy_v}
	\tJ_u(v)=\int_\Omega\big(|\D v|+v|\D u|\big)-\int_{\p^*\Omega}v^\p\,d\H^{n-1}.
\end{equation}
where we use $v^\p$ to denote the BV boundary trace of $v$. The choice of sign in the boundary term becomes manifest in the proof of the equivalence theorem (and heuristically suggests that $u$ wants to be maximal along $\p\Omega$).

We set up a definition based on \eqref{eq-obs:energy_v} and show its equivalence with Definition \ref{def-obs:obstacle1}. The proofs are mostly technical and are not used elsewhere.

To prove the equivalence between \eqref{eq-obs:energy_v} and \eqref{eq-obs:energy_E}, we need $|u|\,|\D u|$ to be locally integrable, so that $\tJ_u(u)$ is defined. This holds for all subsolutions (see Lemma \ref{lemma-obs:integrable} below) but not necessarily for supersolutions (see Remark \ref{rmk-obs:abnormal_supersol}). Therefore, we will impose the integrability condition \eqref{eq-obs:further_assump} for supersolutions.

\begin{lemma}[higher integrability]\label{lemma-obs:integrable}
	Suppose $\Omega\subset M$ is locally Lipschitz, and $u\in\Lip_{\loc}(\Omega)$ is a subsolution of $\IMCF{\Omega}$. Then
	\begin{equation}\label{eq-obs:integrability}
		\int_{\Omega\cap K}|u|^p\,|\D u|<\infty\qquad\forall\,K\Subset M,\ p\geq1.
	\end{equation}
\end{lemma}
\begin{proof}
	For each $K\Subset M$, one can find another precompact Lipschitz domain $\Omega'\subset\Omega$, such that $\Omega'\cap K=\Omega\cap K$. Therefore, it is sufficient to prove the case $\Omega\Subset K\Subset M$. We assume this case for the rest of the proof. By \cite[Theorem 7.4]{Luukkainen-Vaisala_1977}, $\p\Omega'$ admits a Lipschitz collar neighborhood. This implies that there exists a neighborhood $N\subset\bar{\Omega}$ of $\p\Omega$ and a Lipschitz retraction map $\Phi: N\to\p\Omega$.
	
	Since $u\in\Lip_{\loc}(\Omega)$, we have $\sup_{\Omega\setminus N}|u|<T$ for some $T$. By Theorem \ref{thm-obs:auto_subsol}, $u$ is a subsolution of $\IMCFOO{\Omega}{\p\Omega}$. Comparing $\tJ_u^K(E_t(u))\leq\tJ_u^K(\Omega)$, by the coarea formula we obtain
	\begin{equation}\label{eq-obs:aux3}
		\Ps{E_t(u)}+\int_t^\infty\P{E_s(u);\Omega}\,ds\leq\Ps{\Omega}.
	\end{equation}
	For $t\geq T$ we have $E_t(u)\Supset\Omega\setminus N$, thus $\Phi$ maps $\p^*E_t(u)$ surjectively to $\p^*\Omega$ up to a $\H^{n-1}$-null set. By the area formula we have
	\begin{equation}\label{eq-obs:area_formula}
		\Ps{\Omega}\leq\H^{n-1}\big(\p^*E_t(u)\cap\p^*\Omega\big)+\Lip(\Phi)^{n-1}\P{E_t(u);\Omega}.
	\end{equation}
	Cancelling the common portion of perimeters in \eqref{eq-obs:aux3} and \eqref{eq-obs:area_formula}, we obtain
	\[\int_t^\infty\Ps{E_s(u);\Omega}\,ds\leq\big(\!\Lip(\Phi)^{n-1}-1\big)\P{E_t(u);\Omega},\qquad\forall\,t\geq T.\]
	Hence $\P{E_t(u);\Omega'}$ exponentially decays, and $\int_{u\geq T}u^p|\D u|<\infty$ by the coarea formula.
	
	For $t\leq-T$, we compare $\tJ_{u}(E_s(u))\leq\tJ_{u}(E_t(u))$ and take $s\to-\infty$, to find
	\[\int_{-\infty}^t\P{E_s(u);\Omega}\,ds\leq\Ps{E_t(u)}.\]
	Since $E_t(u)\Subset N$, projecting via $\Phi$ we find that
	\[\P{E_t(u)}\leq\big(\!\Lip(\Phi)^{n-1}+1\big)\P{E_t(u);\Omega}.\]
	The combined inequality implies that $\Ps{E_t(u);\Omega}$ exponentially decays when $t\to-\infty$. By the coarea formula, we have $\int_{u\leq-T}(-u)^p|\D u|<\infty$. Finally, we have $\int_{|u|\leq T}|u|^p|\D u|\leq 2T^p\Ps{\Omega}$ by Lemma \ref{lemma-prelim:BV_loc}(i), and thus \eqref{eq-obs:integrability} is proved.
\end{proof}

\begin{defn}\label{def-obs:energy_u}
	Let $\Omega\subset M$ be a locally Lipschitz domain. For a function $u\in\Lip_{\loc}(\Omega)$, a domain $K\Subset M$ and another function $v\in\Lip_{\loc}(\Omega)$, set the energy
	\[\tJ_u^K(v)=\left\{\begin{aligned}
		& \int_{\Omega\cap K}\big(|\D v|+v|\D u|\big)-\int_{\p^*\Omega\cap K}v^\p\,d\H^{n-1}
		\quad\text{\rm{(if $v\in BV(\Omega\cap K)$}}, \\
		& \hspace{240pt}v|\D u|\in L^1(\Omega\cap K)), \\
		& \text{\rm{undefined\quad(otherwise).}}
	\end{aligned}\right.\]
	We say that $u$ locally minimizes $\tJ_u$ (resp. minimizes from below, from above), if for all $v\in\Lip_{\loc}(\Omega)$ (resp. for all $v\leq u$, $v\geq u$) and any domain $K$ satisfying $\{u\ne v\}\Subset K\Subset M$, we have
	\begin{equation}\label{eq-obs:energy_comp_v}
		\tJ_u^K(u)\leq\tJ_u^K(v)
	\end{equation}
	whenever both sides are defined.
\end{defn}

\begin{theorem}[outer obstacle II]\label{thm-obs:obstacle2} {\ }
	
	Let $\Omega\subset M$ be a locally Lipschitz domain, and $u\in\Lip_{\loc}(\Omega)$. Then $u$ is a solution (resp. subsolution) of $\IMCFOOinthm{\Omega}{\p\Omega}$ if and only if $u$ locally minimizes $\tJ_u$ (resp. locally minimizes $\tJ_u$ from below). If we further assume
	\begin{equation}\label{eq-obs:further_assump}
		u\in\BV_{\loc}(\bar\Omega),\qquad |u|\,|\D u|\in L^1_{\loc}(\bar\Omega),
	\end{equation}
	then $u$ is a supersolution if and only if $u$ locally minimizes $\tJ_u$ from above.
\end{theorem}

\begin{proof}
	The proof is a direct generalization of the interior case \cite[Lemma 1.1]{Huisken-Ilmanen_2001}.
	
	(1) Suppose $u$ is a weak (sub-, super-) solution of $\IMCFOO{\Omega}{\p\Omega}$. By Lemma \ref{lemma-prelim:BV_loc} and \ref{lemma-obs:integrable} or by our assumption \eqref{eq-obs:further_assump}, we have $u\in\BV_{\loc}(\bar\Omega)$ and $|u|\,|\D u|\in L^1_{\loc}(\bar\Omega)$ in either case. Let $v\in\Lip_{\loc}(\bar\Omega)$ be a competitor (with additionally $v\leq u$ or $v\geq u$ for the case of subsolution or supersolution), with $\{u\ne v\}\Subset K\Subset M$ and $\tJ_u^K(v)$ is defined. For $b\in\RR$, we set $v_b=\min\{v,b\}$ and truncate the integrals (we suppress the boundary measure $d\H^{n-1}$):
	\begin{equation}\label{eq-obs:aux4}
		\int_{\Omega\cap K}|\D v|=\int_{\{v>b\}\cap K}|\D v|+\int_{\Omega\cap K}|\D v_b|,
	\end{equation}
	\begin{equation}\label{eq-obs:aux5}
		\int_{\p^*\Omega\cap K}v^\p=\int_{\p^*\Omega\cap\{v^\p>b\}\cap K}(v-b)^\p+b\H^{n-1}\big(\p^*\Omega\cap K\big)-\int_{\p^*\Omega\cap K}(b-v_b)^\p,
	\end{equation}
	\begin{equation}\label{eq-obs:aux6}
		\int_{\Omega\cap K}v|\D u|=\int_{\{v>b\}\cap K}(v-b)|\D u|+b\int_{\Omega\cap K}|\D u|-\int_{\Omega\cap K}(b-v_b)|\D u|.
	\end{equation}
	The first term on the right hand sides converge to 0 as $b\to+\infty$, by our integrability assumption for $v$. Then by Cavalieri's formula and the coarea formula for BV functions (see \cite[Section 5.4, 5.5]{Evans-Gariepy}), we have
	\[\int_{\Omega\cap K}\Big(|\D v_b|-(b-v_b)|\D u|\Big)+\int_{\p^*\Omega\cap K}(b-v_b)^\p=\int_{-\infty}^b\tJ_u^K(E_t(v))\,dt.\]
	As a result,
	\begin{equation}\label{eq-obs:aux7}
		\tJ_u^K(v) = o(1)+\int_{-\infty}^b\tJ_u^K(E_t(v))\,dt+b\int_{\Omega\cap K}|\D u|-b\H^{n-1}\big(\p^*\Omega\cap K\big).
	\end{equation}
	Decomposing $\tJ_u^K(u)$ in the same manner, applying Definition \ref{def-obs:obstacle1}, cancelling the common terms in \eqref{eq-obs:aux7} and finally taking $b\to\infty$, we find that $\tJ_u^K(u)\leq\tJ_u^K(v)$.
	
	(2) Suppose $u$ satisfies \eqref{eq-obs:further_assump} and locally minimizes $\tJ_u$ from above; let us verify the supersolution condition: $\tJ_u^K(E_t(u))\leq\tJ_u^K(E)$ for all $E\subset E_t$ with $E_t\setminus E\Subset K\Subset M$. The proof here differs only slightly from \cite[Lemma 1.1]{Huisken-Ilmanen_2001}. By selecting a $\tJ_u^K$--minimizer among all the sets $F$ with $E\subset F\subset E_t$, we may assume that $\tJ_u^K(E)\leq\tJ_u^K(E')$ for all $E\subset E'\subset E_t$. Next, we define the function
	\[v=\left\{\begin{aligned}
		t\qquad\text{on $E_t\setminus E$}, \\
		u\qquad\text{elsewhere}.
	\end{aligned}\right.\]
	It is easy to see that $u\leq v\leq\max\{u,t\}$ and $\{u\ne v\}\Subset K$. Writing $v=u+(t-u)\chi_{E_t\setminus E}$, we find $v\in\BV_{\loc}(\bar\Omega)$ and $|v|\,|\D u|\in L^1_{\loc}(\bar\Omega)$ by \eqref{eq-obs:further_assump}. Note that $E_s(v)=E_s(u)\cap E$ for all $s\leq t$, and therefore we have 
	\begin{equation}\label{eq-obs:rev_comp}
		\tJ_u^K(E_s(v))\leq\tJ_u^K(E_s(u))+\tJ_u^K(E)-\tJ_u^K(E_s(v)\cup E)\leq\tJ_u^K(E_s(u))
	\end{equation}
	By the standard approximation results (see \cite[Theorem 3.9 and 3.88]{Ambrosio-Fusco-Pallara}), there exist nonnegative functions $w_i\in C^\infty(\Omega)\cap\BV(\Omega)$ supported in $K$, such that
	\[w_i\xrightarrow{L^1}(v-u),\qquad \|Dw_i\|(\Omega)\to\|D(v-u)\|(\Omega),\qquad
	\int_{\p^*\Omega}w_i^\p\to\int_{\p^*\Omega}(v^\p-u^\p).\]
	Since $|\D u|\in L^\infty_{\loc}(\Omega)$, with a slight modification of \cite[Theorem 3.9]{Ambrosio-Fusco-Pallara} we can achieve
	\[\int_{\Omega} w_i|\D u|\to\int_\Omega(v-u)|\D u|.\]
	Using the fact $\tJ_u^K(u)\leq\tJ_u^K(u+w_i)$ from our hypotheses, then taking $i\to\infty$, we find that $\tJ_u^K(u)\leq\tJ_u^K(v)$ in the BV sense. Truncating the integrals at an arbitrary $b>t$ as in \eqref{eq-obs:aux4}\,$\sim$\,\eqref{eq-obs:aux6}, and noting that $u=v\in\Lip_{\loc}(\{u>t\})$, we find that
	\begin{equation}\label{eq-obs:rev_comp_2}
		\int_{-\infty}^t\tJ_u^K(E_s(u))\,ds\leq\int_{-\infty}^t\tJ_u^K(E_s(v))\,ds.
	\end{equation}
	The combination \eqref{eq-obs:rev_comp} \eqref{eq-obs:rev_comp_2} implies $\tJ_u^K(E_s(v))=\tJ_u^K(E_s(u))$ for a.e. $s<t$. This implies $\tJ_u^K(E_s(u)\cup E)\leq\tJ_u^K(E)$, and taking $s\nearrow t$ implies $\tJ_u^K(E_t(u))\leq\tJ_u^K(E)$.
	
	This argument verifies the supersolution case. For subsolutions we argue in the same manner, where the integrability condition for $u$ comes from Lemma \ref{lemma-prelim:BV_loc} and \ref{lemma-obs:integrable}. The case of solution follows by combining the supersolution and subsolution cases.
\end{proof}

\subsection{Boundary orthogonality of calibration}\label{subsec:obs_tangent}

Recall from Subsection \ref{subsec:calibrated} the notion of a calibrated solution: there exists a measurable vector field $\nu$ satisfying $|\nu|\leq1$ and $\nu\cdot\D u=|\D u|$ almost everywhere, and $\div(\nu)=|\D u|$ weakly in $\Omega$. Moreover, recall by Lemma \ref{lemma-prelim:BV_loc}(i) that $|\D u|\in L^1_{\loc}(\bar\Omega)$. Here we show that if $\metric{\nu}{\nu_\Omega}=1$ on $\p\Omega$ in the trace sense, then $u$ solves $\IMCFOO{\Omega}{\p\Omega}$.

Let us first introduce the notion of boundary trace. Suppose $\Omega$ is locally Lipschitz. Define the space
\[X=\big\{\nu\in L^\infty(T\Omega): \div(\nu)\in L^1_{\loc}(\bar\Omega)\big\}.\]
There is a well-defined operator $\nu\mapsto[\nu\cdot\nu_\Omega]$ from $X$ to $L^\infty(\p\Omega,\H^{n-1})$, called the \textit{normal trace}, with the following properties:

(i) for all $\nu\in X$ we have $\|[\nu\cdot\nu_\Omega]\|_\infty\leq\|\nu\|_\infty$;

(ii) for $\H^{n-1}$-a.e. $x\in\p\Omega$ we have
\begin{equation}\label{eq-obs:normal_trace_avg}
	[\nu\cdot\nu_\Omega](x)=-\frac n{|B^{n-1}|}\lim_{r\to0}\frac1{r^n}\int_{\Omega\cap B(x,r)}\metric{\nu}{\D d_x},
\end{equation}
where $d_x=d(\cdot,x)$ is the distance function from $x$. In particular, if $\Omega$ is a $C^1$ domain and $\nu\in C^1(\bar\Omega)$, then $[\nu\cdot\nu_\Omega]=\nu\cdot\nu_\Omega$ in the classical sense;

(iii) the following divergence formula holds: for $\varphi\in\BV(\Omega)\cap\Lip_{\loc}(\Omega)$ with $\spt(\varphi)\Subset M$ and $\varphi\div(\nu)\in L^1(\Omega)$, we have
\begin{equation}\label{eq-obs:gauss_green}
	\int_\Omega\varphi\div(\nu)+\metric{\nu}{\D\varphi}=\int_{\p\Omega}[\nu\cdot\nu_\Omega]\varphi^\p\,d\H^{n-1}.
\end{equation}

Properties (i)(iii) can be found (as the Riemannian and local version) in \cite[Section 1]{Anzellotti_1983}. Property (ii) is proved in \cite[Theorem 4.4]{Silhavy_2005}. The formula \eqref{eq-obs:gauss_green} is proved in \cite[Theorem 1.9]{Anzellotti_1983} assuming $\varphi\in \BV(\Omega)\cap\Lip_{\loc}(\Omega)\cap L^\infty$, but what is stated here follows by a truncation argument. The divergence formula in fact holds in a much broader sense: it suffices to assume that $\div(\nu)$ is a Radon measure and $\varphi\in\BV(\Omega)$ has bounded support. In this case, the second term in \eqref{eq-obs:gauss_green} needs to be replaced with an abstract pairing. We refer the reader to \cite{Anzellotti_1983, Chen-Frid_1999} for more details; the current formulation is enough for our purpose.

\begin{lemma}\label{lemma-obs:bd_orthogonal}
	Suppose $\Omega\subset M$ is locally Lipschitz, and $u$ solves $\IMCF{\Omega}$ and is calibrated by $\nu$. If $[\nu\cdot\nu_\Omega]=1$ $\H^{n-1}$-a.e. on $\p\Omega$, then $u$ solves $\IMCFOOinthm{\Omega}{\p\Omega}$.
\end{lemma}

Combining Lemma \ref{lemma-obs:bd_orthogonal}, \eqref{eq-obs:normal_trace_avg}, and an elementary computation, we obtain the following convenient criterion:

\begin{cor}\label{cor-obs:bd_orthogonal_2}
	Let $\Omega,u,\nu$ be as in Lemma \ref{lemma-obs:bd_orthogonal}. If for $\H^{n-1}$-almost every $x\in\p^*\Omega$, and in some geodesic normal coordinate near $x$, we have
	\begin{equation}\label{eq-obs:ae_orthogonal}
		\lim_{r\to0}\esssup\Big\{\big|\nu(y)-\nu_\Omega(x)\big|: y\in\Omega\cap B(x,r)\Big\}=0,
	\end{equation}
	then $u$ solves $\IMCFOOinthm{\Omega}{\p\Omega}$.
\end{cor}

\begin{proof}[Proof of Lemma \ref{lemma-obs:bd_orthogonal}]
	Since $\div(\nu)=|\D u|$, and $u\in\BV_{\loc}(\bar\Omega)$ due to Lemma \ref{lemma-prelim:BV_loc}, we indeed have $\nu\in X$. Let $v\in\Lip_{\loc}(\Omega)$ be a competitor, such that $\{u\ne v\}\Subset K\Subset M$ and $\tJ_u^K(v)$ is defined. Applying \eqref{eq-obs:gauss_green} with the function $\varphi=v-u$, we obtain
	\[\begin{aligned}
		\int_{\p\Omega}(v^\p-u^\p)\,d\H^{n-1} &= \int_{\p\Omega}[\nu\cdot\nu_\Omega]\varphi^\p\,d\H^{n-1}
		= \int_\Omega(v-u)|\D u|+\int_\Omega\nu\cdot(\D v-\D u) \\
		&\leq \int_\Omega(v-u)|\D u|+\int_\Omega\big(|\D v|-|\D u|\big),
	\end{aligned}\]
	since $\nu\cdot\D u=|\D u|$ a.e.. This exactly implies $\tJ_u^K(u)\leq\tJ_u^K(v)$.
\end{proof}

\subsection{A weak maximum principle}\label{subsec:obs_max_prin}

We extend the interior maximum principle (Theorem \ref{thm-prelim:max_principle}) to a version with the presence of obstacle. In complement with Theorem \ref{thm-obs:auto_subsol}, this has the following consequence: for an initial value problem, the weak solution respecting the boundary obstacle (if it exists) is maximal among all interior weak solutions of $\IVP{\Omega;E_0}$. Some technical explanations are included in Remark \ref{rmk-obs:max_principle}.

\begin{theorem}[weak maximum principle]\label{thm-obs:max_principle} {\ }
	
	Given $\Omega\subset M$ a locally Lipschitz domain. Let $u,v\in\Lip_{\loc}(\Omega)$ be respectively a solution and subsolution of $\IMCFOOinthm{\Omega}{\p\Omega}$. If $\{u<v\}\Subset M$, and
	\begin{equation}\label{eq-obs:bounded_below}
		\inf_{\Omega\cap K}u>-\infty\qquad\text{for some $K$ with }\{u<v\}\Subset K\Subset M,
	\end{equation}
	then $u\geq v$ in $\Omega$.
\end{theorem}
\begin{proof}
	We may add a common constant to $u, v$ and assume that $u\geq0$ in $K$. Replacing $v$ with $\min\{v,T\}$ (which is still a weak subsolution by Definition \ref{def-obs:obstacle1}) and taking $T\to+\infty$, we may assume that $v\leq T$ for some $T>0$. Thus we are reduced to the case $u\geq0$, $v\leq T$. For a constant $\epsilon>0$ we set $u_\epsilon=\frac{u}{1-\epsilon}$. Then it is sufficient to show that $u_\epsilon\geq v$. Suppose that this fails to hold. Then we may increase $u$ by an appropriate constant, and assume that $u_\epsilon>v-\epsilon$ but $\{u_\epsilon<v\}\ne\emptyset$.
	
	From Theorem \ref{thm-obs:obstacle2}, it is not hard to see that $u_\epsilon$ satisfies a strict minimizing condition
	\[\int\big(|\D u_\epsilon|-|\D w|\big)-\int_{\p^*\Omega}(u_\epsilon^\p-w^\p)\leq(1-\epsilon)\int(w-u_\epsilon)|\D u_\epsilon|\]
	for all competitors $w\geq u_\epsilon$ such that $\{w\ne u_\epsilon\}\Subset M$ and $\tJ_u^K(w)$ is defined. Setting $w=\max\{u_\epsilon,v\}$ (which is a valid competitor), we obtain
	\begin{equation}\label{eq-obs:comp1}
		\int_{\{v>u_\epsilon\}}\big(|\D u_\epsilon|-|\D v|\big)-\int_{\p^*\Omega\cap\{v^\p>u_\epsilon^\p\}}(u_\epsilon^\p-v^\p)\leq(1-\epsilon)\int_{\{v>u_\epsilon\}}(v-u_\epsilon)|\D u_\epsilon|.
	\end{equation}
	Testing the subsolution property of $v$ with the competitor $\min\{u_\epsilon,v\}$, we find that
	\begin{equation}\label{eq-obs:comp2}
		\int_{\{u_\epsilon<v\}}\big(|\D v|-|\D u_\epsilon|\big)-\int_{\p^*\Omega\cap\{u_\epsilon^\p<v^\p\}}(v^\p-u_\epsilon^\p)\leq\int_{\{u_\epsilon<v\}}(u_\epsilon-v)|\D v|.
	\end{equation}
	Adding \eqref{eq-obs:comp1} \eqref{eq-obs:comp2} and canceling the common terms, we obtain
	\begin{equation}\label{eq-obs:strict_comp}
		\int_{\{v>u_\epsilon\}}(v-u_\epsilon)|\D v|\leq(1-\epsilon)\int_{\{v>u_\epsilon\}}(v-u_\epsilon)|\D u_\epsilon|.
	\end{equation}
	
	Next, for $s\geq0$ we compare $\tJ_{u_\epsilon}^K(u_\epsilon)\leq\tJ_{u_\epsilon}^K(\max\{u_\epsilon,v-s\})$ to obtain
	\[\begin{aligned}
		&\int_{\{v-s>u_\epsilon\}}\big(|\D u_\epsilon|-|\D v|\big)-\int_{\p^*\Omega\cap\{v^\p-s>u_\epsilon^\p\}}(u_\epsilon^\p-v^\p+s) \\
		&\hspace{240pt}\leq \int_{\{v-s>u_\epsilon\}}(v-s-u_\epsilon)|\D u_\epsilon|.
	\end{aligned}\]
	The boundary term has the favorable sign and thus can be discarded. Integrating this inequality over $0\leq s<\infty$ (note that the integrand is nonzero only within a bounded interval of $s$), by Fubini's theorem we obtain
	\begin{equation}\label{eq-obs:strict_comp_2}
		0\leq\int_{\{v>u_\epsilon\}}(v-u_\epsilon)\big(|\D v|-|\D u_\epsilon|\big)+\frac12(v-u_\epsilon)^2|\D u_\epsilon|.
	\end{equation}
	Combining \eqref{eq-obs:strict_comp} \eqref{eq-obs:strict_comp_2} we have
	\begin{equation}\label{eq-obs:strict_comp_3}
		\epsilon\int_{\{v>u_\epsilon\}}(v-u_\epsilon)|\D u_\epsilon|\leq\frac12\int_{\{v>u_\epsilon\}}(v-u_\epsilon)^2|\D u_\epsilon|.
	\end{equation}
	Recall our reduction hypothesis at the beginning, that $u_\epsilon>v-\epsilon$ and $\{v>u-\epsilon\}\ne\emptyset$. Then \eqref{eq-obs:strict_comp_3} implies that $u_\epsilon$, hence $u$ as well, is constant on $\{v>u_\epsilon\}$. Then \eqref{eq-obs:strict_comp} implies in turn that $v$ is constant on $\{v>u_\epsilon\}$. Then the only possibility is that $\{v>u_\epsilon\}$ is a precompact connected component of $\Omega$. However, this contradicts the supersolution property of $u$, see Remark \ref{rmk-obs:definition}(v). This completes the proof.
\end{proof}

\begin{remark}\label{rmk-obs:max_principle} {\ }
	
	(i) The lower bound \eqref{eq-obs:bounded_below} is not removable due to Example \ref{ex-ex:bounded_domain}. Indeed, the function $u$ therein is a solution of $\IMCFOO{\Omega}{\p\Omega}$ for some $\Omega\Subset\RR^2$. Setting $v=u+1$, we have $\{u<v\}\Subset\RR^2$ but obviously $u\ngeqslant v$. In the above proof, the condition \eqref{eq-obs:bounded_below} is used only in the initial reduction, and is not involved in deriving \eqref{eq-obs:strict_comp_3}.
	
	(ii) The proof here is a minor adaptation of \cite[Theorem 2.2]{Huisken-Ilmanen_2001}. We assumed $u$ to be an exact solution, while the proof works when $u$ is a supersolution with enough integrability $u\in BV(\Omega\cap K)$, $u|\D u|\in L^1(\Omega\cap K)$. These assumptions can actually be removed, by invoking the original Definition \ref{def-obs:obstacle1} (hence avoiding Theorem \ref{thm-obs:obstacle2}). By comparing the energies $\tJ_u(E_t(u_\epsilon))\leq\tJ_u(E_t(u_\epsilon)\cap E_t(v))$ and $\tJ_v(E_t(v))\leq\tJ_v(E_t(v)\cup E_t(u_\epsilon))$, and integrating over $t$, we can recover the intermediate inequality \eqref{eq-obs:strict_comp}. The other inequality \eqref{eq-obs:strict_comp_2} can similarly be obtained by comparing among various sub-level sets. For technical simplicity, we do not pursue proving the case with the widest possible generality.
\end{remark}

\begin{cor}[maximality]\label{cor-obs:maximality} {\ }
	
	Given a Lipschitz domain $\Omega\Subset M$ and $C^{1,1}$ domains $E_0\subset E'_0\Subset\Omega$. Suppose
	
	(1) $u\in\Lip_{\loc}(\Omega)$ is a solution of $\IVPOOinthm{\Omega;E_0}{\p\Omega}$, 
	
	(2) $v\in\Lip_{\loc}(\Omega)$ is a subsolution of $\IVP{\Omega;E'_0}$.
	
	\noindent Then $u\geq v$ in $\Omega\setminus E'_0$. In particular, the weak solution of $\IVPOOinthm{\Omega;E_0}{\p\Omega}$ is unique up to equivalence (if it exists).
\end{cor}
\begin{proof}
	For any $\epsilon>0$, the functions $u+\epsilon$ and $v$ are respectively a solution (by Definition \ref{def-obs:ivp} and restriction) and subsolution (by Theorem \ref{thm-obs:auto_subsol}) of $\IMCFOO{\Omega\!\setminus\!\bar{E'_0}}{\p\Omega}$. Moreover, we have $\{v>u+\epsilon\}\Subset M\setminus\bar{E'_0}$. Then $u\geq v$ in $\Omega\setminus E'_0$, by Theorem \ref{thm-obs:max_principle} and by taking $\epsilon\to0$. Finally, the uniqueness follows from maximality.
\end{proof}

\section{Liouville theorems on the half space}\label{sec:halfplane}

The aim of this section is to establish Theorem \ref{thm-halfplane:liouville} and \ref{thm-halfplane:approx}. They are Liouville theorems in the half space, for weak solutions that respect the boundary obstacle and respectively a ``soft obstacle'' respectively. See Theorem \ref{thm-halfplane:approx} for the precise description of the latter. These results enter the existence theorem in showing that there are only trivial limits in the blow-up procedures; see Section 6 for further details. The Liouville theorems are proved by combining barrier arguments (involving the interior maximum principle) and blow-up\,/\,blow-down arguments (which involves analyzing the perimeter-minimizing effects). The techniques presented here may also be of their own interests.

We adopt the following notations: for a point $x=(x_1,x_2,\cdots,x_n)\in\RR^n$, we denote $x'=(x_1,\cdots,x_{n-1})$ and write $x=(x',x_n)$. Moreover, set $e_n=(0,\cdots,0,1)$.

\begin{theorem}\label{thm-halfplane:liouville}
	Let $\Omega=\{x_n<0\}\subset\RR^n$, and suppose that $u\in\Lip_{\loc}(\Omega)$ is a solution of $\IMCFOOinthm{\Omega}{\p\Omega}$. Moreover, suppose
	\begin{equation}\label{eq-halfplane:gradient}
		\inf_\Omega(u)>-\infty\qquad\text{and}\qquad\ |\D u(x)|\leq\frac C{|x_n|}\ \ \text{for a.e. $x\in\Omega$}.
	\end{equation}
	Then $u$ must be constant.
\end{theorem}

The condition $\inf_\Omega(u)>-\infty$ cannot be removed, due to the Poisson kernel type Example \ref{ex-ex:nephroid}. In the proof we need the following direct extension of Lemma \ref{lemma-prelim:excess_ineq}.

\begin{lemma}\label{lemma-halfplane:excess_ineq_2}
	Suppose $\Omega$ is a locally Lipschitz domain, and $u\in\Lip_{\loc}(\Omega)$ is a supersolution of $\IMCFOOinthm{\Omega}{\p\Omega}$, such that $\inf_\Omega(u)=t_0>-\infty$. Suppose $F$ has finite perimeter, and $K$ is a domain, such that $F\Subset K\Subset M$. Then we have
	\begin{equation}\label{eq-halfplane:excess_ineq}
		\Ps{E_t;K}\leq\Ps{E_t\setminus F;K}+\int_{t_0}^t e^{t-s}\Ps{F;E_s}\,ds.
	\end{equation}
\end{lemma}
\begin{proof}
	By Lemma \ref{lemma-prelim:BV_loc}(ii), both $\tJ_u^K(E_t)$ and $\tJ_u^K(E_t\setminus F)$ are finite. From the fact $\tJ_u^K(E_t)\leq\tJ_u^K(E_t\setminus F)$ and the coarea formula, we see that
	\[\begin{aligned}
		\Ps{E_t;K} &\leq \Ps{E_t\setminus F;K}+\int_{t_0}^t\H^{n-1}\big(\p^*E_s\cap E_t\cap F\cap\Omega\big)\,ds \\
		&\leq \Ps{E_t\setminus F;K}+\int_{t_0}^t\H^{n-1}\big(\p^*E_s\cap F\big)\,ds.
	\end{aligned}\]
	The remaining proof is the same as in Lemma \ref{lemma-prelim:excess_ineq}.
\end{proof}

\begin{proof}[Proof of Theorem \ref{thm-halfplane:liouville}] {\ }
	
	Shifting $u$ by a constant, we may assume $\inf_\Omega(u)=0$. Thus we aim to show $u\equiv0$. We separate the proof into two cases:
	
	\vspace{3pt}
	
	\hypertarget{thm4.1case1}{\textbf{Case (i)}}: suppose there exist sequences $t_i\searrow0$ and $\th_i\searrow0$, such that
	\begin{equation}\label{eq-halfplane:case1}
		\essinf_{\p^*E_{t_i}}\metric{\nu_{E_{t_i}}}{e_n}\geq\cos\th_i.
	\end{equation}
	For this case we use a barrier argument. Suppose that $u$ is not identically zero, so there exists $x_0\in\Omega$ with $u(x_0)>0$. We may remove finitely many terms in the sequence and assume $u(x_0)>t_i$ for all $i$. By continuity and $\inf_\Omega(u)=0$, for each $i$ there exists a point $y_i=y'_i+y_{i,n}e_n$ ($y_{i,n}<0$) that lies on $\p^*E_{t_i}$. By Lemma \ref{lemma-gmt:contain_cone}, \eqref{eq-halfplane:case1}, and Lemma \ref{lemma-prelim:density_one}, we have
	\begin{equation}\label{eq-halfplane:contain_cone}
		E_{t_i}\supset\Big\{(x',x_n): x_n<y_{i,n}-\big|x'-y'_i\big|\tan\th_i\Big\}.
	\end{equation}
	For a constant $N>1$, consider the family of sets
	\[F_t=B\big(y'_i+Ny_{i,n}e_n,e^\frac t{n-1}(N-1)|y_{i,n}|\cos\th_i\big),\]
	whose boundary smoothly solves the IMCF. By the interior maximum principle and the fact that $F_0\subset E_{t_i}$, we have
	\[F_t\subset E_{t+t_i}\qquad\text{whenever $F_t\Subset\Omega$.}\]
	The condition $F_t\Subset\Omega$ holds precisely when $t<T:=(n-1)\log\big(\frac{N}{(N-1)\cos\th_i}\big)$. Taking $t\nearrow T$, we obtain
	\[E_{t_i+T}\supset B\big(y'_i+Ny_{i,n}e_n,N|y_{i,n}|\big).\]
	
	Finally, taking $N\to\infty$ and then $i\to\infty$, we find that $E_t\supset\{x_n<0\}$ for any $t>0$. This implies $u\equiv0$.
	
	\vspace{3pt}
	
	\textbf{Case (ii)}: suppose the hypothesis of case (i) fails. This implies that we can find a constant $\th_0<1$, a sequence of times $t_i\searrow0$, and a sequence of points $y_i=(y'_i,y_{i,n})\in\p^*E_{t_i}$, with $\metric{\nu_{E_{t_i}}(y_i)}{e_n}\leq\cos\th_0$. The points $y_i$ are necessarily contained in the interior of $\Omega$ (otherwise $\nu_{E_{t_i}}(y_i)=\nu_\Omega=e_n$). Consider the rescaled sets
	\[F_i=\frac1{|y_{i,n}|}(E_{t_i}-y'_i)\subset\Omega,\]
	where clearly $-e_n\in\p^*F_i$ and $\metric{\nu_{F_i}(-e_n)}{e_n}\leq\cos\th_0$. We investigate the minimizing properties of $F_i$: by Theorem \ref{thm-obs:out_min}, each $F_i$ is locally outward minimizing in $\bar\Omega$, hence locally outward minimizing in $\RR^n$, by direct verification. On the other hand, by Lemma \ref{lemma-halfplane:excess_ineq_2} and the scale invariance of \eqref{eq-halfplane:excess_ineq}, we obtain the inward-minimizing property
	\[\Ps{F_i;K}\leq\Ps{F_i\setminus G;K}+(e^{t_i}-1)\Ps{G}\qquad\forall\,G\Subset K\Subset\RR^n.\]
	Taking $i\to\infty$, a subsequence of $F_i$ converges locally to some $F_\infty\subset\Omega$. Taking limit of the minimizing properties stated above, and by a standard set replacing argument, it follows that $F_\infty$ is locally perimeter-minimizing in $\RR^n$.
	
	The gradient estimate in \eqref{eq-halfplane:gradient} and energy comparison implies the following: for all $x\in\Omega$ and any competitor set $E$ with $E\Delta E_t\Subset B(x,|x_n|/2)$, we have
	\[\P{E_t;B(x,\frac{|x_n|}2)}\geq\P{E;B(x,\frac{|x_n|}2)}+\frac C{|x_n|}\big|E\Delta E_t\big|.\]
	Rescaling this, we see that $F_i$ satisfies the condition
	\[\P{F_i;B(-e_n,\frac12)}\geq\P{F;B(-e_n,\frac12)}+C\big|F\Delta F_i\big|,\]
	for all $F$ with $F\Delta F_i\Subset B(-e_n,1/2)$.
	Thus $F_i$ are uniform $(\Lambda,r_0)$-perimeter minimizers in $B(-e_n,1/2)$, in the sense stated in Appendix \ref{sec:gmt}. Then we apply Lemma \ref{lemma-gmt:Lam_r0_convergence}(i) to obtain $-e_n\in\spt(|D\chi_{F_\infty}|)$, and subsequently apply Lemma \ref{lemma-gmt:min_in_halfplane} to obtain $F_\infty=\big\{x_n<-1\big\}$. Finally, by Lemma \ref{lemma-gmt:Lam_r0_convergence}(ii) we have the convergence of normal vectors $\nu_{F_i}(-e_n)\to\nu_{F_\infty}(-e_n)=e_n$, contradicting our hypothesis. This proves the theorem.
\end{proof}

We also establish the following approximate Liouville theorem, where the obstacle $\{x_n=0\}$ is replaced by a ``soft obstacle'', represented by the weight function $\psi(x_n)$ unbounded near $x_n=1$. See Subsection \ref{subsec:weight} for the definition of weighted weak solutions.

\begin{theorem}\label{thm-halfplane:approx}
	Fix a smooth function $\psi:(-\infty,1)\to[0,\infty)$ that satisfies $\psi|_{(-\infty,0]}\equiv0$, and $\psi>0$, $\psi'>0$, $\psi''>0$, $\psi'''>0$ on $(0,1)$, and $\lim_{x\to1}\psi(x)=+\infty$. Let $\Omega=\{x_n<1\}\subset\RR^n$, and $u\in\Lip_{\loc}(\Omega)$ be a weak solution of the weighted IMCF
	\begin{equation}\label{eq-halfplane:weighted_eq}
		\div\Big(e^{\psi(x_n)}\frac{\D u}{|\D u|}\Big)=e^{\psi(x_n)}|\D u|
	\end{equation}
	in $\Omega$. Moreover, assume there is a constant $C$ so that
	
	(1) $u\geq\psi(x_n)-C$ in $\Omega$ (in particular, $\inf_\Omega(u)>-\infty$),
	
	(2) $|\D u(x)|\leq\frac{C}{|x_n|}$ for a.e. $x\in\{x_n<0\}$.
	
	\noindent Then $u(x',x_n)=\psi(x_n)-C'$ for some other constant $C'$.
\end{theorem}

Let us remark that condition (1) is not removable. Consider the (non-smooth) weight $\psi(x)=(n-1)\log\frac1{1-x}$ for $x>0$. By Lemma \ref{lemma-prelim:def_weighted}, we may equivalently consider the usual weak IMCF in the metric $g'=e^{2\psi(x_n)/(n-1)}g$ (which is hyperbolic in $\{0<x_n<1\}$), that initiates with the horoball $\{x_n<0\}$. Without condition (1), the weak solution may choose to first ``jump'' to the totally geodesic hypersurface $\Sigma_0^+=\{x_n=1-\sqrt{1-|x'|^2}\}$, and then continue to evolve by $\Sigma_t=\big\{x: d(x,\Sigma_0^+)=\arccosh e^{t/(n-1)}\big\}$.

\begin{proof}
	Shifting $u$ by a constant, we may assume $\inf_\Omega(u)=0$. We aim to prove $u(x',x_n)=\psi(x_n)$. The proof consists of three steps: we first show $u\geq\psi(x_n)$ using outer barriers, and next show that $u=0$ on $\{x_n\leq0\}$ by arguing similarly as in Theorem \ref{thm-halfplane:liouville}, and finally prove $u\leq\psi(x_n)$ using inner barriers.
	
	\vspace{3pt}
	
	\textbf{Step 1}. We show that $u\geq\psi(x_n)$. The following fact can be directly verified: if $\Sigma_t$ are graphs of functions $f=f(\cdot,t)$ over $\RR^{n-1}$, then $\Sigma_t$ is a solution (resp. subsolution, supersolution) of the smooth weighted IMCF $\frac{\p\Sigma_t}{\p t}=\frac{\nu}{H+\p\psi/\p\nu}$ if and only if
	\begin{equation}\label{eq-halfplane:graph_sol}
		\frac{\p f}{\p t}\Big(-\div\frac{\D f}{\sqrt{1+|\D f|^2}}+\frac{\psi'(f)}{\sqrt{1+|\D f|^2}}\Big)=\ \text{(resp. $\geq\,,\leq$\,)}\ \sqrt{1+|\D f|^2}.
	\end{equation}
	For $\mu>0$ and $R\geq4$, consider the function
	\begin{equation}
		\uf(x',t)=(1+\mu)\psi^{-1}(t)+\mu+R-\sqrt{R^2-|x'|^2},
	\end{equation}
	where $\psi^{-1}$ is the inverse of $\psi$. Note that $\uf$ is smooth in the region $\{(x',t)\in\RR^{n-1}\times(0,\infty): \uf(x',t)<1\}$. We claim that there exists $R_0$ depending on $\psi$ and $\mu$, such that $\uf$ is a subsolution of \eqref{eq-halfplane:graph_sol} whenever $R\geq R_0$, $t>0$ and $\uf<1$. To verify this, we compute
	\begin{equation}\label{eq-halfplane:sub_1}
		\frac{\p\uf}{\p t}=\frac{1+\mu}{\psi'(\psi^{-1}(t))}
		\qquad\text{and}\qquad
		\div\frac{\D\uf}{\sqrt{1+|\D\uf|^2}}=\frac{n-1}R.
	\end{equation}
	To evaluate the second term in \eqref{eq-halfplane:graph_sol}, we note that $\uf<1$ implies $R-\sqrt{R^2-|x'|^2}<1$, which further implies $|x'|^2\leq 2R-1$. Therefore, by direct estimation
	\begin{equation}\label{eq-halfplane:sub_2}
		1+|\D\uf|^2=\frac{R^2}{R^2-|x'|^2}\leq\frac{R^2}{(R-1)^2}\leq 1+\frac4R\qquad\text{whenever $\uf<1$},
	\end{equation}
	where we used $R\geq4$. Moreover, using the convexity of $\psi$ and $\psi'$ we have
	\begin{equation}\label{eq-halfplane:sub_3}
		\begin{aligned}
			\psi'(\uf) &\geq \psi'\big(\psi^{-1}(t)+\frac\mu2\big)+\psi''\big(\psi^{-1}(t)+\frac\mu2\big)\cdot\big(\uf-\psi^{-1}(t)-\frac\mu2\big) \\
			&\geq \psi'(\psi^{-1}(t))+\psi''(\frac\mu2)\cdot\frac\mu2.
		\end{aligned}
	\end{equation}
	Combining \eqref{eq-halfplane:sub_1} \eqref{eq-halfplane:sub_2} \eqref{eq-halfplane:sub_3}, we obtain
	\[\begin{aligned}
		& \frac{\p\uf}{\p t}\Big(-\div\frac{\D\uf}{\sqrt{1+|\D\uf|^2}}+\frac{\psi'(\uf)}{\sqrt{1+|\D\uf|^2}}\Big) \\
		\geq&\, \frac{1+\mu}{\psi'(\psi^{-1}(t))}\cdot\Big(-\frac{n-1}R+\frac{\psi'(\psi^{-1}(t))+(\mu/2)\psi''(\mu/2)}{\sqrt{1+4/R}}\Big).
	\end{aligned}\]
	If we choose $R_0$ sufficiently large so that
	\[\frac\mu2\psi''(\frac\mu2)>\sqrt{1+\frac4R}\cdot\frac{n-1}R
	\quad\text{and}\quad
	\frac{1+\mu}{\sqrt{1+4/R}}>\sqrt{1+\frac4R}
	\quad\ \,\text{for all $R\geq R_0$},\]
	then combined with \eqref{eq-halfplane:sub_2}, it follows that $\uf$ is a strict subsolution of \eqref{eq-halfplane:graph_sol} whenever $t>0$ and $\uf<1$. Switching to the level set description, we find that the function
	\[\uu(x',x_n)=\psi\Big(\frac1{1+\mu}\big(x_n-\mu-R+\sqrt{R^2-|x'|^2}\big)\Big),\]
	defined such that $\big\{\uu=t\big\}=\graph\big(f(\cdot,t)\big)$, satisfies the subsolution condition 
	\[\div\Big(e^\psi\frac{\D\uu}{|\D\uu|}\Big)>e^\psi|\D\uu|\]
	in the region $\Omega'=\big\{\mu+R-\sqrt{R^2-|x'|^2}<x_n<1\big\}\subset\Omega$.
	
	For $\epsilon>0$, we apply the interior maximum principle (Theorem \ref{thm-prelim:max_principle}) to the functions $u'=u+\epsilon$ and $\uu$. By our hypothesis we have $u\geq\max\{0,\psi(x_n)-C\}$. Hence
	\[u'(x)<\uu(x)\ \ \Rightarrow\ \ \psi(x_n)-C<\psi\big(\frac{x_n}{1+\mu}\big)\ \ \Rightarrow\ \ x_n<C'\qquad\text{for all $x\in\Omega'$},\]
	where $C'<1$ is a constant depending only on $C,\mu,\psi$. On the other hand, $u'(x)<\uu(x)$ implies $\uu(x)>\epsilon$, therefore the closure of $\{u'<\uu\}$ does not intersect $\{x_n=\mu+R-\sqrt{R^2-|x'|^2}\}$. As a result, we have $\{u'<\uu\}\Subset\Omega'$. By Lemma \ref{lemma-prelim:def_weighted}(3) and Theorem \ref{thm-prelim:max_principle}, we obtain $u'\geq\uu$. Taking $\epsilon\to0$, $R\to\infty$ and then $\mu\to0$, we obtain $u\geq\psi(x_n)$.
	
	\vspace{3pt}
	
	\textbf{Step 2}. We show that $u=0$ on $\{x_n\leq0\}$, using similar methods as in Theorem \ref{thm-halfplane:liouville}. Due to the presence of weight function, the discussion here is more complicated. Recall that we have assumed $\inf(u)=0$ and showed that $E_t\subset\{x_n<\psi^{-1}(t)\}$ for all $t>0$.
	
	\textit{Case 2(i)}: suppose for any $\epsilon>0$, there exists $\delta>0$ such that $\p^*E_t\subset\{x_n\geq -\epsilon\}$ for all $0<t\leq\delta$. Thus for each $t\in(0,\delta]$ there are only two possibilities: either $E_t\supset\{x_n<-\epsilon\}$, or $E_t\subset\{-\epsilon<x_n<\psi^{-1}(t)\}$. We claim that latter case never occurs. Once this is proved, we have $\{u\leq0\}\supset\{x_n<-\epsilon\}$, which by taking $\epsilon\to0$ proves the goal of Step 2. It remains to prove our claim. Suppose we have $E_t\subset\{-\epsilon<x<\psi^{-1}(t)\}$ for some $\epsilon,\delta$ and $t\in(0,\delta]$. Let $\pi:(x',x_n)\mapsto x'$ be the projection map. Note that $\P{E_t;A\times\RR}\geq\L^{n-1}(\pi(E_t)\cap A)$ for any open set $A\subset\RR^{n-1}$. Lemma \ref{lemma-prelim:excess_ineq} is easily generalized to the weighted case, which implies that
	\begin{equation}\label{eq-halfplane:weighted_excess}
		\int_{\p^*E_t\cap K}e^\psi\,d\H^{n-1}\leq \int_{\p^*(E_t\setminus F)\cap K}e^\psi\,d\H^{n-1}
		+(e^t-1)\int_{\p^* F\cap E_t}e^\psi\,d\H^{n-1},
	\end{equation}
	whenever $F\Subset K\Subset\Omega$. If $\H^{n-1}(\p^*E_t\cap\p^*F)=0$ and $F=F^{(1)}$, then \eqref{eq-halfplane:weighted_excess} is reduced to
	\begin{equation}\label{eq-halfplane:weighted_excess_2}
		\int_{\p^*E_t\cap F}e^\psi\,d\H^{n-1}\leq e^t\int_{\p^* F\cap E_t}e^\psi\,d\H^{n-1}.
	\end{equation}
	Since $e^\psi\geq1$ everywhere and $e^\psi\leq e^t$ in $\bar{E_t}$, we have
	\begin{equation}\label{eq-halfplane:approx_min}
		\Ps{E_t;F}\leq e^{2t}\Ps{F;E_t}.
	\end{equation}
	Now for a radius $R>0$ we choose $F=B'(0,R)\times(-1/2,1/2)$, where $B'$ denotes open balls in $\RR^{n-1}$. For almost every $R$, \eqref{eq-halfplane:approx_min} implies
	\[\begin{aligned}
		\L^{n-1}(\pi(E_t)\cap B'(0,R)) &\leq \P{E_t;B'(0,R)\times\RR}
		= \Ps{E_t;F}
		\leq e^{2t}\Ps{F;E_t} \\
		&\leq e^{2t}\big(\epsilon+\psi^{-1}(t)\big)\H^{n-2}(\p B'(0,R)\cap\pi(E_t)).
	\end{aligned}\]
	Denote $f(R)=\L^{n-1}(\pi(E_t)\cap B'(0,R))$, thus $f(R)\leq e^{2t}\big(\epsilon+\psi^{-1}(t)\big)f'(R)$. As $\inf_\Omega(u)=0$, we have $f(R)>0$ for some $R$, therefore $f(R)$ grows exponentially as $R\to\infty$. This is impossible. From the setup made above, this proves the desired result for Case 2(i).
	
	For the rest of Step 2, we may assume the contrary of Case 2(i): that there is a fixed $\epsilon>0$ and a sequence of times $t_i\to0$, such that $\p^*E_{t_i}\cap\{x_n<-\epsilon\}\ne\emptyset$ for all $i$.
	
	\textit{Case 2(ii)}: assume that there are angles $\th_i\to0$, such that for each $i$ it holds
	\[\essinf_{\p^*E_{t_i}\cap\{x_n<-\epsilon\}}\metric{\nu_{E_{t_i}}}{e_n}\geq\cos\th_i.\]
	Then arguing similarly as in \hyperlink{thm4.1case1}{Case (i)} of Theorem \ref{thm-halfplane:liouville}, we have $E_t\supset\{x_n<0\}$ for all $t>0$. This proves the desired result for Step 2.
	
	\textit{Case 2(iii)}: suppose the hypothesis of Case 2(ii) does not hold. Then passing to a subsequence, there is a constant $\th_0>0$ and points $y_i=(y'_i,y_{i,n})\in\p^*E_{t_i}\cap\{x_n<-\epsilon\}$, such that $\metric{\nu_{E_{t_i}}(y_i)}{e_n}\leq\cos\th_0$. We argue that this is impossible. For convenience, we denote $E_i=E_{t_i}$. Recall from step 1 that $E_i\subset\{x_n<\psi^{-1}(t_i)\}$. We start with investigating the minimizing properties of $E_i$.
	
	1. $E_i$ are local outward minimizers of the weighted perimeter $E\mapsto\int_{\p^*E}e^\psi\,d\H^{n-1}$ in $\{x_n<1\}$, by Lemma \ref{lemma-prelim:def_weighted}(iii). This immediately implies that $E_i$ are local outward minimizers of $E\mapsto\int_{\p^*E}e^{\min\{\psi,t_i\}}\,d\H^{n-1}$ in $\big\{x_n<\psi^{-1}(t_i)\big\}$. Then by direct verification, $E_i$ locally outward minimizes the same functional $E\mapsto\int_{\p^*E}e^{\min\{\psi,t_i\}}\,d\H^{n-1}$ in $\RR^n$.
	
	2. The inward minimizing effect of $E_i$ is given by \eqref{eq-halfplane:weighted_excess}. Since all the integrals in \eqref{eq-halfplane:weighted_excess} occur inside the set $\{\psi\leq t_i\}$, it makes no difference to replace each $\psi$ by $\min\{\psi,t_i\}$.
	
	Combining the inward and outward minimizing properties, we conclude that
	\[\begin{aligned}
		\int_{\p^*E_i\cap K}e^{\min\{\psi,t_i\}}\,d\H^{n-1} &\leq \int_{\p^* F\cap K}e^{\min\{\psi,t_i\}}\,d\H^{n-1} \\
		&\qquad\qquad +(e^{t_i}-1)\int_{\p^*(E_i\setminus F)}e^{\min\{\psi,t_i\}}\,d\H^{n-1}
	\end{aligned}\]
	whenever $F\Delta E_i\Subset K\Subset\RR^n$. This directly implies
	\[\Ps{E_i;K}\leq e^{t_i}\Ps{F;K}+e^{t_i}(e^{t_i}-1)\Ps{E_i\setminus F}.\]
	
	Next, consider the rescaled sets $F_i=(E_i-y'_i)/|y_{i,n}|$. It follows that $F_i\subset\{x_n<1/\epsilon\}$, $-e_n\in\p^*F_i$, and $\metric{\nu_{F_i}(-e_n)}{e_n}\leq\cos\th_0$, and $F_i$ satisfy the minimizing property
	\begin{equation}\label{eq-halfplane:almost_min_2}
		\Ps{F_i;K}\leq e^{t_i}\Ps{F;K}+e^{t_i}(e^{t_i}-1)\Ps{F_i\setminus F}
	\end{equation}
	whenever $F\Delta F_i\Subset K\Subset\RR^n$. Moreover, since $\psi_i=0$ in $\{x_n<0\}$, and by the gradient estimate $|\D u|\leq C/|x_n|$, we have the uniform $(\Lambda,r_0)$-minimizing property
	\begin{equation}\label{eq-halfplane:uniform_Lam_r0}
		\P{F_i;B(-e_n,\frac12)}\leq\P{F;B(-e_n,\frac12)}+C|F\Delta F_i|
	\end{equation}
	whenever $F\Delta F_i\Subset B(-e_n,\frac12)$.
	
	We are in a position to pass to the limit: there is a subsequence of $F_i$ that converge to a set $F_\infty$ in $L^1_{\loc}$. By \eqref{eq-halfplane:uniform_Lam_r0} and Lemma \ref{lemma-gmt:Lam_r0_convergence}(i) we have $-e_n\in\p^*F_\infty$. By \eqref{eq-halfplane:almost_min_2} and the fact $t_i\to0$ and the standard set replacing argument, the limit $F_\infty$ locally minimizes the perimeter in $\RR^n$. By Lemma \ref{lemma-gmt:min_in_halfplane} and the fact $F_\infty\subset\{x_n<1/\epsilon\}$, we have $F_\infty=\{x_n<-1\}$. Finally, by Lemma \ref{lemma-gmt:Lam_r0_convergence}(ii) we have convergence of normal vectors $\nu_{F_i}(-e_n)\to\nu_{F_\infty}(-e_n)=e_n$, contradicting our hypothesis. This completes the proof of step 2.
	
	\vspace{3pt}
	
	\textbf{Step 3.} We have shown that $u\geq\psi(x_n)$ and $u=0$ on $\{x_n<0\}$. In this step we use inner barriers to show that $u\leq\psi(x_n)$, which completes the proof of the theorem. For constants $\mu\ll1$, $R\gg1$ and for $t>0$, we consider the function
	\[\of(x',t)=(1-\mu)\int_0^t\frac{ds}{\psi'(\psi^{-1}(s))+(n-1)/R}-R+\sqrt{R^2-|x'|^2}.\]
	We claim that $\of$ is a supersolution of \eqref{eq-halfplane:graph_sol} whenever $t>0$ and $\of>0$. Note that $\of\leq(1-\mu)\int_0^t\frac{ds}{\psi'(\psi^{-1}(s))}<(1-\mu)\psi^{-1}(t)$. In particular, $\of<1$ always holds. To verify the supersolution property, according to \eqref{eq-halfplane:graph_sol} we compute
	\[\frac{\p\of}{\p t}=\frac{1-\mu}{\psi'(\psi^{-1}(t))+(n-1)/R}
	\qquad\text{and}\qquad
	\div\Big(\frac{\D f}{\sqrt{1+|\D f|^2}}\Big)=-\frac{n-1}R.\]
	By the convexity of $\psi$, we have $\psi'(\of)\leq\psi'(\psi^{-1}(t))$. Therefore,
	\[\begin{aligned}
		&\, \frac{\p\of}{\p t}\Big(-\div\frac{\D\of}{\sqrt{1+|\D\of|^2}}+\frac{\psi'(\of)}{\sqrt{1+|\D\of|^2}}\Big) \\
		\leq&\, \frac{1-\mu}{\psi'(\psi^{-1}(t))+(n-1)/R}\cdot\Big(\frac{n-1}R+\psi'(\psi^{-1}(t))\Big) \\
		=&\, 1-\mu
		<\, \sqrt{1+|\D\of|^2}.
	\end{aligned}\]
	Thus $\of$ is a supersolution of \eqref{eq-halfplane:graph_sol}.
	
	Consider the unique positive function $\ou$ on $\{0<x_n<1\}$, such that $E_t(\ou)=\big\{(x',x_n): 0<x_n<\of(x',t)\big\}$ for all $t>0$. Then $\ou$ is a smooth supersolution of IMCF in the region $\Omega':=\{0<x_n<1,\ou<\infty\}$. For $\epsilon>0$ we wish to compare $u$ with the function $\ou+\epsilon$. Assembling the facts: $u=0$ on $\{x_n=0\}$, and $\{\ou<\infty\}\subset\big\{x_n<1-\mu-R+\sqrt{R^2-|x'|^2}\big\}$, and $u\in\Lip_{\loc}(\{x_n<1\})$, we conclude that $\{u>\ou+\epsilon\}\Subset\Omega'$. By Lemma \ref{lemma-prelim:def_weighted} (3) and Theorem \ref{thm-prelim:max_principle}, we obtain $u\leq\ou+\epsilon$. Taking $\epsilon\to0$ and then $R\to\infty$ and then $\mu\to0$, we eventually obtain $u\leq\psi(x_n)$.
\end{proof}

\section{Parabolic estimates near smooth obstacles}\label{sec:para}

In this section, we compute several parabolic evolution equations for the smooth IMCF near a smooth obstacle. The aim is to show the following: if $\Sigma_t$ evolves under the IMCF in a smooth domain $\Omega$, such that
\begin{enumerate}[label={(\arabic*)}, nosep]
	\item $\nu_{\Sigma_t}$ is approximately equal to $\nu_{\Omega}$ on $\p\Omega$ (where $\nu$ denotes the outer unit normal of the corresponding objects),
	\item in some neighborhood of $\p\Omega$ we have $\metric{\nu_{\Sigma_t}}{\p_r}\geq\frac12$ (where $\p_r$ is an extension of $\nu_\Omega$, see below),
\end{enumerate}
then $\metric{\nu_{\Sigma_t}}{\p_r}\geq 1-Cr^\gamma-o(1)$ in some smaller neighborhood of $\p\Omega$. This estimate enters the proofs in Section \ref{sec:exist} in showing the $C^{1,\alpha}$ regularity of level sets and boundary regularity of blow-up limits.

\vspace{6pt}

The following notations are used consistently in the present and the next sections. Suppose $\Omega$ is a smooth domain in a Riemannian manifold $M$. Define the signed distance function
\[r(x):=\left\{\begin{aligned}
	& -d(x,\p\Omega)\qquad x\in\Omega, \\
	& d(x,\p\Omega)\qquad x\notin\Omega,
\end{aligned}\right.\]
For $\delta\in\RR$, we denote $\Omega_\delta:=\big\{x\in M: r(x)<\delta\big\}$. We use $r_0$ to denote a radius such that $r(x)$ is smooth in $\Omega\setminus\bar{\Omega_{-r_0}}$. In Lemma \ref{lemma-para:delta_p}$\,\sim\,$\ref{lemma-para:box_eta}, the existence of such a radius is implicitly assumed. In $\Omega\setminus\bar{\Omega_{-r_0}}$ we define the outer radial vector field $\p_r:=\D r$.

For a hypersurface $\Sigma\subset\Omega\setminus\bar{\Omega_{-r_0}}$, we use $\nu, A, H$ to denote the unit normal vector, second fundamental form and mean curvature. We use $|\cdot|_\Sigma, \DSigma, \Delta_\Sigma$ to denote the (Hilbert-Schmidt) norm, gradient and Laplacian on $\Sigma_t$, and use $|\cdot|, \D, \Delta$ for objects with respect to the ambient manifold $M$. Also, denote
\[p:=\metric{\nu}{\p_r}.\]
The following algebraic fact is useful: $|\nu-p\p_r|=|\p_r-p\nu|=|\DSigma r|=\sqrt{1-p^2}$.

When there is a family of hypersurfaces $\{\Sigma_t\}$, we often omit the dependence on $t$ in the above notations. When $\Sigma_t$ evolves under the IMCF, we denote $\Box=\p_t-\frac1{H^2}\Delta_\Sigma$ the associated heat operator.


\vspace{9pt}

We start by evaluating $\Delta_\Sigma p$. The main formula \eqref{eq-para:delta_p} is arranged such that each term vanishes when $\Sigma=\p\Omega_{-r}$ for some $r$, and the terms with unfavorable sign become small when $M$ and $\p\Omega$ are close to being flat.

\begin{lemma}\label{lemma-para:delta_p}
	Let $\Sigma$ be a smooth hypersurface in $\Omega\setminus\bar{\Omega_{-r_0}}$, such that $p>0$ on $\Sigma$. Then we have
	\begin{equation}\label{eq-para:delta_p}
		\begin{aligned}
			\Delta_\Sigma p &\leq \metric{\D_\Sigma H}{\p_r}-p\big|A-p^{-1}\D^2 r\big|^2_\Sigma \\
			&\qquad +|\D^2 r|^2p^{-1}(1-p^2)+\big(|\Ric|+n|\D^2\p_r|\big)\sqrt{1-p^2}.
		\end{aligned}
	\end{equation}
	(In the second term, the expression $\D^2r$ means the restriction of $\D^2_Mr$ to $\Sigma$.)
\end{lemma}
\begin{proof}
	Near a given point $x\in\Sigma$, we let $\{e_i\}_{1\leq i\leq n-1}$ be a local orthonormal frame on $\Sigma$, such that $\D_{e_i}^\Sigma e_j=0$ at $x$. We compute
	\begin{align}
		\Delta_\Sigma p &= \sum_i e_ie_i\metric{\nu}{\p_r}
		= \sum_i e_i\Big[\metric{A(e_i)}{\p_r}+\metric{\nu}{\D_{e_i}\p_r}\Big] \nonumber\\
		&= \sum_i\Big[\metric{\D_{e_i}(A(e_i))}{\p_r}+2\metric{A(e_i)}{\D_{e_i}\p_r}+\metric{\nu}{\D_{e_i}\D_{e_i}\p_r}\Big]. \label{eq-para:del_p_master}
	\end{align}
	Here $A(e_i)$ is understood as a tangent vector field of $\Sigma$. First observe that
	\begin{equation}\label{eq-para:del_p_II}
		2\sum_i\metric{A(e_i)}{\D_{e_i}\p_r}=2\sum_i\D^2r(A(e_i),e_i)=2\metric{A}{\D^2r}_\Sigma.
	\end{equation}
	Next, the first term in \eqref{eq-para:del_p_master} is calculated as
	\begin{align}
		\sum_i\metric{\D_{e_i}A(e_i)}{\p_r} &= \sum_i\Big\langle(\D_{e_i}^\Sigma A)(e_i)+A(\D_{e_i}^\Sigma e_i)-A(e_i,A(e_i))\nu,\,\p_r\Big\rangle \nonumber\\
		&= \metric{\div_\Sigma A-|A|^2_\Sigma\,\nu}{\p_r} \nonumber\\
		&= \metric{\div_\Sigma A}{\p_r-p\nu}-p|A|^2 \nonumber\\
		&\leq \metric{\D_\Sigma H}{\p_r}-p|A|^2_\Sigma+|\Ric|\sqrt{1-p^2}, \label{eq-para:del_p_I}
	\end{align}
	where we use the traced Codazzi equation and note that $\div_\Sigma A$, $\DSigma H$ are tangent to $\Sigma$. It remains to simplify the last term in \eqref{eq-para:del_p_master}. We calculate
	\begin{align}
		\sum_i\metric{\nu}{\D_{e_i}\D_{e_i}\p_r} &= \sum_i\Big[\metric{\nu-p\p_r}{\D_{e_i}\D_{e_i}\p_r}+p\metric{\p_r}{\D_{e_i}\D_{e_i}\p_r}\Big] \nonumber\\
		&\leq n|\D^2\p_r|\sqrt{1-p^2}-p\sum_i\big|\D_{e_i}\p_r\big|^2, \label{eq-para:del_p_III}
	\end{align}
	Continuing to evaluate the second term in \eqref{eq-para:del_p_III}:
	\begin{equation}
		\sum_i\big|\D_{e_i}\p_r\big|^2 = \sum_{ij}\metric{\D_{e_i}\p_r}{e_j}^2+\sum_i\metric{\D_{e_i}\p_r}{\nu}^2 \geq \sum_{ij}\metric{\D_{e_i}\p_r}{e_j}^2=|\D^2 r|^2_\Sigma. \label{eq-para:del_p_III2}
	\end{equation}
	Inserting \eqref{eq-para:del_p_II}$\,\sim\,$\eqref{eq-para:del_p_III2} into \eqref{eq-para:del_p_master}, we obtain
	\[\begin{aligned}
		\Delta_\Sigma p &\leq \metric{\D_\Sigma H}{\p_r}-p|A|^2_\Sigma+2\metric{A}{\D^2r}_\Sigma-p|\D^2 r|^2_\Sigma \\
		&\hspace{144pt} +|\Ric|\sqrt{1-p^2}+n|\D^2\p_r|\sqrt{1-p^2}.
	\end{aligned}\]
	This implies \eqref{eq-para:delta_p} by completing the square and noting that $|\D^2r|_\Sigma\leq|\D^2r|$.
\end{proof}

\begin{lemma}\label{lemma-para:box_p}
	Suppose $\{\Sigma_t\}$ evolves under the IMCF in $\Omega\setminus\bar{\Omega_{-r_0}}$, such that $p>0$ holds everywhere. Then we have
	\begin{equation}\label{eq-para:box_p}
		\begin{aligned}
			\Box p &\geq \frac{|A-p^{-1}\D^2r|^2_\Sigma}{2H^2}p
			-2n\frac{|\D^2r|^2}{H^2p}(1-p^2)
			-\frac{|\Ric|+n|\D^2\p_r|}{H^2}\sqrt{1-p^2}.
		\end{aligned}
	\end{equation}
\end{lemma}
\begin{proof}
	We compute
	\[\begin{aligned}
		\frac{\p p}{\p t} &= \p_t\metric{\nu}{\p_r}
		= \metric{H^{-2}\DSigma H}{\p_r}+\metric{\nu}{\D_{H^{-1}\nu}\p_r}.
	\end{aligned}\]
	For the second term, we use the fact $\D^2r(\p_r,X)=0$ for all $X$ to evaluate
	\[\begin{aligned}
		\metric{\nu}{\D_{H^{-1}\nu}\p_r} = \frac1H\D^2 r(\nu,\nu)
		= \frac1H\D^2 r(\nu-p\p_r,\nu-p\p_r)
		\geq -\frac1H|\D^2r|(1-p^2).
	\end{aligned}\]
	Combined with \eqref{eq-para:delta_p}, we obtain
	\begin{equation}\label{eq-para:box_p_aux}
		\begin{aligned}
			\Box p &\geq \frac{|A-p^{-1}\D^2r|_\Sigma^2}{H^2}p-\frac1H|\D^2r|(1-p^2) \\
			&\qquad\qquad\qquad -\frac{|\D^2r|^2}{H^2p}(1-p^2)-\frac{|\Ric|+n|\D^2\p_r|}{H^2}\sqrt{1-p^2}.
		\end{aligned}
	\end{equation}
	We bound the second term in this expression as follows:
	\[\begin{aligned}
		\frac1H|\D^2r|(1-p^2) &= \frac{H-p^{-1}\tr_\Sigma\D^2r+p^{-1}\tr_\Sigma\D^2r}{H^2}|\D^2r|(1-p^2) \\
		&\leq \sqrt{n-1}\frac{|A-p^{-1}\D^2r|_\Sigma+p^{-1}|\D^2r|_\Sigma}{H^2}|\D^2r|(1-p^2).
	\end{aligned}\]
	Using Young's inequality, we continue the estimate:
	\[\begin{aligned}
		\frac1H|\D^2r|(1-p^2) &\leq \Big[\frac{|A-p^{-1}\D^2r|_\Sigma^2}{2H^2}p+\frac{n}{2H^2p}|\D^2r|^2(1-p^2)^2\Big]+n\frac{|\D^2r|^2}{H^2p}(1-p^2).
	\end{aligned}\]
	The result follows by combining this into \eqref{eq-para:box_p_aux} and noting that $0<p\leq1$.
\end{proof}

\begin{lemma}\label{lemma-para:box_eta}
	Assume the same conditions as in Lemma \ref{lemma-para:box_p}. Let $\eta=\eta(r)$ be a smooth radial function, and denote $\eta'=d\eta/dr$, $\eta''=d^2\eta/dr^2$. Then we have
	\begin{equation}\label{eq-para:box_eta}
		\begin{aligned}
			\Box\eta &\leq \frac{2p\eta'}{H^2}\big(H-p^{-1}\tr_\Sigma\D^2r\big)-\frac{\eta''}{H^2}(1-p^2)+\frac{n|\eta'|}{H^2}|\D^2r|.
		\end{aligned}
	\end{equation}
	and
	\begin{equation}\label{eq-para:box_eta_2}
		\Box\eta \leq \frac{2p\eta'}{H}-\frac{\eta''}{H^2}(1-p^2)+\frac{n|\eta'|}{H^2}|\D^2r|.
	\end{equation}
\end{lemma}
\begin{proof}
	This is obtained by combining
	\[\frac{\p\eta}{\p t}
		=\frac{\metric{\D\eta}{\nu}}{H}
		=\frac{p\eta'}{H}
		=\frac{p\eta'}{H^2}\big(H-p^{-1}\tr_\Sigma\D^2r+p^{-1}\tr_\Sigma\D^2r\big)\]
	with
	\[\begin{aligned}
		\Delta_\Sigma\eta &= \eta'\Delta_\Sigma r+\eta''|\DSigma r|^2 
		= \eta'(\tr_\Sigma\D^2r-pH)+\eta''(1-p^2),
	\end{aligned}\]
	and noting that $\tr_\Sigma\D^2r\leq\sqrt{n-1}|\D^2r|_\Sigma\leq n|\D^2r|$.
\end{proof}

\begin{lemma}\label{lemma-para:box_qeta}
	Suppose $\Omega$ is a smooth domain, and $r_0>0$ is a radius such that $r(x)$ is smooth in $\Omega\setminus\bar{\Omega_{-r_0}}$. Assume that
	\begin{equation}\label{eq-para:small_curv}
		|\Ric|\leq r_0^{-2},\qquad |\D^2r|\leq r_0^{-1},\qquad |\D^2\p_r|\leq r_0^{-2}
	\end{equation}
	hold inside $\Omega\setminus\bar{\Omega_{-r_0}}$. Let $\big\{\Sigma_t\subset\Omega\setminus\bar{\Omega_{-r_0}}\big\}$ be a smooth family of hypersurfaces evolving under the IMCF. Assume additionally that $p\geq\frac12$ on all $\Sigma_t$. Then there exist constants $\gamma\in(0,1/2)$ and $C_1,C_2>0$ depending only on $n$, such that the following holds: if we set
	\begin{equation}
		\eta(r)=\big(b-r/r_0\big)^{-\gamma},\qquad F=(1-p)\eta,
	\end{equation}
	with any choice $b\in(0,1]$, then we have the evolution inequality
	\begin{equation}
		\begin{aligned}
			(1-p)^{\frac2\gamma-1}H^2\Box F+\metric{\DSigma F}{X}\leq \frac{C_1}{r_0^2}\big(-F^{\frac{\gamma+2}\gamma}+C_2\,\big)
		\end{aligned}
	\end{equation}
	on $\Sigma_t$, where $X$ is a certain smooth vector field.
\end{lemma}
\begin{proof}
	Note that $r<0$ inside $\Omega$, hence $\eta>0$ and increasing when approaching $\p\Omega$. Combining Lemma \ref{lemma-para:box_p} and equation \eqref{eq-para:box_eta} in Lemma \ref{lemma-para:box_eta}, then using \eqref{eq-para:small_curv} with $\frac12\leq p\leq1$ to simplify the resulting expressions, we have
	\begin{align}
		\Box F &= -\eta\Box p+(1-p)\Box\eta-\frac2{H^2}\metric{\DSigma(1-p)}{\DSigma\eta} \nonumber\\
		&\leq -\frac{|A-p^{-1}\D^2r|^2_\Sigma}{2H^2}\eta p+2n\frac{|\D^2r|^2}{H^2p}\eta(1-p^2)+\frac{|\Ric|+n|\D^2\p_r|}{H^2}\eta\sqrt{1-p^2} \nonumber\\
		&\qquad +\frac{2p\eta'}{H^2}(1-p)\big(H-p^{-1}\tr_\Sigma\D^2r\big)-\frac{\eta''}{H^2}(1-p)(1-p^2)+n\frac{|\eta'|}{H^2}|\D^2r|(1-p) \nonumber\\
		&\qquad -\frac2{H^2\eta}\metric{\DSigma F}{\DSigma\eta}+\frac2{H^2\eta}(1-p)|\DSigma\eta|^2 \nonumber\\
		&\leq -\frac{|A-p^{-1}\D^2r|^2_\Sigma}{2H^2}\eta p+\frac{8n}{H^2r_0^2}\eta(1-p)+\frac{4n\eta}{H^2r_0^2}\sqrt{1-p} \nonumber\\
		&\qquad +\frac{2p\eta'}{H^2}(1-p)\big(H-p^{-1}\tr_\Sigma\D^2r\big)-\frac{\eta''}{H^2}(1-p)^2(1+p)+\frac{n|\eta'|}{H^2r_0}(1-p) \nonumber\\
		&\qquad -\frac2{H^2\eta}\metric{\DSigma F}{\DSigma\eta}+\frac2{H^2}\frac{(\eta')^2}{\eta}(1-p)^2(1+p). \nonumber
	\end{align}
	We use Young's inequality to estimate the fourth term:
	\[\begin{aligned}
		\frac{2p\eta'}{H^2}(1-p)\big(H-p^{-1}\tr_\Sigma\D^2r\big) &\leq \frac{(H-p^{-1}\tr_\Sigma\D^2r)^2}{2(n-1)H^2}\eta p+8(n-1)\frac{(\eta')^2}{H^2\eta}p(1-p)^2 \\
		&\leq \frac{|A-p^{-1}\D^2r|^2_\Sigma}{2H^2}\eta p+8n\frac{(\eta')^2}{H^2\eta}p(1-p)^2.
	\end{aligned}\]
	Further, we calculate
	\[\eta'=\frac{\gamma}{r_0}\eta^{\frac{\gamma+1}\gamma},\qquad\eta''=\frac{\gamma(\gamma+1)}{r_0^2}\eta^{\frac{\gamma+2}\gamma}.\]
	Inserting these to the main estimate, we obtain
	\[\begin{aligned}
		H^2\Box F &\leq \frac{8n}{r_0^2}\eta(1-p)
		+ \frac{4n\eta}{r_0^2}\sqrt{1-p}
		+ 8n\frac{\gamma^2}{r_0^2}\eta^{\frac{\gamma+2}\gamma}p(1-p)^2 \\
		&\qquad -\frac{\gamma(\gamma+1)}{r_0^2}\eta^{\frac{\gamma+2}\gamma}(1-p)^2(1+p)
		+\frac{n\gamma}{r_0^2}\eta^{\frac{\gamma+1}\gamma}(1-p) \\
		&\qquad -\frac2\eta\metric{\DSigma F}{\DSigma\eta}
		+ \frac{2\gamma^2}{r_0^2}\eta^{\frac{\gamma+2}\gamma}(1-p)^2(1+p).
	\end{aligned}\]
	Multiplying both sides by $(1-p)^{\frac2\gamma-1}$, and using the facts $0<p\leq1$, $1-p\leq\sqrt{1-p}$, we obtain
	\begin{equation}\label{eq-para:intermediate1}
		\begin{aligned}
			(1-p)^{\frac2\gamma-1}H^2\Box F &\leq \frac{12n}{r_0^2}\eta(1-p)^{\frac2\gamma-\frac12}+\frac{n\gamma}{r_0^2}\eta^{\frac{\gamma+1}\gamma}(1-p)^{\frac2\gamma}-\metric{\DSigma F}{X} \\
			&\qquad +\frac1{r_0^2}\Big[8n\gamma^2+4\gamma^2-\gamma(\gamma+1)\Big]\eta^{\frac{\gamma+2}\gamma}(1-p)^{\frac{\gamma+2}\gamma},
		\end{aligned}
	\end{equation}
	where $X$ is a certain smooth vector field. Choosing $\gamma$ sufficiently small (depending only on $n$), we can ansure
	\begin{equation}\label{eq-para:aux1}
		8n\gamma^2+4\gamma^2-\gamma(\gamma+1)<-3\gamma^2.
	\end{equation}
	Using H\"older's inequality, we have
	\begin{equation}\label{eq-para:aux2}
		12n\eta(1-p)^{\frac2\gamma-\frac12}
		\leq \gamma^2\eta^{\frac{\gamma+2}\gamma}(1-p)^{\frac{\gamma+2}\gamma}+C(n)\eta^{-\frac{(4-3\gamma)(2+\gamma)}{3\gamma^2}}
	\end{equation}
	and
	\begin{equation}\label{eq-para:aux3}
		n\gamma\eta^{\frac{\gamma+1}\gamma}(1-p)^{\frac2\gamma}
		\leq \gamma^2\eta^{\frac{\gamma+2}\gamma}(1-p)^{\frac{\gamma+2}\gamma}+C(n)\eta^{-\frac{(1-\gamma)(2+\gamma)}{\gamma^2}}
	\end{equation}
	Notice that $\eta\geq2^{-\gamma}$ in $\Omega\setminus\bar{\Omega_{-r_0}}$, which bounds the last terms uniformly by constants. Inserting \eqref{eq-para:aux1}$\,\sim\,$\eqref{eq-para:aux3} into \eqref{eq-para:intermediate1}, we finally obtain
	\[(1-p)^{\frac2\gamma-1}H^2\Box F+\metric{\DSigma F}{X}\leq \frac1{r_0^2}\Big(-\gamma^2F^{\frac{\gamma+2}\gamma}+C(n)\Big).\qedhere\]
\end{proof}

\begin{lemma}\label{lemma-para:box_etaH}
	Suppose $b,\gamma\in(0,1/2)$ are constants, $\Omega$ is a smooth domain, and $r_0>0$ is a radius such that $r(x)$ is smooth in $\Omega\setminus\bar{\Omega_{-r_0}}$. Assume that
	\begin{equation}\label{eq-para:small_curv2}
		|\Ric|\leq r_0^{-2},\qquad |\D^2r|\leq r_0^{-1},\qquad |\D^2\p_r|\leq r_0^{-2}
	\end{equation}
	holds inside $\Omega\setminus\Omega_{-r_0}$. Suppose $\big\{\Sigma_t\subset\Omega\setminus\bar{\Omega_{-r_0}}\big\}$ evolves by the IMCF. Additionally, assume $p\geq\frac12$ and $1-p\leq C_3(b-r/r_0)^\gamma$ on each $\Sigma_t$. Set
	\begin{equation}
		\eta(r)=\big(b-r/r_0\big)^{1-\gamma}-(2b)^{1-\gamma},\qquad G=r_0H\eta(r).
	\end{equation}
	Then in $\Omega_{-br_0}\setminus\bar{\Omega_{-r_0}}$ we have the evolution inequality
	\begin{equation}\label{eq-para:box_etaH}
		\begin{aligned}
			\Box G+\metric{\DSigma G}{X} &\leq -\frac Gn+\frac{4n}G+C_4\big(b-r/r_0\big)^{-\gamma}\Big(\frac{C_5}G-1\Big),
		\end{aligned}
	\end{equation}
	where $X$ is a certain smooth vector field, and the constants $C_4,C_5$ depend on $n,C_3$.
\end{lemma}
\begin{proof}
	Recall the evolution equation of the mean curvature:
	\[\Box H=-\frac2{H^3}|\DSigma H|^2-\Big(|A|^2+\Ric(\nu,\nu)\Big)\frac1H.\]
	Combining \eqref{eq-para:box_eta_2} and \eqref{eq-para:small_curv2} we have
	\[\begin{aligned}
		r_0^{-1}\Box G &= \eta\Box H+H\Box\eta-\frac{2}{H^2}\metric{\DSigma H}{\DSigma\eta} \\
		&\leq -\frac{2\eta}{H^3}|\DSigma H|^2-\frac{\eta H}{n-1}+\frac{\eta}{r_0^2H}
		+2p\eta'-\frac{\eta''}{H}(1-p^2)+\frac{n|\eta'|}{r_0H} \\
		&\qquad -\frac2{H^3}\metric{\DSigma H}{\DSigma(H\eta)-\eta\DSigma H} \\
		&= -\frac G{(n-1)r_0}+\frac{\eta^2}{r_0G}+2p\eta'-r_0\frac{\eta\eta''}{G}(1-p^2)+\frac{n\eta|\eta'|}{G}-\metric{\DSigma G}{X}.
	\end{aligned}\]
	Note that
	\[\eta'=-(1-\gamma)r_0^{-1}(b-r/r_0)^{-\gamma},\qquad \eta''=-\gamma(1-\gamma)r_0^{-2}(b-r/r_0)^{-1-\gamma}.\]
	These imply that $|\eta\eta'|\leq 2r_0^{-1}$ in $\Omega_{-br_0}\setminus\bar{\Omega_{-r_0}}$, since $\gamma<1/2$. In addition, we have $\eta\leq2$ and $1-p^2\leq2(1-p)$. Hence
	\[\Box G \leq -\frac Gn+\frac{4+2n}{G}-2p(1-\gamma)(b-r/r_0)^{-\gamma}+2\gamma(1-\gamma)\frac{1-p}{G}(b-r/r_0)^{-2\gamma}.\]
	Inserting our assumption $1-p\leq C_3(b-r/r_0)^{\gamma}$, we obtain
	\[\Box G\leq -\frac{G}{n}+\frac{4n}{G}+(b-r/r_0)^{-\gamma}\Big[-2p(1-\gamma)+\frac{C}{G}\Big]-\metric{\DSigma G}{X}.\]
	Since $p\geq\frac12,\gamma\leq\frac12\,\Rightarrow-2p(1-\gamma)\leq-\frac12$, the result follows.
\end{proof}

\section{Initial value problems in smooth domains}\label{sec:exist}

In this section we prove the main existence and regularity theorem, which is stated below. In Subsection \ref{subsec:exist_setup}, we set up some definitions and preliminary estimates, then we define the approximation scheme mentioned in the introduction, which gives a candidate solution $u$ (it is temporarily only an interior solution). In the end of Subsection \ref{subsec:exist_setup} we state the remaining tasks needed to conclude that $u$ satisfies Theorem \ref{thm-exist:main}. Then in Subsection \ref{subsec:exist_strategy} we make an outline of the blow up strategy, and finally prove these results in Subsection \ref{subsec:exist_proofs}.

\begin{theorem}\label{thm-exist:main}
	Let $\Omega\subset M$ be a precompact domain with smooth boundary, and $E_0\Subset\Omega$ be a $C^{1,1}$ domain. Then there exists a solution $u$ of $\IVPOOinthm{\Omega;E_0}{\p\Omega}$, unique up to equivalence. There exists $\gamma\in(0,1)$ depending on $n$, such that the following holds.
	
	(i) We have $u\in\Lip_{\loc}(\Omega)\cap\BV(\Omega)\cap C^{0,\gamma}(\bar\Omega)$. More precisely,
	
	\begin{equation}\label{eq-exist:main_grad_int}
		|\D u(x)|\leq\sup_{\p E_0\cap B(x,r)}H_++\frac{C(n)}{r},\qquad x\in\Omega\setminus E_0,\ \ r\leq\sigma(x;\Omega,g),
	\end{equation}
	where $\sigma(x;\Omega,g)$ is as in Definition \ref{def-prelim:sigma_x}, and
	\begin{equation}\label{eq-exist:main_grad_bd}
		|\D u(x)|\leq Cd(x,\p\Omega)^{\gamma-1},\qquad x\in\Omega\setminus\bar{E_0},
	\end{equation}
	for some constant $C>0$.
	
	(ii) The solution $u$ is calibrated in $\Omega\setminus\bar{E_0}$ by a vector field $\nu$, which satisfies
	\begin{equation}\label{eq-exist:main_calib_dir}
		\metric{\nu}{\p_r}(x)\geq 1-Cd(x,\p\Omega)^\gamma
	\end{equation}
	in some small neighborhood of $\p\Omega$. Here $C>0$ is a constant, and $\p_r:=-\D d(\cdot,\p\Omega)$ is the outpointing unit vector perpendicular to $\p\Omega$.
	
	(iii) Each level set $\p E_t$ is a $C^{1,\gamma/2}$ hypersurface in some small neighborhood of $\p\Omega$.

	(iv) If $v\in\Lip_{\loc}(\Omega)$ is some other solution of $\IVP{\Omega;E_0}$, then $u\geq v$ in $\Omega\setminus E_0$.
\end{theorem}

\subsection{Setup, notations, and the approximating scheme}\label{subsec:exist_setup}

We make the following setup and constructions.

\vspace{3pt}

\noindent\textbf{The signed distance $r(x)$ and regular radius $r_I$.}

Fix $\Omega$, $E_0$ as in Theorem \ref{thm-exist:main}. Let $r(x)$ be the signed distance function to $\p\Omega$, taking negative values in $\Omega$ and positive values in $M\setminus\Omega$. For $\delta\in\RR$, we set $\Omega_\delta=\big\{x\in M: r(x)<\delta\big\}$. Thus $\Omega_\delta\supset\Omega$ for $\delta>0$ and $\Omega_\delta\subset\Omega$ for $\delta<0$. Let $r_I$ be a sufficiently small radius, such that

(1) $d(\p E_0,\p\Omega)>3r_I$,

(2) $r(x)$ is smooth in $\Omega_{3r_I}\setminus\bar{\Omega_{-3r_I}}$.

\noindent Define the radial vector field $\p_r:=\D r$, which is smooth in the same region. We further decrease $r_I$ such that the following hold:

(3) in $\Omega_{3r_I}\setminus\bar{\Omega_{-3r_I}}$ we have
\begin{equation}\label{eq-exist:small_curv}
	|\Ric|\leq\frac1{100n^2r_I^2},\qquad |\D\p_r|\leq\frac1{100n^2r_I},\qquad|\D^2\p_r|\leq\frac1{100n^2r_I^2},
\end{equation}

(4) for all $x\in\Omega\setminus\Omega_{-r_I}$ and $r,s\leq r_I$, we have $\H^{n-1}\big(B(x,r)\cap\p\Omega_{-s}\big)\leq 2|B^{n-1}|r^{n-1}$,

(5) for all $x\in\Omega\setminus\Omega_{-r_I}$ we have $\sigma(x;\Omega,g)\geq\frac12|r(x)|$ (see Definition \ref{def-prelim:sigma_x}).

\vspace{6pt}
\noindent\textbf{The weight functions $\psi_\delta$.}

We fix a function $\psi_0:(-\infty,1)\to\RR_+$ such that:

(1) the conditions of Theorem \ref{thm-halfplane:approx} are satisfied (in particular, $\psi_0$ is strictly increasing and convex in $(0,1)$, and $\psi_0|_{[-\infty,0]}\equiv0$, $\lim_{x\to1}\psi_0(x)=+\infty$),

(2) $\psi_0(\frac12)>\frac12$, and $\psi'_0(x)>1$ for all $x\in[\frac12,1)$, and $\psi_0(x)>\psi_0(\frac12)+3$ for all $x\in[\frac34,1)$.

(3) $\psi_0(x)\geq(n-1)\log\frac1{1-x}$ for all $x\in[\frac12,1)$.

\vspace{3pt}

\noindent For each $\delta\in(0,r_I)$, we define the (smooth) weight function
\[\psi_\delta(x)=\left\{\begin{aligned}
	& 0,\qquad x\in\Omega, \\
	& \psi_0\big(\delta^{-1}r(x)\big),\qquad x\in\Omega_\delta\setminus\Omega.
\end{aligned}\right.\]
Note that $\lim_{x\to\p\Omega_\delta}\psi_\delta(x)=+\infty$.

\vspace{6pt}
\noindent\textbf{The approximating equations}.

For $\epsilon>0$, we consider the weighted IMCF
\begin{equation}\label{eq-exist:weighted_eq}
	\div\Big(e^{\psi_\delta}\frac{\D u}{|\D u|}\Big)=e^{\psi_\delta}|\D u|,
\end{equation}
and its elliptic regularization (see Remark \ref{rmk-ellreg:weighted}):
\begin{equation}\label{eq-exist:elliptic_reg}
	\div\Big(e^{\psi_\delta}\frac{\D u}{\sqrt{\epsilon^2e^{2\psi_\delta/(n-1)}+|\D u|^2}}\Big)=e^{\psi_\delta}\sqrt{\epsilon^2e^{2\psi_\delta/(n-1)}+|\D u|^2}.
\end{equation}

The following technical lemma provides upper and lower barriers for solving \eqref{eq-exist:elliptic_reg}:

\begin{lemma}[barrier functions]\label{lemma-exist:subsol}
	Suppose $\delta\in(0,r_I)$. Then the function
	\begin{equation}\label{eq-exist:usub}
		\usub(x)=\psi_\delta(x)-\psi_0\big(\frac12\big)-\Big(\frac{r(x)}{\delta}-\frac12\Big)
	\end{equation}
	is negative in $\Omega_{\delta/2}\setminus\Omega$ and positive in $\Omega_\delta\setminus\bar{\Omega_{\delta/2}}$. Moreover, $\usub$ is a strict subsolution of \eqref{eq-exist:weighted_eq} with nonvanishing gradient in $\Omega_\delta\setminus\Omega_{\delta/2}$. The function
	\begin{equation}\label{eq-exist:usup}
		\usup(x)=\psi_\delta(x)+\frac{r(x)}{r_I}
	\end{equation}
	is a strict supersolution of \eqref{eq-exist:weighted_eq} in the region $\Omega_\delta\setminus\Omega_{-r_I}$.
\end{lemma}
\begin{proof}
	The claim on the sign of $\usub$ follows from $\psi_0(1/2)>1/2$ and the strict convexity of $\psi_0$. In $\Omega_\delta\setminus\Omega_{\delta/2}$ we have
	\[|\D\usub|=\Big|\frac{\p\psi_\delta}{\p r}-\frac1\delta\Big|=\frac1\delta\big|\psi'_0\big(\delta^{-1}r(x)\big)-1\big|=\frac1\delta\Big(\psi'_0(\delta^{-1}r(x))-1\Big)=\frac{\p\psi_\delta}{\p r}-\frac1\delta.\]
	The third equality is because $\psi'_0>1$ on $[1/2,1)$. The same facts imply that $|\D\usub|\ne0$ everywhere in $\Omega_\delta\setminus\Omega_{\delta/2}$. On the other hand, we calculate
	\[\begin{aligned}
		\div\Big(e^{\psi_\delta}\frac{\D\usub}{|\D\usub|}\Big) &= \div\big(e^{\psi_\delta}\p_r\big)=e^{\psi_\delta}\Big(\frac{\p\psi_\delta}{\p r}+\div\p_r\Big)
		> e^{\psi_\delta}\Big(\frac{\p\psi_\delta}{\p r}-\frac1{100r_I}\Big),
	\end{aligned}\]
	where the last inequality comes from \eqref{eq-exist:small_curv}. Hence $\usub$ is a strict subsolution on $\Omega_\delta\setminus\Omega_{\delta/2}$.
	
	Next, inside $\Omega_\delta\setminus\Omega_{-r_I}$ we calculate
	\[\div\Big(e^{\psi_\delta}\frac{\D\usup}{|\D\usup|}\Big)\leq e^{\psi_\delta}\Big(\frac{\p\psi_\delta}{\p r}+\frac1{100r_I}\Big)<e^{\psi_\delta}\Big(\frac{\p\psi_\delta}{\p r}+\frac1{r_I}\Big)=|\D\usup|,\]
	confirming the supersolution property.
\end{proof}

\begin{lemma}[approximating solutions]\label{lemma-exist:reg_solvability}
	For each $\delta<r_I$ and $\lambda\in(3/4,1)$, there exists $\epsilon(\delta,\lambda)>0$ such that the boundary value problem
	\begin{align}[left=\empheqlbrace]
		& \div\Big(e^{\psi_\delta}\frac{\D\ureg}{\sqrt{\epsilon^2e^{2\psi_\delta/(n-1)}+|\D\ureg|^2}}\Big)=e^{\psi_\delta}\sqrt{\epsilon^2e^{2\psi_\delta/(n-1)}+|\D\ureg|^2} \label{eq-exist:reg_bvp}\\
		&\hspace{320pt}\text{in $\Omega_{\lambda\delta}\setminus\bar{E_0}$}, \nonumber\\
		& \ureg=0\qquad\text{on $\p E_0$}, \label{eq-exist:reg_bvp2}\\
		& \ureg=\usub-2\qquad\text{on $\p\Omega_{\lambda\delta}$} \label{eq-exist:reg_bvp3}
	\end{align}
	admits a solution $\ureg\in C^\infty\big(\bar{\Omega_{\lambda\delta}}\setminus E_0\big)$ for all $0<\epsilon\leq\epsilon(\delta,\lambda)$. We have the $C^0$ bounds
	\begin{equation}\label{eq-exist:reg_C0_bounds}
		\max\big\{\!-\epsilon,\psi_\delta(x)-C\big\}\leq\ureg(x)\leq\psi_\delta(x)+C\qquad\forall x\in\Omega_{\lambda\delta}\setminus E_0,
	\end{equation}
	for some constant $C>0$ independent of $\epsilon,\delta,\lambda$. In particular, $-\epsilon\leq\ureg\leq C$ in $\bar\Omega\setminus E_0$.
	
	We also have the gradient estimate
	\begin{equation}\label{eq-exist:reg_gradient}
		|\D\ureg(x)|\leq \sup_{\p E_0\cap B(x,r)}H_++2\epsilon+\frac{C(n)}{r}
	\end{equation}
	for all $x\in\Omega\setminus E_0$ and $0<r\leq\sigma(x;\Omega,g)$, where $H_+=\max\big\{H_{\p E_0},0\big\}$.
\end{lemma}
\begin{proof}
	In $\Omega_\delta$ we consider the conformally transformed metric $g'=e^{2\psi_\delta/(n-1)}g$. The fact $\psi_0\geq(n-1)\log\frac1{1-x}$ in $[1/2,1)$ ensures that $g'$ is a complete metric in $\Omega_\delta$. Let $\usub$ be as in \eqref{eq-exist:usub}. For convenience, we modify $\usub$ inside $\Omega$, so that it is smooth with negative value there (this does not affect any argument below). Then $\usub$ is smooth and proper in $\Omega_\delta$, with $\{\usub<0\}=\Omega_{\delta/2}$. By Lemma \ref{lemma-exist:subsol} and Lemma \ref{lemma-prelim:def_weighted}(3), $\usub$ is a smooth subsolution of IMCF in the region $\big(\Omega_\delta\setminus\bar{\Omega_{\delta/2}},g'\big)$, with nonvanishing gradient there. Finally, note that $\Omega_{\lambda\delta}$ is a sub-level set of $\usub$ (namely, $\Omega_{\lambda\delta}=\{\usub<L\}$ for some $L>2$). Thus we may invoke Theorem \ref{thm-ellreg:existence} on $(\Omega_\delta,g')$ to obtain that: there is $\epsilon(\delta,\lambda)>0$ such that the regularized equation \eqref{eq-ellreg:reg_eq}\,$\sim$\,\eqref{eq-ellreg:reg_eq3} admits a solution for all $\epsilon\leq\epsilon(\delta,\lambda)$. But under the present settings, the regularized equation is exactly \eqref{eq-exist:reg_bvp}\,$\sim$\,\eqref{eq-exist:reg_bvp3}, by Remark \ref{rmk-ellreg:weighted}. This shows the existence of the solution $\ureg$.
	
	To obtain \eqref{eq-exist:reg_gradient}, we note that the gradient estimate in Theorem \ref{thm-ellreg:existence} states that
	\[\big|\D_{g'}\ureg(x)\big|\leq\max\Big\{\sup_{B_{g'}(x,r)\cap\p E_0}H_+,\ \sup_{B_{g'}(x,r)\cap\p\Omega_{\lambda\delta}}|\D_{g'}\ureg|_{g'}\Big\}+2\epsilon+\frac{C(n)}r\]
	for all $x\in\bar{\Omega_{\lambda\delta}}\setminus E_0$ and $r\leq\sigma(x;\Omega_\delta,g')$. Since $g'=g$ inside $\Omega$, and $\sigma(x;\Omega_\delta,g')\geq\sigma(x;\Omega,g)$ for all $x\in\Omega$, and $B_{g'}(x,r)\cap\p\Omega_{\lambda\delta}=\emptyset$ when $r\leq\sigma(x;\Omega,g)$, this gradient estimate directly implies \eqref{eq-exist:reg_gradient}. Next, the lower bound in \eqref{eq-exist:reg_C0_bounds} follows from \eqref{eq-ellreg:lower_bound} and \eqref{eq-exist:usub}. The upper bound is derived as follows: by \eqref{eq-exist:reg_gradient} we have
	\begin{equation}\label{eq-exist:aux5}
		\sup\big\{\ureg(x): x\in\bar{\Omega_{-r_I}}\,\big\}\leq C,
	\end{equation}
	where $C>0$ is independent of $\epsilon,\delta,\lambda$. We compare $\ureg$ with $\usup+C+1$ inside $\bar{\Omega_{\lambda\delta}}\setminus\Omega_{-r_I}$, where $\usup$ is as in \eqref{eq-exist:usup}. On $\p\Omega_{-r_I}$ we have
	\[\ureg\leq C=\usup+C+1.\]
	On $\p\Omega_{\lambda\delta}$ we have
	\[\ureg=\usub-2<\usup<\usup+C+1.\]
	Since $\usup$ is a strict supersolution of \eqref{eq-exist:weighted_eq} in the compact set $\bar{\Omega_{\lambda\delta}}\setminus\Omega_{-r_I}$, by continuity, we may further decrease $\epsilon(\delta,\lambda)$ so that $\usup$ is a strict supersolution of \eqref{eq-exist:reg_bvp} in the same region. By the maximum principle, we obtain $\ureg\leq\usup+C+1$ in $\bar{\Omega_{\lambda\delta}}\setminus\Omega_{-r_I}$. This gives the upper bound along with \eqref{eq-exist:aux5}.
\end{proof}

\noindent\textbf{Convergence to an interior solution.}

The following lemma is used to obtain a candidate solution of $\IVP{\Omega;E_0}$:


\begin{lemma}\label{lemma-exist:convergence}
	For any sequences $\delta_i\to0$, $\lambda_i\to1$, $\epsilon_i\to0$ with $\epsilon_i\leq\epsilon(\delta_i,\lambda_i)$, there exists a sequence of solutions $u_i:=u_{\epsilon_i,\delta_i,\lambda_i}$ of \eqref{eq-exist:reg_bvp}\,$\sim$\,\eqref{eq-exist:reg_bvp3}. Moreover, a subsequence of $u_i$ converges in $C^0_{\loc}(\Omega\!\setminus\!E_0)$ to a function $u\in\Lip_{\loc}(\Omega\!\setminus\!E_0)$ as $i\to\infty$, and $u$ solves $\IVP{\Omega;E_0}$ and is calibrated by some vector field $\nu$ in $\Omega\setminus\bar{E_0}$.
\end{lemma}
\begin{proof}
	The approximate solutions $u_i$ are directly given by Lemma \ref{lemma-exist:reg_solvability}. Applying Theorem \ref{thm-ellreg:convergence} to the data $\Omega_i=\Omega\setminus\bar{E_0}$, $g_i=g$ (note that $\psi_{\delta_i}\equiv0$ in $\Omega$, so the weighted equation is reduced to the ordinary IMCF), we find a subsequence of $u_i$ that converge in $C^0_{\loc}(\Omega\setminus\bar{E_0})$ to a calibrated solution $u\in\Lip_{\loc}(\Omega\setminus\bar{E_0})$. Additionally, by \eqref{eq-exist:reg_gradient}, the functions $u_i$ are uniformly Lipschitz up to $\p E_0$. Therefore, the uniform convergence and Lipschitz regularity of $u$ holds up to $\p E_0$. Then by \eqref{eq-exist:reg_C0_bounds} \eqref{eq-exist:reg_bvp2}, we have $u\geq0$ in $\Omega\setminus E_0$ and $u|_{\p E_0}=0$. Extending $u$ with negative values in $E_0$, we obtain a solution of $\IVP{\Omega;E_0}$.
\end{proof}

\noindent\textbf{The remaining tasks.}

Take any sequence $\delta_i\to0$, $\lambda_i\to1$ and $\epsilon_i\leq\epsilon(\delta_i,\lambda_i)$ with $\epsilon_i\to0$. Let $u$ be the limit solution given by Lemma \ref{lemma-exist:convergence}, and $\nu$ be the corresponding calibration. Note that $\nu$ is not an arbitrary calibration, but one that comes from the use of Theorem \ref{thm-ellreg:convergence} (later we will make use of the statements in Theorem \ref{thm-ellreg:convergence}). We would prove Theorem \ref{thm-exist:main} if we show that $u$ respects the boundary obstacle and satisfies conditions (i)\,$\sim$\,(iv) there.

We note that since Theorem \ref{thm-exist:main} contains a unique statement, it in turn implies that the limit function $u$ in Lemma \ref{lemma-exist:convergence} is independent of the choice of $\delta_i,\lambda_i,\epsilon_i$.

For the clarity of proofs, we split the nontrivial tasks in the following lemmas:

\begin{prop}\label{prop-exist:goal_obs}
	Let $u$, $\nu$ be the solution and calibration obtained in Lemma \ref{lemma-exist:convergence}. Then there is a radius $r_0\in(0,r_I)$ and dimensional constants $\gamma=\gamma(n)\in(0,1)$, $C=C(n)>0$, such that
	\begin{equation}\label{eq-exist:bound_direction}
		1-\metric{\nu}{\p_r}\leq Cr_0^{-\gamma}|r|^\gamma
	\end{equation}
	and
	\begin{equation}\label{eq-exist:bound_du}
		|\D u|\leq Cr_0^{-\gamma}|r|^{\gamma-1}
	\end{equation}
	hold inside $\Omega\setminus\bar{\Omega_{-r_0}}$. In particular, $u$ solves $\IMCFOOinthm{\Omega\setminus\bar{E_0}}{\p\Omega}$.
\end{prop}

\begin{prop}\label{prop-exist:goal_holder}
	$u$ can be extended continuously to $\p\Omega$, and we have $u\in C^{0,\gamma}(\bar\Omega\setminus\Omega_{-r_0/2})$.
\end{prop}

\begin{prop}\label{prop-exist:goal_reg}
	There exists a sufficiently small radius $r_1\in(0,r_0)$ such that: for all $t>0$, the set $\p E_t\setminus\bar{\Omega_{-r_1}}$ is a $C^{1,\gamma/2}$ hypersurface.
\end{prop}

\begin{proof}[Proof of Theorem \ref{thm-exist:main}] {\ }
	
	It is stated in Proposition \ref{prop-exist:goal_obs} that $u$ respects the obstacle $\p\Omega$. The condition $u\in\Lip_{\loc}(\Omega)\cap BV(\Omega)\cap C^{0,\gamma}(\bar\Omega)$ follows by interior regularity and Lemma \ref{lemma-prelim:BV_loc} and Proposition \ref{prop-exist:goal_holder}. The gradient bound \eqref{eq-exist:main_grad_int} follows by taking limit of \eqref{eq-exist:reg_gradient}. The boundary regularity \eqref{eq-exist:main_grad_bd} follows from \eqref{eq-exist:bound_du} inside $\Omega\setminus\Omega_{-r_0}$, and follows from the interior regularity inside $\Omega_{-r_0}$. The bound on calibration follows from \eqref{eq-exist:bound_direction}. Finally, the regularity of level sets follows from Proposition \ref{prop-exist:goal_reg}, and the maximality follows from Corollary \ref{cor-obs:maximality}. The uniqueness of solution follows from maximality.
\end{proof}

\subsection{The spacetime foliation; strategies of the proof}\label{subsec:exist_strategy}

The approximation process described above, which essentially invokes the elliptic regularization in Theorem \ref{thm-ellreg:convergence}, provides an additional set of data. For each $i$, the family of hypersurfaces $\Sigma_t^i:=\graph\big(\epsilon_i^{-1}(u_i-t)\big)$ forms a downward translating soliton of the IMCF in the product domain $\big(\Omega_{\lambda_i\delta_i}\times\RR,\  \exp(\frac{2\psi_{\delta_i}}{n-1})g+dz^2\big)$, by Remark \ref{rmk-ellreg:weighted} and the geometric meaning of elliptic regularization. Equivalently, the function $U_i(x,z)=u_i(x)-\epsilon_iz$ solves the smooth IMCF \eqref{eq-prelim:level_set} in the same region. The proof of Proposition \ref{prop-exist:goal_obs}\,$\sim$\,\ref{prop-exist:goal_reg} is done by obtaining the corresponding estimates on $\Sigma_t^i$, and then passing to the limit.

To distinguish from objects on $M$, we will use bold symbols to denote objects on $M\times\RR$. We denote by $\br(x,z)=r(x)$ the signed distance to $\p\Omega\times\RR$. Note that $\br$ is smooth in $(\Omega_{3r_I}\setminus\bar{\Omega_{-3r_I}})\times\RR$. Then denote the product metric $\bg=g+dz^2$ and radial vector field $\bpr=\D_{\bg}\br$. In the region $(\bar\Omega\setminus\bar{E_0})\times\RR$, let
\begin{equation}\label{eq-exist:def_bnu}
	\bnu_i:=\frac{\D_{\bg}U_i}{|\D_{\bg}U_i|_{\bg}}=\frac{\D u_i-\epsilon_i\p_z}{\sqrt{\epsilon_i^2+|\D u_i|^2}}
\end{equation}
be the downward unit normal vector field of the foliation $\Sigma_t^i$ (we remind that $\psi_{\delta_i}=0$ in $\bar\Omega$, so the weighted IMCF is reduced to the usual one). According to Theorem \ref{thm-ellreg:convergence}, the vector fields $\bnu_i$ converge to some $\bnu$ in the weak sense in $L^1_{\loc}(\Omega\setminus\bar{E_0})$, and the projection of $\bnu$ onto the $\Omega$ factor is the calibration $\nu$ given in Lemma \ref{lemma-exist:convergence}.

The above propositions are proved via the following steps:

1. We show that when $i\to\infty$, we have $\inf_{\p\Omega\times\RR}\metric{\bnu_i}{\bpr}_{\bg}\to1$ (Lemma \ref{lemma-exist:tangent_at_pOmega}). The proof uses a blow-up argument: if there is a sequence of exceptional points $x_i\in\p\Omega\times\RR$, then we may rescale the solutions $u_i$ near $x_i$ and obtain an exceptional blow-up limit. The limit function will be a weak solution of
\[\div\Big(e^{\psi_0(x_n)}\frac{\D u}{|\D u|}\Big)=e^{\psi_0(x_n)}|\D u|.\]
However, Theorem \ref{thm-halfplane:approx} implies that this limit must have the form $u=\psi(x_n)-C$. Then, standard geometric measure theory is used to show that the space-time graphs $\Sigma_t^i$ converge to a hyperplane at $x_i$ in $C^1$ sense, contradicting our hypotheses for bad points.

2. Let $r_i$ be the largest radius so that $\metric{\bnu_i}{\bpr}_{\bg}\geq1/2$ holds inside $(\Omega\setminus\Omega_{-r_i})\times\RR$. We show that $r_i$ is uniformly bounded below (Lemma \ref{lemma-exist:graphical}). If this is false, then we find a subsequence of radii $r_i\to0$ and exceptional points $x_i\in\p\Omega_{-r_i}\times\RR$, with $\metric{\bnu_i}{\bpr}_{\bg}(x_i,0)=1/2$. Using the result of step 1 and the parabolic estimate in Lemma \ref{lemma-para:box_qeta}, we will obtain the bounds
\begin{equation}\label{eq-exist:aux6}
	1-\metric{\bnu_i}{\bpr}_{\bg}\leq C\big(o(1)+|r|/r_i\big)^\gamma
\end{equation}
for some uniform constants $C,\gamma$. Then consider a blow-up sequence centered at $x_i$. The limit is an exceptional weak solution $u'$ on $\{x_n<0\}\subset\RR^{n+1}$. The bound \eqref{eq-exist:aux6} passes to the limit and implies that $u'$ respects the obstacle $\{x_n=0\}$, due to Corollary \ref{cor-obs:bd_orthogonal_2}. This eventually contradicts Theorem \ref{thm-halfplane:liouville}, by looking at the convergence of the space-time graphs $\Sigma_t^i$ (this is more involved compared to the previous step).

3. Combining the above two steps, we eventually obtain \eqref{eq-exist:bound_direction} through another application of parabolic estimate (Lemma \ref{lemma-para:box_qeta}). The boundary gradient estimate \eqref{eq-exist:bound_du} and H\"older regularity follows by applying Lemma \ref{lemma-para:box_etaH}. Finally, Proposition \ref{prop-exist:goal_reg} follows by an $\epsilon$-regularity theorem for almost perimeter-minimizers (in the sense of \eqref{eq-exist:almost_min} below).

\subsection{\texorpdfstring{Proof of Propositions \ref{prop-exist:goal_obs}\,$\sim$\,\ref{prop-exist:goal_reg}}{Proof of Propositions 6.5-6.7}}\label{subsec:exist_proofs}

We assume all the setup and notations made in subsection \ref{subsec:exist_setup} and \ref{subsec:exist_strategy}.

\begin{lemma}\label{lemma-exist:tangent_at_pOmega}
	For any $\th<1$, there exists $i_0\in\NN$ such that for all $i\geq i_0$ we have
	\[\inf_{\p\Omega\times\RR}\metric{\bnu_i}{\bpr}_{\bg}\geq\th.\]
\end{lemma}
\begin{proof}
	Suppose that the conclusion does not hold. Then passing to a subsequence (which we do not relabel), we can find a constant $\th_0<1$ and a sequence of points $q_i=(x_i,z_i)\in\p\Omega\times\RR$, such that $\metric{\bnu_i}{\bpr}_{\bg}(x_i,z_i)\leq\th_0$. Since $\bnu_i$ is invariant under vertical translation, we may assume $z_i=0$. Passing to a further subsequence, we may assume $x_i\to x_0\in\p\Omega$.
	
	Consider the domains
	\[\Omega_i:=B_g(x_i,\sqrt{\delta_i})\cap\Omega_{\lambda_i\delta_i},\qquad\bOmega_i:=\Omega_i\times(-1,1)\]
	with the rescaled and conformally transformed metrics
	\[h_i:=\delta_i^{-2}e^{2\psi_{\delta_i}/(n-1)}g,\qquad\bh_i:=h_i+\delta_i^{-2}dz^2.\]
	Set the normalized function $u'_i(x):=u_i(x)-u_i(x_i)$. By the equation \eqref{eq-exist:reg_bvp}, Remark \ref{rmk-ellreg:weighted} and the scaling invariance of IMCF, the functions $U'_i(x,z):=u'_i(x)-\epsilon_iz$ are smooth solutions of the IMCF in $(\bOmega_i,\bh_i)$.
	
	To properly state the blow-up process, we define some suitable coordinate maps. In each tangent space $T_{x_i}M$, fix an orthonormal frame $\{e_i\}$ such that $e_1,\cdots,e_{n-1}$ are tangential to $\p\Omega$, and $e_n=\p_r$. With respect to this frame, the $g$-exponential map at $x_i$ (denoted by $\exp_{x_i}$) is a diffeomorphism from a small ball in $\RR^n$ to its image. Define the maps
	\[\Phi_i:x\mapsto\exp_{x_i}(\delta_ix),\qquad \bPhi_i:(x,z)\mapsto\big(\!\exp_{x_i}(\delta_ix),\delta_iz\big),\]
	and set the preimages\footnote{The rule for setting up the notations: boldface $\to$ objects in $M\times\RR$; tilde $\to$ objects in the pulled-back spaces; prime $\to$ normalized functions (so that the value at $x_i$ is zero).} $\tilde\Omega_i:=\Phi_i^{-1}(\Omega_i)\subset\RR^n$, $\tilde\bOmega_i:=\bPhi_i^{-1}(\bOmega_i)=\tilde\Omega_i\times(-\delta_i^{-1},\delta_i^{-1})$. Also, set the pullbacks $\tilde h_i:=\Phi_i^*h_i$, $\tilde\bh_i:=\bPhi_i^*\bh_i=\tilde h_i+dz^2$, and $\tilde u'_i:=u'_i\circ\Phi_i$, $\tilde U'_i:=U'_i\circ\bPhi_i$. Note that $\tilde U'_i(0)=U'_i(x_i,0)=0$ and $\tilde U'_i(x,z)=\tilde u'_i(x)-\epsilon_i\delta_iz$. By diffeomorphism invariance, $\tilde U'_i$ are smooth solutions of IMCF in $(\tilde\bOmega_i,\tilde\bh_i)$. Therefore, $\tilde u'_i$ solves the $(\epsilon_i\delta_i)$-regularized equation
	\[\div\left(\frac{\D_{\tilde h_i}\tilde u'_i}{\big(\epsilon_i^2\delta_i^2+|\D_{\tilde h_i}\tilde u'_i|^2\big)^{1/2}}\right)=\sqrt{\epsilon_i^2\delta_i^2+|\D_{\tilde h_i}\tilde u'_i|^2}\qquad\text{in }(\tilde\Omega_i,\tilde h_i).\]
	
	Next, we take limit of the rescaled objects. Note that $\tilde\Omega_i\to\{x_n<1\}\subset\RR^n$ locally, since $\delta_i\to0$ and $\lambda_i\to1$. Also, $\tilde h_i\to\tilde h:=e^{2\psi_0(x_n)/(n-1)}(dx_1^2+\cdots+dx_n^2)$ locally smoothly. Apply Theorem \ref{thm-ellreg:convergence} with the data $\tilde\Omega_i$, $\tilde h_i$ and $\tilde u'_i$. As a result, there is a subsequence such that $\tilde u'_i$ converge in $C^0_{\loc}$ to a function $\tilde u'\in\Lip_{\loc}(\{x_n<1\})$, and $\tilde u'$ solves $\IMCF{\{x_n<1\},\tilde h}$. By Lemma \ref{lemma-prelim:def_weighted}(3), $\tilde u'$ is a weak solution of the weighted IMCF
	\[\div\Big(e^{\psi_0(x_n)}\frac{\D_0\tilde u'}{|\D_0\tilde u'|}\Big)=e^{\psi_0(x_n)}|\D_0\tilde u'|,\]
	where $\D_0$ denotes the Euclidean gradient. The gradient estimate \eqref{eq-prelim:grad_est_limit} implies
	\[|\D_0\tilde u'(x)|=|\D_{\tilde h}\tilde u'(x)|\leq\frac{C(n)}{\sigma\big(x;\{x_n<1\},\tilde h\big)}\leq\frac{C(n)}{|x_n|},\qquad\forall\,x\in\{x_n<0\}.\]
	Furthermore, the $C^0$ bounds in Lemma \ref{lemma-exist:reg_solvability} give (recall $u'_i=u_i-u_i(x_i)$)
	\[u'_i\geq\psi_{\delta_i}-2C\ \ \Rightarrow\ \ \tilde u'_i(x)\geq\psi_{\delta_i}\big(\!\exp_{x_i}(\delta_ix)\big)-2C=\psi_0(x_n)+o(1)-2C.\]
	This passes to the limit and gives the bound $\tilde u'(x)\geq\psi_0(x_n)-2C$. Now all the assumptions of Theorem \ref{thm-halfplane:approx} are met, and we obtain $\tilde u'(x)=\psi_0(x_n)-C'$ for some other $C'$.
	
	Set $\tilde U'(x,z):=\tilde u'(x)=\psi_0(x_n)-C'$, which is the $C^0_{\loc}$ limit of $\tilde U'_i$. The sub-level sets $E_0(\tilde U'_i)$ have locally uniformly bounded perimeter, by Lemma \ref{lemma-prelim:BV_loc}(i), Thus passing to a further sequence, we may assume $E_0(\tilde U'_i)\to E$ in $L^1_{\loc}$, for some set $E\subset\big\{(x,z):x_n<1\big\}$. Since $\tilde U'_i\to\tilde U'$ in $C^0_{\loc}$, we have (see Remark \ref{rmk-prelim:level_set_convergence})
	\begin{equation}\label{eq-exist:aux2}
		E_0(\tilde U')\subset E\subset E_0^+(\tilde U')
	\end{equation}
	up to zero measure. Next, by Corollary \ref{cor-ellreg:grad_est} and the nice convergence of $\tilde\bh_i$, the gradients $|\D_{\tilde\bh_i}\tilde U'_i|$ are uniformly bounded in $B_{\RR^{n+1}}(0,1/2)$, for all large $i$. Then by \eqref{eq-prelim:Lam_r0_min}, $E_0(\tilde U'_i)$ are uniform $(\Lambda,r_0)$-perimeter minimizers in $\big(B_0(0,1/2),\tilde\bh_i\big)$. By Lemma \ref{lemma-gmt:Lam_r0_convergence}(i) and the fact $0\in\p E_0(\tilde U'_i)$, we have $0\in\spt(|D\chi_E|)$. Combined with \eqref{eq-exist:aux2} and the expression $\tilde U'=\psi(x_n)-C$, we see that there are only two possibilities:
	
	\vspace{2pt}
	
	(1) $C'>0$ and $E=\{x_n<0\}$, or
	
	(2) $C'=0$, $\tilde U'(x,z)=\psi_0(x_n)$, and $E\subset\{x_n<0\}$.
	
	\vspace{2pt}
	
	\noindent Suppose the second case holds. By Theorem \ref{thm-prelim:compactness}, the limit set $E$ locally minimizes the energy $J_{\tilde U'}$ in $\big\{(x,z):x_n<1\big\}$. As $\tilde U'\equiv0$ in $\big\{(x,z):x_n<0\big\}$, this implies that $E$ is locally inward-minimizing. On the other hand, $E$ is an outward minimizer of the weighted perimeter $\int_{\p^*E}e^{\psi_0}$. By direct verification, this implies that $E$ is locally outward perimeter-minimizing in $\RR^{n+1}$. As a result, $E$ is locally perimeter-minimizing in $\RR^{n+1}$. By Lemma \ref{lemma-gmt:min_in_halfplane} and the facts $E\subset\{x_n<0\}$, $0\in\spt(|D\chi_E|)$, we have $E=\{x_n<0\}$.
	
	So in either case we conclude that $E=\{x_n<0\}$. Then by Lemma \ref{lemma-gmt:Lam_r0_convergence}(ii), we have
	\[\nu_{E_0(\tilde U'_i)}(0)\to\nu_E(0)=\p_{x_n},\]
	where $\nu_{E_0(\tilde U'_i)}$ is the outer unit normal of $E_0(\tilde U'_i)$ with respect to $\tilde\bh_i$. On the other hand, we may evaluate by unraveling the pullbacks
	\[\big\langle\nu_{E_0(\tilde U'_i)},\p_{x_n}\big\rangle_{\tilde\bh_i}(0)=\metric{\bnu_i}{\bpr}_{\bg}(x_i,0).\]
	This contradicts our initial assumption $\metric{\bnu_i}{\bpr}_{\bg}(x_i,0)\leq\th_0$, thus proving the lemma.
\end{proof}

\begin{lemma}\label{lemma-exist:graphical}
	There exists $r_0<r_I$ and $i_0\in\NN$, such that for all $i\geq i_0$ we have
	\[\inf_{(\Omega\setminus\Omega_{-r_0})\times\RR}\metric{\bnu_i}{\bpr}_{\bg}\geq\frac12.\]
\end{lemma}
\begin{proof}
	By Lemma \ref{lemma-exist:tangent_at_pOmega}, we have $\inf_{\p\Omega\times\RR}\metric{\bnu_i}{\bpr}_{\bg}\geq3/4$ for all sufficiently large $i$. Suppose that the lemma is false. Then passing to a subsequence, we may find radii $r_i\in(0,r_I)$, $r_i\to0$, and a sequence of points $q_i=(x_i,z_i)\in\p\Omega_{-r_i}\times\RR$, such that
	\begin{equation}\label{eq-exist:to_contradict}
		\metric{\bnu_i}{\bpr}_{\bg}(x_i,z_i)\leq\frac12.
	\end{equation}
	By translation invariance, we may assume $z_i=0$. By re-selecting each $x_i$ to have the smallest distance to $\p\Omega$ (and decreasing $r_i$ correspondingly), we can further assume that
	\begin{equation}\label{eq-exist:smallest_dist}
		\metric{\bnu_i}{\bpr}_{\bg}\geq\frac12 \quad\text{inside}\quad (\Omega\setminus\Omega_{-r_i})\times\RR.
	\end{equation}
	Finally, passing to a further subsequence, we may assume $x_i\to x_0\in\p\Omega$. By Lemma \ref{lemma-exist:tangent_at_pOmega}, there exist numbers $\th_i\to1$ such that
	\begin{equation}\label{eq-exist:result_on_bd}
		\inf_{\p\Omega\times\RR}\metric{\bnu_i}{\bpr}_{\bg}\geq\th_i.
	\end{equation}
	
	\hypertarget{lemma:6.9step1}{\textbf{Step 1.}} We establish a uniform bound on $\bnu_i$ by parabolic estimates.
	
	Recall that $\Sigma_t^i=\graph\big(\epsilon_i^{-1}(u_i-t)\big)$ forms a downward translating soliton of the usual IMCF in $(\bar\Omega\setminus\bar{E_0})\times\RR$. We apply Lemma \ref{lemma-para:box_qeta} to the flow $\Sigma_t^i$, with $\Omega$ replaced by $\Omega\times\RR$ and with the choice $r_0=r_i$. The small curvature condition \eqref{eq-para:small_curv} is implied by \eqref{eq-exist:small_curv}, and the condition $p\geq\frac12$ there comes from \eqref{eq-exist:smallest_dist}. Choosing the parameter $b_i=(1-\th_i)^{1/\gamma}$ there, it follows that the quantity $F=\big(1-\metric{\bnu_i}{\bpr}_{\bg}\big)(b_i-r/r_i)^{-\gamma}$ satisfies
	\[\big(1-\metric{\bnu_i}{\bpr}_{\bg}\big)^{\frac2\gamma-1}H^2\Box F+\metric{\DSigma F}{X}\leq \frac{C_1}{r_i^2}\big(\!-F^{\frac{\gamma+2}\gamma}+C_2\big)\]
	inside $(\Omega\setminus\bar{\Omega_{-r_i}})\times\RR$, where $\Box=\p_t-H^{-2}\Delta_\Sigma$ is the heat operator associated to the IMCF, and the subscripts $\Sigma$ are shorthands for $\Sigma_t^i$. Since $\Sigma_t^i$ evolves as translating soliton in the $z$-direction, we have $\p_t F=\metric{\DSigma F}{Y}$ for some smooth vector field $Y$. Hence
	\[-\big(1-\metric{\bnu_i}{\bpr}_{\bg}\big)^{\frac2\gamma-1}\Delta_{\Sigma} F+\metric{\DSigma F}{Z}\leq\frac{C_1}{r_i^2}\big(\!-F^{\frac{\gamma+2}\gamma}+C_2\big),\]
	for some $\gamma\in(0,1/2)$ and $C_1,C_2>0$ depending only on $n$. Since $\Sigma_t^i$ is graphical, the intersection $\Sigma_t^i\cap\big((\bar\Omega\setminus\Omega_{-r_I})\times\RR\big)$ is compact. By the maximum principle and invariance in the $z$-direction, we obtain
	\[\begin{aligned}
		\max_{(\bar\Omega\setminus\Omega_{-r_i})\times\RR}(F) &\leq \max\Big\{\max_{\p\Omega\times\RR}(F),\ \max_{\p\Omega_{-r_i}\times\RR}(F),\ C_2^{\frac{\gamma}{\gamma+2}}\Big\} \\
		&\leq\max\Big\{1,\ \frac12,\ C_2^{\frac{\gamma}{\gamma+2}}\Big\}.
	\end{aligned}\]
	From this we have
	\begin{equation}\label{eq-exist:aux1}
		1-\metric{\bnu_i}{\bpr}_{\bg}\leq C_3(b_i-\br/r_i)^\gamma\qquad\text{in}\ \ (\bar\Omega\setminus\Omega_{-r_i})\times\RR.
	\end{equation}
	
	\textbf{Step 2.} We establish an inward minimizing property of the level sets of $U_i$, where recall $U_i(x,z)=u_i(x)-\epsilon_iz$. Denote $\bOmega=\Omega\times\RR$. Fix $t\in\RR$, and consider $E=E_t(U_i)\cap\bOmega$. Suppose $F\Subset K\Subset(M\setminus\bar{E_0})\times\RR$. Using the boundary condition \eqref{eq-exist:result_on_bd}, the fact $\nu_E=\bnu_i$ in $\bOmega$, and the smooth IMCF equation $\div(\bnu_i)=|\D_{\bg}U_i|$, we evaluate by the divergence theorem as follows. In the expressions, we write $|\cdot|=\H^{n-1}(\cdot)$.
	\[\begin{aligned}
		&\quad\ \ \P{E;\bOmega\cap K}-\P{E\!\setminus\!F;\bOmega\cap K}+\th_i\Big(\big|\p^*E\cap\p\bOmega\cap K\big|-\big|\p^*(E\!\setminus\!F)\cap\p\bOmega\cap K\big|\Big) \\
		&\leq \int_{\p^*E\cap\bOmega\cap K}\metric{\nu_E}{\bnu_i}_{\bg}
			- \int_{\p^*(E\setminus F)\cap\bOmega\cap K}\metric{\nu_{E\setminus F}}{\bnu_i}_{\bg}
			+ \int_{\p\bOmega\cap K}\big(\chi_{\p^*E}-\chi_{\p^*(E\setminus F)}\big)\metric{\nu_{\bOmega}}{\bnu_i}_{\bg} \\
		&= \int_{E\cap F}|\D_{\bg}U_i|_{\bg}.
	\end{aligned}\]
	Next, by the coarea formula we have
	\[\int_{E\cap F}|\D_{\bg}U_i|_{\bg}=\int_{\inf_{F\cap\bOmega}(U_i)}^t\big|\p^*E_s(U_i)\cap\bOmega\cap F\big|\,ds.\]
	By a perimeter decomposition and Gronwall argument, as in Lemma \ref{lemma-prelim:excess_ineq}, these imply
	\[\begin{aligned}
		\P{E;\bOmega\cap K}+\th_i\big|\p^*E\cap\p\bOmega\cap K\big| &\leq \P{E\!\setminus\!F;\bOmega\cap K}+\th_i\big|\p^*(E\!\setminus\!F)\cap\p\bOmega\cap K\big| \\
		&\qquad\quad +\int_{\inf_{F\cap\bOmega}(U_i)}^t e^{t-s}\P{F;E_s(U_i)\cap\bOmega}\,ds.
	\end{aligned}\]
	Notice that $\Ps{E;K}=\Ps{E;\bOmega\cap K}+|\p^*E\cap\p\bOmega\cap K|$, since $E\subset\bOmega$. So the above inequality clearly implies (where we have plugged in the definition of $E$):
	\begin{equation}\label{eq-exist:aux_excess_ineq}
		\th_i\P{E_t(U_i)\cap\bOmega;K}\leq\P{(E_t(U_i)\cap\bOmega)\setminus F;K}+\big(e^{t-\inf_{F\cap\bOmega}(U_i)}-1\big)\Ps{F}.
	\end{equation}
	
	We remark that, by the spirit of Theorem \ref{thm-obs:auto_subsol}, outward-minimizing properties are sort of automatic. On the other hand, the inward-minimizing property is closely related to the boundary condition. This explains why \eqref{eq-exist:result_on_bd} is invoked in this step.
	
	\vspace{3pt}
	
	\textbf{Step 3.} We are ready for the blow-up argument. Let $y_i\in\p\Omega$ be the unique point with the smallest distance from $x_i$, thus $d(x_i,y_i)=r_i$. Consider the domains
	\[\Omega_i:=B_g(y_i,\sqrt{r_i})\cap\Omega,\qquad\bOmega_i:=\Omega_i\times(-1,1),\]
	with the metrics
	\[h_i:=r_i^{-2}g,\qquad \bh_i:=r_i^{-2}\big(g+dz^2\big).\]
	
	Define $u'_i(x):=u_i(x)-u_i(x_i)$, $U'_i(x,z):=u'_i(x)-\epsilon_iz$. Near $y_i$ we set up a geodesic normal coordinate in the same way as in Lemma \ref{lemma-exist:tangent_at_pOmega}. Let $\exp_{y_i}$ be the induced exponential map; thus we have $\exp_{y_i}(-r_ie_n)=x_i$. Consider the scaled coordinate maps
	\[\Phi_i: x\mapsto\exp_{y_i}(r_ix),\qquad \bPhi_i: (x,z)\mapsto(\exp_{y_i}(r_ix), r_iz).\]
	Denote $\tilde\Omega_i:=\Phi_i^{-1}(\Omega_i)\subset\RR^n$ and $\tilde\bOmega_i:=\bPhi_i^{-1}(\bOmega_i)=\tilde\Omega_i\times(-r_i^{-1},r_i^{-1})\subset\RR^{n+1}$, equipped with the metrics and functions
	\[\begin{aligned}
		& \tilde h_i=\Phi_i^*h_i,\qquad\ \ \ \tilde\bh_i:=\bPhi_i^*\bh_i=\tilde h_i+dz^2, \\
		& \tilde u'_i:=u'_i\circ\Phi_i,\qquad \tilde U'_i:=U'_i\circ\bPhi_i.
	\end{aligned}\]
	So $\tilde\Omega_i$ is the intersection of $B_{\RR^n}(0,1/\sqrt{r_i})$ with a set that locally approximates $\{x_n<0\}$ as $i\to\infty$. Also, notice that $\Phi_i^{-1}(\p\Omega_{-r_i})$ locally approaches $\{x_n=-1\}$ when $i\to\infty$. Let us view $\tilde u'_i$, $\tilde U'_i$ as functions defined only on $\tilde\Omega_i$, $\tilde\bOmega_i$ (thus we are discarding the values outside these domains).
	
	As usual, denote $e_n=(0,\cdots,0,1)\in\RR^n$. Further denote $\be_n=(e_n,0)\in\RR^{n+1}$. Note the following facts: $x_i=\Phi_i(-e_n)$, and $\tilde U'_i(x,z)=\tilde u'_i(x)-\epsilon_ir_iz$, so $\tilde U'_i(-\be_n)=0$. Also, $\tilde U'_i$ solves the smooth IMCF in $(\tilde\bOmega_i,\tilde\bh_i)$, by scaling and diffeomorphism invariance.
	
	We have $\tilde\bOmega_i\to\{(x,z): x_n<0\}\subset\RR^{n+1}$ and $\tilde\bh_i\to\beuc:=dx_1^2+\cdots+dx_{n+1}^2$ locally smoothly. By Corollary \ref{cor-ellreg:grad_est}, we have the gradient estimate
	\begin{equation}\label{eq-exist:scaled_grad_est}
		\big|\D_{\tilde\bh_i}\tilde U'_i\big|_{\tilde\bh_i}(x,z)\leq\frac{C(n)}{|x_n|}+o(1),
	\end{equation}
	where the $o(1)$ term is uniform in all compact sets and goes to zero as $i\to\infty$.
	
	By Theorem \ref{thm-ellreg:convergence} (applied with the data $\tilde\bOmega_i$, $\tilde\bh_i$, $\tilde U'_i$), up to a subsequence, we have $\tilde U'_i\to\tilde U'$ in $C^0_{loc}$ for some $\tilde U'\in\Lip_{\loc}(\{(x,z): x_n<0\})$, and the calibrations $\tilde\bnu_i:=\D_{\tilde\bh_i}\tilde U'_i/|\D_{\tilde\bh_i}\tilde U'_i|$ converge to some $\tilde\bnu$ in the weak $L^1_{\loc}$ topology. Finally, $\tilde\bnu$ calibrates $\tilde U'$ as a solution of $\IMCF{\{x_n<0\};\beuc}$.
	
	Taking limit of \eqref{eq-exist:scaled_grad_est} and the $C^0$ bounds in Lemma \ref{lemma-exist:reg_solvability}, we have
	\begin{equation}\label{eq-exist:obs_aux1}
		|\D_{\tilde\bh_0}\tilde U'|\leq\frac{C(n)}{|x_n|},\qquad
		\tilde U'(x,z)\geq-C.
	\end{equation}
	For $(x,z)\in\tilde\bOmega_i$ with $-1\leq x_n<0$, we may calculate by unraveling the pullbacks and using the asymptotics of the exponential map:
	\begin{equation}\label{eq-exist:exp_asymp}
		\big\langle \tilde\bnu_i,\,\p_{x_n}\big\rangle_{\tilde\bh_i}(x,z)=\metric{\bnu_i}{\bpr}_{\bg}(\Phi_i(x),r_iz)+o(1),
	\end{equation}
	where the $o(1)$ term locally uniformly converges to zero as $\delta\to0$. By \eqref{eq-exist:aux1} this implies
	\begin{equation}\label{eq-exist:exp_asymp_2}
		\big\langle \tilde\bnu_i,\,\p_{x_n}\big\rangle_{\tilde\bh_i}(x,z)\geq 1-C_3\big(b_i-r_i^{-1}r(\Phi_i(x))\big)^\gamma-o(1).
	\end{equation}
	Note that $\lim_{i\to\infty}r_i^{-1}r(\Phi_i(x))=x_n$ and $\lim_{i\to\infty}b_i=0$. Thus taking the limit $i\to\infty$, we obtain
	\begin{equation}\label{eq-exist:obs_aux2}
		\big\langle \tilde\bnu,\,\p_{x_n}\big\rangle_{\beuc}(x,z)\geq 1-C_3|x_n|^\gamma\qquad\text{a.e. when}\ \ -1<x_n<0.
	\end{equation}
	Then by Corollary \ref{cor-obs:bd_orthogonal_2}, $\tilde U'$ actually solves $\IMCFOO{\{x_n<0\},\beuc}{\{x_n=0\}}$. Then applying Theorem \ref{thm-halfplane:liouville} with \eqref{eq-exist:obs_aux1} with $\tilde U'(-\be_n)=0$, it follows that $\tilde U'\equiv0$.
	
	\vspace{3pt}
	
	\textbf{Step 4.} We argue similarly as in Lemma \ref{lemma-exist:tangent_at_pOmega}, to obtain a contradiction from the level sets. Recall $-\be_n\in\p E_0(\tilde U'_i)$. Applying Lemma \ref{lemma-prelim:BV_loc}(i) to $\tilde\bOmega_i\subset\RR^{n+1}$, we have the uniform bound (recall that we are treating $E_0(\tilde U'_i)$ as subsets of $\tilde\bOmega_i$)
	\[\P{E_0(\tilde U'_i);K}\leq\P{\tilde\bOmega_i\cap K},\qquad\forall K\Subset\RR^{n+1}.\]
	Thus, a subsequence of $E_0(\tilde U'_i)$ converges to some $E\subset\big\{(x,z):x_n<0\big\}$ in $L^1_{\loc}(\RR^{n+1})$.
	
	The gradient bound \eqref{eq-exist:scaled_grad_est} and \eqref{eq-prelim:Lam_r0_min} imply the uniform $(\Lambda,r_0)$-minimizing effect
	\begin{equation}\label{eq-exist:Lam_r0}
		\P{E_0(\tilde U'_i);B(-\be_n,\frac12)}\leq\Ps{F;B(-\be_n,\frac12)}+C\big|E_0(\tilde U'_i)\Delta F\big|,
	\end{equation}
	whenever $E_0(\tilde U'_i)\Delta F\Subset B(-\be_n,1/2)$. Then by Lemma \ref{lemma-gmt:Lam_r0_convergence}(i), we obtain $(-e_n,0)\in\spt(|D\chi_E|)$. By Theorem \ref{thm-prelim:compactness}, $E$ is a local minimizer of $J_{\tilde U'}$ in $\big\{(x,z):x_n<0\big\}$, and since $\tilde U'$ is shown to be constant, $E$ is actually a local perimeter minimizer in $\big\{(x,z):x_n<0\big\}$ (with respect to the Euclidean metric). At this point, the minimizing is unknown to hold up to the boundary. However, one may show that $E$ is locally outward-minimizing in $\RR^{n+1}$, by an approximation argument as in Theorem \ref{thm-obs:auto_subsol}.
	
	If we can show that $E$ is locally inward perimeter-minimizing in $\RR^{n+1}$, then Lemma \ref{lemma-gmt:min_in_halfplane} implies $E=\{(x,z):x_n<-1\}$, and then Lemma \ref{lemma-gmt:Lam_r0_convergence}(ii) implies the convergence of normal vectors
	\[\tilde\bnu_i(-\be_n)=\nu_{E_0(\tilde U'_i)}(-\be_n)\to\p_{x_n}.\]
	This would contradict \eqref{eq-exist:to_contradict}, by \eqref{eq-exist:exp_asymp}, thus proving the lemma. It remains to show the (non-trivial) inward-minimizing of $E$.
	
	For each fixed $R$, we denote $Q_R=\{|x'|\leq R, |x_n|\leq R, |z|\leq R\}\subset\RR^{n+1}$. Let us prove the following almost constancy conclusion: for each $R\geq2$ we have
	\begin{equation}\label{eq-exist:min_aux}
		\lim_{i\to\infty}\min\Big\{\tilde U'_i(x,z): (x,z)\in Q_R\cap\tilde\bOmega_i\Big\}=0.
	\end{equation}
	Recall the direction bound \eqref{eq-exist:smallest_dist} and the definition $\bnu_i=\D_{\bg}U'_i/|\D_{\bg}U'_i|_{\bg}$. Pulling back to $\tilde\bOmega_i$, these imply that
	\[\frac{\p\tilde U'_i}{\p x_n}\geq0\quad\text{in}\quad\Big\{\!-1/2\leq x_n<0, |x'|\leq R, |z|\leq R\Big\}\cap\tilde\bOmega_i.\]
	for sufficiently large $i$. Therefore, the minimum in \eqref{eq-exist:min_aux} is attained at a point with $-R\leq x_n\leq-1/2$, for each large $i$. Since $\tilde U'_i\to\tilde U'\equiv0$ uniformly on the compact set $\big\{|x'|\leq R, -R\leq x_n\leq-1/2, |z|\leq R\big\}$, our claim \eqref{eq-exist:min_aux} immediately follows.
	
	Now we pull back \eqref{eq-exist:aux_excess_ineq} via $\bPhi_i$. In this process we notice that: $U_i$ and $U'_i$ differs only by additive constants, so \eqref{eq-exist:aux_excess_ineq} holds for $\tilde U'_i$ as well. Also, the inequality \eqref{eq-exist:aux_excess_ineq} is invariant under scaling. Moreover, notice a subtlety in notations: in our setting, $E_0(\tilde U'_i)$ is the pull back of $E_0(U'_i)\cap\bOmega$ (i.e. the intersection with $\bOmega$ is already done via $\bPhi_i$). So for any $F\Subset Q_R\Subset B_{\RR^{n+1}}(0,1/\sqrt{r_i})$, we have
	\[\th_i\P{E_0(\tilde U'_i);Q_R}\leq\P{E_0(\tilde U'_i)\setminus F;Q_R}+\big(e^{-\inf_{F\cap\tilde\bOmega_i}(\tilde U'_i)}-1\big)\Ps{F},\]
	where the perimeters are with respect to the rescaled and pulled-back metric $\tilde\bh_i$. Taking $i\to\infty$, using \eqref{eq-exist:min_aux}, $\th_i\to1$, and the convergence $\tilde\bh_i\to\beuc$ in $C^0_{\loc}(\RR^{n+1})$, we obtain by the standard set replacing argument
	\[\Ps{E;Q_R}\leq\Ps{E\setminus F;Q_R},\]
	where the perimeters are with respect to the Euclidean metric. This proves the inward minimizing of $E$ in $\RR^{n+1}$, thus completes the proof.
\end{proof}

\begin{proof}[Proof of Proposition \ref{prop-exist:goal_obs}] {\ }
	
	Combining Lemma \ref{lemma-exist:tangent_at_pOmega} and \ref{lemma-exist:graphical}, by taking a subsequence we can assume the following: there is a radius $r_0<r_I$ such that
	\begin{equation}\label{eq-exist:previous_steps}
		\inf_{\p\Omega\times\RR}\metric{\bnu_i}{\bpr}_{\bg}\geq1-\frac1i,\qquad
		\inf_{(\Omega\setminus\Omega_{-r_0})\times\RR}\metric{\bnu_i}{\bpr}_{\bg}\geq\frac12,\qquad\text{for all }i.
	\end{equation}
	Apply Lemma \ref{lemma-para:box_qeta} with the domain $\Omega\times\RR$ in place of $\Omega$, and with the smooth solutions $\Sigma_t^i=\graph\big(\epsilon_i^{-1}(u_i-t)\big)$ and the parameter $b=i^{-1/\gamma}$. The small curvature and $p\geq\frac12$ condition there are satisfied by \eqref{eq-exist:small_curv} \eqref{eq-exist:previous_steps}. Arguing verbatim as in \hyperlink{lemma:6.9step1}{Step 1} of Lemma \ref{lemma-exist:graphical}, we obtain
	\begin{equation}\label{eq-exist:approx_tangency}
		1-\metric{\bnu_i}{\bpr}_{\bg}\leq C_3\big(i^{-1/\gamma}-r/r_0\big)^\gamma\qquad\text{in }(\Omega\setminus\bar{\Omega_{-r_0}})\times\RR.
	\end{equation}
	We next apply Lemma \ref{lemma-para:box_etaH} in the region $(\Omega\setminus\bar{\Omega_{-r_0}})\times\RR$. Consider the quantity
	\[G=r_0\Big((i^{-1/\gamma}-r/r_0)^{1-\gamma}-(2i^{-1/\gamma})^{1-\gamma}\Big)H,\]
	where $H$ is the mean curvature of $\Sigma_t^i$. The conditions in Lemma \ref{lemma-para:box_etaH} are satisfied by \eqref{eq-exist:small_curv} \eqref{eq-exist:previous_steps} \eqref{eq-exist:approx_tangency} respectively. By the result \eqref{eq-para:box_etaH} and the fact that $\Sigma_t^i$ forms a translating soliton, we obtain the inequality
	\[-\frac1{H^2}\Delta_\Sigma G+\metric{\DSigma G}{Y}\leq -\frac Gn+\frac{4n}G+C_4\big(i^{-1/\gamma}-r/r_0\big)^{-\gamma}\big(\frac{C_5}G-1\big)\]
	in $(\Omega\setminus\bar{\Omega_{-r_0}})\times\RR$, for some $C_4,C_5>0$ independent of $i$. By the maximum principle, this implies
	\begin{align}
		\max_{(\bar{\Omega_{-br_0}}\setminus\Omega_{-r_0})\times\RR}(G) &\leq \max\Big\{\max_{\p\Omega_{-r_0}\times\RR}(G),\ \max_{\p\Omega_{-br_0}\times\RR}(G),2n+C_5\Big\} \nonumber\\
		&\leq \max\Big\{C(n),0,2n+C_5\Big\}. \label{eq-exist:bound_G}
	\end{align}
	Here, the control of $G$ on $\p\Omega_{-r_0}\times\RR$ follows by Corollary \ref{cor-ellreg:grad_est} and the assumptions for $r_I$:
	\[G(x)\leq r_0\cdot2\cdot H(x)\leq\frac{2C(n+1)r_0}{\sigma(x;\Omega\times\RR,\bg)}\leq\frac{2C(n+1)r_0}{\sigma(x;\Omega,g)}\leq C,\quad x\in\p\Omega_{-r_0}\times\RR.\]
	So \eqref{eq-exist:bound_G} implies
	\begin{equation}\label{eq-exist:approx_grad}
		\sqrt{\epsilon_i^2+|\D u_i|^2}=H\leq\frac{C(n)r_0^{-1}}{(i^{-1/\gamma}-r/r_0)^{1-\gamma}-(2i^{-1/\gamma})^{1-\gamma}}.
	\end{equation}
	
	As $i\to\infty$, we pass \eqref{eq-exist:approx_tangency} to the limit and project onto the $\Omega$ factor, and obtain
	\begin{equation}\label{eq-exist:aux4}
		1-\metric{\nu}{\p_r}\leq C_3\big(|r|/r_0\big)^\gamma\qquad\text{a.e. in}\ \ \Omega\setminus\Omega_{-r_0},
	\end{equation}
	see below \eqref{eq-exist:def_bnu} for the convergence to $\nu$. Passing \eqref{eq-exist:approx_grad} to the limit, we obtain
	\[|\D u|\leq C(n)r_0^{-\gamma}|r|^{\gamma-1}\qquad\text{in }(\Omega\setminus\bar{\Omega_{-r_0}})\times\RR.\]
	Finally, by Corollary \ref{cor-obs:bd_orthogonal_2} and \eqref{eq-exist:aux4}, the solution $u$ respects the obstacle $\p\Omega$.
\end{proof}

\begin{proof}[Proof of Proposition \ref{prop-exist:goal_holder}] {\ }
	
	Let $g_0$ be the product metric on $\p\Omega\times(0,r_0)$, and $\Phi:\p\Omega\times(0,r_0)\to\Omega\setminus\bar{\Omega_{-r_0}}$ be the normal exponential map. By \eqref{eq-exist:small_curv}, we have $\frac12g_0\leq\Phi^*g\leq2g_0$. By \eqref{eq-exist:bound_du}, (pulling back via $\Phi$,) we may view $u=u(y,r)$ as a function on $\p\Omega\times(0,r_0)$ with $|\D_{g_0}u|\leq Cr_0^{-\gamma}r^{\gamma-1}$.
	
	For $0<r_1,r_2\leq r_0/2$, from the gradient bound we have
	\begin{equation}\label{eq-exist:holder1}
		\big|u(y,r_1)-u(y,r_2)\big|\leq C\gamma^{-1}r_0^{-\gamma}|r_1^\gamma-r_2^\gamma|
		\leq C\gamma^{-1}r_0^{-\gamma}|r_1-r_2|^\gamma.
	\end{equation}
	Note that this implies $u\in L^\infty\big(\p\Omega\times(0,r_0/2)\big)$. Next, for $r\leq r_0/2$ and $d_{\p\Omega}(y_1,y_2)\leq r$, we have
	\begin{equation}\label{eq-exist:holder2}
		\big|u(y_1,r)-u(y_2,r)\big|\leq Cr_0^{-\gamma}r^{\gamma-1}d_{\p\Omega}(y_1,y_2)\leq Cr_0^{-\gamma}d_{\p\Omega}(y_1,y_2)^\gamma.
	\end{equation}
	Finally, suppose $r\leq r_0/2$ and $d_{\p\Omega}(y_1,y_2)\geq r$. If $d_{\p\Omega}(y_1,y_2)\geq r_0/2$, then we have
	\begin{equation}\label{eq-exist:holder3}
		\big|u(y_1,r)-u(y_2,r)\big|\leq 2\|u\|_{L^\infty(\p\Omega\times(0,r_0/2))}\cdot(r_0/2)^{-\gamma}\cdot d_{\p\Omega}(y_1,y_2)^\gamma.
	\end{equation}
	If $r\leq d_{\p\Omega}(y_1,y_2)\leq r_0/2$, then we set $s=d_{\p\Omega}(y_1,y_2)$ and estimate
	\begin{align}
		\big|u(y_1,r)-u(y_2,r)\big| &\leq \big|u(y_1,r)-u(y_1,s)\big|+\big|u(y_2,r)-u(y_2,s)\big|+\big|u(y_1,s)-u(y_2,s)\big| \nonumber\\
		&\leq 2C\gamma^{-1}r_0^{-\gamma}|s-r|^\gamma+Cr_0^{-\gamma}s^{\gamma-1}d_{\p\Omega}(y_1,y_2) \nonumber\\
		&\leq (2C\gamma^{-1}+C)r_0^{-\gamma}d_{\p\Omega}(y_1,y_2)^\gamma. \nonumber
	\end{align}
	Combined with \eqref{eq-exist:holder1}\,$\sim$\,\eqref{eq-exist:holder3}, it follows that $u$ can be extended to a $C^{0,\gamma}$ function on $\p\Omega\times[0,r_0/2]$. Finally, by the smallness of $r_0$, there is a constant $C$ such that $d_{\p\Omega\times[0,r_0/2]}(x,y)\leq Cd_g(\Phi(x),\Phi(y))$. This shows that $u$ extends to a $C^{0,\gamma}$ function on $\bar\Omega\setminus\Omega_{-r_0/2}$.
\end{proof}

\begin{proof}[Proof of Proposition \ref{prop-exist:goal_reg}] {\ }
	
	For any $x\in\Omega\setminus\Omega_{-r_0/3}$ and $r<r_0/3$, by \eqref{eq-exist:bound_du} and the choice of $r_I$ we have
	\begin{equation}\label{eq-exist:du_integral}
		\begin{aligned}
			\int_{B(x,r)\cap\Omega}|\D u| &\leq \int_{\max\{0,|r(x)|-r\}}^{|r(x)|+r}\H^{n-1}\big(B(x,r)\cap\Omega_{-s}\big)Cr_0^{-\gamma}s^{\gamma-1}\,ds \\
			&\leq C'r_0^{-\gamma}r^{n-1+\gamma}.
		\end{aligned}
	\end{equation}
	It is already known that $u$ respects the obstacle $\p\Omega$. For all $t>0$ and each competitor set $E\subset M$ satisfying $E\Delta E_t\Subset B(x,r)$, we compare the energy $\tJ_u^{B(x,r)}(E_t)\leq\tJ_u^{B(x,r)}(E\cap\Omega)$ and obtain (note that $E_t\Delta(E\cap\Omega)\Subset B(x,r)$ since $E_t\subset\Omega$)
	\[\begin{aligned}
		\P{E_t;B(x,r)} &\leq \P{E\cap\Omega;B(x,r)}+\int_{E_t\Delta(E\cap\Omega)}|\D u| \\
		&\leq \P{E;B(x,r)}+\P{\Omega;B(x,r)}-\P{E\cup\Omega;B(x,r)}+Cr_0^{-\gamma}r^{n-1+\gamma}.
	\end{aligned}\]
	It is directly verifiable, using the $C^2$ smoothness of $\p\Omega$ (see \cite[Section 1.6]{Tamanini_1984}), that
	\[\P{\Omega;B(x,r)}\leq \P{E\cup\Omega;B(x,r)}+Cr_0^{-2}r^{n+1}.\]
	Thus we obtain the almost-minimizing condition
	\begin{equation}\label{eq-exist:almost_min}
		\P{E_t;B(x,r)}\leq\P{E;B(x,r)}+Cr_0^{-\gamma}r^{n-1+\gamma}.
	\end{equation}
	
	Next, combining \eqref{eq-exist:bound_direction} and Remark \ref{rmk-prelim:calibration_normal_vec}, we obtain that almost every $E_t$ satisfies
	\begin{equation}
		1-\metric{\nu_{E_t}}{\p_r}\leq C_3r_0^{-\gamma}|r|^\gamma.
	\end{equation}
	In addition, on $\p^*E_t\cap\p^*\Omega$ it is clear that $\nu_{E_t}=\p_r$. Since $\p_r$ is smooth, this has the following consequence: for any $l>0$ there exists a sufficiently small radius $r(l)$, such that for all $x\in\bar\Omega\setminus\Omega_{-r(l)}$, $s\leq r(l)$, the set $E_t\cap B(x,s)$ is representable as the sub-graph of a $l$-Lipschitz function in some geodesic normal coordinate near $x$. Combining \eqref{eq-exist:almost_min} and this small slope condition, the classical small excess regularity theorem (see \cite{Tamanini_1983} and \cite{Almgren_1976, Tamanini_1984}) implies the following: there is a sufficiently small $r_1<r_0$, such that almost every $E_t\setminus\bar{\Omega_{-r_1}}$ is a $C^{1,\gamma/2}$ hypersurface. By the Arzela-Ascoli theorem, the same conclusion holds for all $t>0$, as desired.
\end{proof}

\section{On maximal solutions}\label{sec:max}

Let $E_0\subset M$ be an initial value. Recall that $u$ is a \textit{maximal solution} of $\IVP{M;E_0}$ if $u\geq v$ on $M\setminus E_0$ for any other solution $v$ of $\IVP{M;E_0}$. In this section, we prove the existence and basic properties of such a solution.

\subsection{Existence and properties}

The following is a restatement of Theorem \ref{thm-intro:max_sol}. In the proof we note that: since the limit solution $u$ is unique, it (a posteriori) does not depend on the choice of $E_0^i$ and $\Omega_i$.

\begin{theorem}[existence]\label{thm-max:existence} {\ }
	
	Let $(M,g)$ be a (possibly incomplete) smooth Riemannian manifold, and $E_0\subset M$ be a (possibly unbounded) $C^{1,1}$ domain. Then there exists, up to equivalence, a unique maximal solution of $\IVP{M;E_0}$, in the sense of Definition \ref{def-prelim:ivp}.
\end{theorem}
\begin{proof}
	Find a sequence of precompact $C^{1,1}$ domains $E_0^i$, such that $E_0^1\subset E_0^2\subset\cdots\subset E_0$, and $\bigcup_{i\geq1}E_0^i=E_0$, and that for all $K\Subset M$, the sets $E_0^i\cap K$ eventually stabilize (i.e. $E_0^i\cap K=E_0^{i+1}\cap K=E_0^{i+2}\cap K=\cdots$ for large enough $i$). Next, find a sequence of smooth precompact domains $\Omega_i\Supset E_0^{i+1}$ such that $\Omega_1\Subset\Omega_2\Subset\cdots\Subset M$ and $\bigcup_{i\geq1}\Omega_i=M$. Let $u_i$ be the (unique) solution of $\IVPOO{\Omega_i;E_0^i}{\p\Omega_i}$, given by Theorem \ref{thm-exist:main}.
	
	Applying Corollary \ref{cor-obs:maximality} to each $\Omega_i$, we have $u_1\geq u_2\geq u_3\geq\cdots$ outside $E_0$. By the interior gradient estimates \eqref{eq-exist:main_grad_int} and Arzela-Ascoli theorem, some subsequence of $u_i$ converges in $C^0_{\loc}$ to a function $u\in\Lip_{\loc}(M\setminus\bar{E_0})$. By Theorem \ref{thm-prelim:cptness_calibrated}, $u$ is a calibrated weak solution in $M\setminus\bar{E_0}$. Since \eqref{eq-exist:main_grad_int} provides a uniform gradient estimate up to $\p E_0$, the resulting function $u$ is Lipschitz up to $\p E_0$, and satisfies $u\geq0$, $u|_{\p E_0}=0$. Extending $u$ by negative values inside $E_0$, we obtain a solution of $\IVP{M;E_0}$.
	
	It remains to show that $u$ is maximal. Suppose $v\in\Lip_{\loc}(M)$ is another solution of $\IVP{M;E_0}$. For each $i$ we find functions $v_i\in\Lip_{\loc}(\Omega_i)$, such that $v_i=v$ on $\Omega_i\setminus E_0$, and $v_i=0$ on $E_0\setminus E_0^i$, and $v_i<0$ on $E_0^i$. By Definition \ref{def-prelim:weak_sol}(ii) it can be verified that $v_i$ is a subsolution of $\IVP{\Omega_i;E_0^i}$. Then by Corollary \ref{cor-obs:maximality}, we obtain $v_i\leq u_i$ on $\Omega_i\setminus E_0^i$, hence $v\leq u_i$ on $\Omega_i\setminus E_0$. As $u$ is the descending limit of $u_i$, it follows that $v\leq u$ on $M\setminus E_0$.
\end{proof}

The following theorem summarizes the useful properties of maximal solutions.

\begin{theorem}[properties]\label{thm-max:properties} {\ }
	
	Let $M,g,E_0$ be as in Theorem \ref{thm-max:existence}, and $u$ be the maximal solution of $\IVP{M;E_0}$ given there. Then the following hold.
	\begin{enumerate}[label={(\roman*)}, nosep]
		\item We have the interior gradient estimate
		\[|\D u|(x)\leq\sup_{\p E_0\cap B(x,r)}H_++\frac{C(n)}r,\qquad x\in M\setminus E_0,\ \ r\leq\sigma(x;M,g),\]
		where $H_+$ denotes the positive part of the mean curvature of $\p E_0$, and $\sigma(x;M,g)$ is as in Definition \ref{def-prelim:sigma_x}.
		\item If $E_0$ is connected, then $E_t$ is connected for all $t>0$.
		\item If $E_0\Subset M$, then $\Ps{E_t}\leq e^t\Ps{E_0}$.
	\end{enumerate}
\end{theorem}
\begin{proof}
	By the proof of Theorem \ref{thm-max:existence}, the maximal solution $u$ arises as a limit of the solutions $u_i$ of $\IVPOO{\Omega_i;E_0^i}{\p\Omega_i}$, where $E_0^i$ and $\Omega_i$ are precompact exhaustions of $E_0$ and $M$ as stated in Theorem \ref{thm-max:existence}. Moreover, $u$ does not depend on the choice of these exhaustions (since it is a unique object).
	
	Given this setup, item (i) follows by passing Theorem \ref{thm-exist:main}(i) to the limit. Item (iii) follows by passing Corollary \ref{cor-obs:sub_exp} to the limit and using the lower semi-continuity of perimeter. It remains to prove (ii). Given that $E_0$ is connected, we claim that there is a sequence of \textit{connected} precompact $C^{1,1}$ domains $E_0^1\subset E_0^2\subset\cdots\subset E_0$ with $\bigcup E_0^i=E_0$. Once such an exhaustion is found, Lemma \ref{lemma-obs:connectedness} implies that each $\bar{E_t(u_i)}$ is connected. Then note that each $E_t(u_i)$ is connected as well: otherwise, since $u_i\in C^0(\bar{\Omega_i})$ by Theorem \ref{thm-exist:main}, there will be a slightly smaller $t'$ so that $\bar{E_{t'}(u_i)}$ is disconnected, contradiction. Finally, as $E_t(u)=\bigcup_{i\geq1}E_t(u_i)$, it follows that $E_t(u)$ is connected.
	
	The connected exhaustion $\{E_0^i\}$ is constructed as follows. Fix a basepoint $x_0\in E_0$. We start with an arbitrary precompact $C^{1,1}$ exhaustion $\bar E_0^1\subset\bar E_0^2\subset\cdots\subset E_0$, then let $E_0^i$ be the connected component of $\bar E_0^i$ containing $x_0$. It follows that $\{E_0^i\}$ is also an exhaustion: for any $x\in E_0$, there is a path $\gamma\Subset E_0$ joining $x$ and $x_0$. We have $\gamma\Subset\bar{E_0^i}$ for all large $i$, hence $\gamma\Subset E_0^i$ as well, hence $x\in E_0^i$.
\end{proof}

When $\Omega$ is a non-compact domain, the fact that $u$ solves $\IVPOO{\Omega;E_0}{\p\Omega}$ does not imply that it is the maximal solution of $\IVP{\Omega;E_0}$. Intuitively, the maximal solution respects not only the obstacle $\p\Omega$, but also the ``invisible obstacle'' at infinity. Note that item (iii) may be a strict inequality, whereas we always have equality for proper solutions. Thus, a portion of the perimeter may be lost at infinity for non-proper solutions.

\vspace{3pt}
	
We mention the following examples of maximal solutions.
	
1. In Choi-Daskalopoulos \cite{Choi-Daskalopoulos_2021}, one considers the (smooth) IMCF with a convex, non-compact, $C^{1,1}$ initial domain $E_0\subset\RR^n$. A compact approximation argument is employed to obtain a solution: one finds bounded convex domains $E_0^1\subset E_0^2\subset\cdots$ with $\bigcup_{i\geq1}E_0^i=E_0$, and then takes $u_i$ to be the proper (hence maximal) solution of $\IVP{\RR^n;E_0^i}$, and finally take the descending limit $u=\lim_{i\to\infty}u_i$. Arguing similarly as above, it follows that $u$ is the maximal solution of $\IVP{\RR^n;E_0}$.

2. In the setting of Lemma \ref{lemma-prelim:sphere_sym}, the function
\[u_0(r)=(n-1)\log\big[f(r)/f(r_0)\big]\]
is the maximal solution of $\IVP{\Omega;E_0}$.
	
3. Consider $E_0$ a half-space in the hyperbolic space, and let $u$ be the maximal solution of $\IVP{\HH^n;E_0}$. By uniqueness, the function $u|_{\HH^n\setminus E_0}$ must be invariant under any isometry that preserves $E_0$. Therefore, the level sets of $u$ are equi-distance sets from $E_0$. This reduces the flow to a one-dimension problem: the maximal solution must be
\[u=(n-1)\log\cosh d(x,E_0).\]

\begin{remark}\label{rmk-max:connectedness}
	The connectedness of level sets (Theorem \ref{thm-max:properties}(ii)) only holds for maximal solutions, but not general solutions. A concrete counterexample is as follows. Let $E_0,E'\subset\HH^n$ be two disjoint half-spaces that are sufficiently far separated. Denote $l_1=\arccosh\big(e^{1/(n-1)}\big)$ and $l_2=\arccosh\big(e^{2/(n-1)}\big)$. Define $u\in\Lip_{\loc}(\HH^n)$ such that
	\begin{enumerate}[label={(\arabic*)}, nosep]
		\item $u|_{E_0}<0$, and $u|_{N(E_0,l_2)}=(n-1)\log\cosh d(\cdot,E_0)$;
		\item $u|_{E'}\equiv1$, and $u|_{N(E',l_1)}=1+(n-1)\log\cosh d(\cdot,E')$;
		\item $u|_{\HH^n\setminus(N(E_0,l_2)\cup N(E',l_1))}\equiv2$.
	\end{enumerate}
	\vspace{1pt}
	It follows that $u$ solves $\IVP{\HH^n;E_0}$, but $E_2(u)$ is not connected.
	\begin{figure}[htbp]
		\centering
		\vspace{-24pt}
		\includegraphics[scale=1.4]{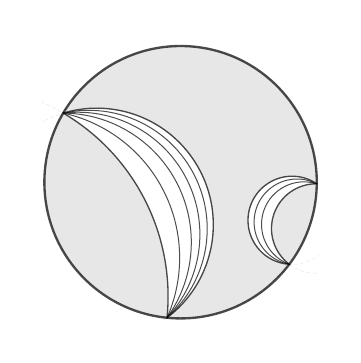}
		\begin{picture}(0,0)
			\put(-186,102){$E_0$}
			\put(-198,86){$(u<0)$}
			\put(-55,93){$E'$}
			\put(-34,86){$u\equiv1$}
			\put(-106,161){$u\equiv2$}
		\end{picture}
		\vspace*{-24pt}
		\caption{An example of disconnected sub-level set.}
	\end{figure}
\end{remark}

\begin{remark}
	By the outward perimeter-minimizing property, each $M\setminus E_t$ ($t>0$) has no compact connected components. When $M$ satisfies certain topological properties (for example, when $M$ has one end and $H_1(M,\ZZ)=0$), Theorem \ref{thm-max:properties}(ii) and this will together imply that $\p E_t$ is connected as well. This observation was already made in \cite[Lemma 4.2]{Huisken-Ilmanen_2001} for proper solutions.
\end{remark}

\begin{remark}
	When $E_0\Subset M$, a sufficient condition for the maximal solution to be nontrivial (i.e. $u\not\equiv0$) is
	\[\begin{aligned}
		&\text{there is $A>\Ps{E_0}$ and $K\Subset M$, so that} \\
		&\hspace{108pt} \text{$\Ps{E}>A$ for any $C^1$ domain $E$ with $K\Subset E\Subset M$.}
	\end{aligned}\]
	Indeed, if the maximal solution is trivial, then the sequence of solutions in the proof of Theorem \ref{thm-max:existence} satisfies $u_i\to0$ in $C^0_{\loc}(M\setminus E_0)$. So for any $\delta>0$ and $K\Subset M$, we may find a large enough $i$ with $u_i|_K<\delta$. Therefore, we find a bounded $C^1$ domain $E=E_\delta(u_i)\Supset K$ that satisfies $\Ps{E}\leq e^\delta\Ps{E_0}$. This is exactly the converse of the above statement.
\end{remark}

\subsection{Properness and isoperimetry revisited}

The purpose of this subsection is to re-prove the properness criterion \cite[Theorem 1.2]{Xu_2023_proper}. The proof presented here essentially captures the technical core \cite[Lemma 4.4]{Xu_2023_proper}, while the conic cutoff argument in \cite{Xu_2023_proper} is replaced here by Theorem \ref{thm-exist:main}. This proof is conceptually cleaner (but not actually simpler since it relies on Theorem \ref{thm-exist:main}). Define the isoperimetric profile of a Riemannian manifold:
\[\ip(v)=\inf\Big\{|\p^*E|: E\Subset M\text{ has finite perimeter, }|E|=v\Big\}.\]

\begin{theorem}[identical with {\cite[Theorem 1.2]{Xu_2023_proper}}]\label{thm-exist:properness} {\ }
	
	Suppose $(M,g)$ is a complete, connected, non-compact Riemannian manifold with infinite volume and with dimension $n\geq2$, such that the isoperimetric profile satisfies
	\begin{equation}\label{eq-max:nondeg_ip}
		\liminf_{v\to\infty}\ip(v)=\infty
		\ \ \ \text{and}\ \ \ 
		\int_0^{v_0}\frac{dv}{\ip(v)}<\infty\,\ \text{for some}\ v_0>0.
	\end{equation}
	Then for any $C^{1,1}$ domain $E_0\Subset M$, the maximal solution of $\IVP{M;E_0}$ is proper.
\end{theorem}
\begin{proof}
	Define the upper inverse isoperimetric profile
	\[\sip^{-1}(a)=\sup\Big\{|E|: E\Subset M\text{ has finite perimeter, }|\p^*E|\leq a\Big\}.\]
	Clearly $a\mapsto\sip^{-1}(a)$ is non-decreasing, and is finite for all $a>0$ by \eqref{eq-max:nondeg_ip}. Under \eqref{eq-max:nondeg_ip}, it is proved in \cite[Subsection 3.1]{Xu_2023_proper} that $v\mapsto1/\ip(v)$ is locally bounded in $(0,\infty)$.
	
	Fix an exhaustion $E_0\Subset\Omega_1\Subset\Omega_2\Subset\cdots$, with $\Omega_i$ smooth and $\bigcup_{i\geq1}\Omega_i=M$. Let $u_i$ be the solution of $\IVPOO{\Omega_i;E_0}{\p\Omega_i}$ given by Theorem \ref{thm-exist:main}. Fix a basepoint $x_0\in E_0$, and denote $B(r)=B(x_0,r)$. For each $i\in\NN$, $t>0$, and for almost every $r>\diam E_0$ (so that $\p B(r)$ is rectifiable), we apply Lemma \ref{lemma-halfplane:excess_ineq_2} with $F=\Omega_{i+1}\setminus B(r)$ to find
	\[\P{E_t(u_i)}\leq\P{E_t(u_i)\cap B(r)}+(e^t-1)\P{B(r);E_t(u_i)}.\]
	Using the decomposition identities \eqref{eq-gmt:set_operation}, this further implies
	\[\H^{n-1}\big(E_t(u_i);B(r)^{(0)}\big)\leq e^t\P{B(r);E_t}\qquad\text{for a.e. }r>\diam E_0.\]
	Consider $V(r):=|E_t(u_i)\setminus B(r)|$, so $V'(r)=-\P{B(r);E_t(u_i)}$ for a.e. $r$. We also have
	\[\H^{n-1}\big(E_t(u_i);B(r)^{(0)}\big)+\P{B(r);E_t}=\P{E_t(u_i)\setminus B(r)}\geq\ip(V(r))\]
	for almost every $r$. Combining all these inequalities, we obtain
	\begin{equation}\label{eq-max:ODE_ineq}
		V'(r)\leq-\frac{\ip(V(r))}{e^t+1}\qquad\text{for a.e. }r>\diam E_0.
	\end{equation}
	In addition, $V(\diam E_0)\leq|E_t(u_i)|\leq\sip^{-1}\big(\Ps{E_t(u_i)}\big)\leq\sip^{-1}\big(e^t\Ps{E_0}\big)$, where we used Corollary \ref{cor-obs:sub_exp}. Hence solving the ODE inequality \eqref{eq-max:ODE_ineq}, we obtain
	\[V(r)\equiv0\quad\text{for all}\ \ r>r_*(t):=\diam E_0+(1+e^t)\int_0^{\sip^{-1}(e^t\Ps{E_0})}\frac{dv}{\ip(v)}.\]
	This implies $E_t(u_i)\subset B(r_*(t))$ for all $t>0$, which is a bound independent of $i$. By \eqref{eq-max:nondeg_ip} and the observations at the beginning of the proof, we have $r_*(t)<\infty$.
	
	Now we take the descending limit $u=\lim_{i\to\infty}u_i$. Arguing as in Theorem \ref{thm-max:existence}, such limit exists and is the maximal solution of $\IVP{M;E_0}$. The bounds $E_t(u_i)\subset B(r_*(t))$ pass to the limit and give $E_t(u)\subset B(r_*(t))\Subset M$. Hence $u$ is proper.
\end{proof}

\appendix

\section{Soliton examples in \texorpdfstring{$\RR^2$}{R2}}\label{sec:ex}

In this section, we collect several smooth solitons of IMCF in the plane. We also refer to the previous works \cite{Castro-Lerma_2017,Chang_2020, Drugan-Lee-Wheeler_2016, Kim-Pyo_2019, Kim-Pyo_2020}. In these works, it was already noticed that planar solitons of IMCF are usually incomplete. However, due to the nature of singularities at their endpoints, the soliton curves contact tangentially with the boundary of the region that they sweep out. Therefore, the solutions that they generate would respect the boundary obstacle. This observation provides us the first class of nontrivial explicit examples.

We expect that the solitons presented here, as well as their analogues in more general circumstances, appear as the blow-down limits at infinity or the blow-up limits at non-smooth boundary points. See Question \ref{qs-ex:asymp_expanding}, \ref{qs-ex:asymp_shrinking} for more details. Another role of the examples is to provide insights in forming the main theory in Section \ref{sec:obs} and \ref{sec:halfplane}.

The following setup are assumed. Suppose $\gamma\subset\RR^2$ is a smooth curve with nonvanishing curvature, with unit speed parametrization $s$. Let $\tau:=d\gamma/ds$ be its unit tangential vector, and $\th$ be the angle function so that $\tau=(\cos\th,\sin\th)$, and $\nu:=(\sin\th,-\cos\th)$ be the normal vector, finally $\kappa:=d\th/ds$ be the curvature. Set the support function $h=x\cdot\nu$, where $x$ is the position vector. The following formula is well-known:
\begin{equation}\label{eq-ex:support}
	\gamma=h\nu+\frac{dh}{d\th}\tau,
\end{equation}
where $dh/d\th=\frac{dh/ds}{d\th/ds}=\frac{dh/ds}{\kappa}$ is a well-defined quantity.

\begin{example}[cycloids] {\ }
	
	By \cite[Theorem 8]{Drugan-Lee-Wheeler_2016}, any translating soliton $\gamma$ of IMCF in $\RR^2$ is isometric to a cycloid. Indeed, after a scaling and rigid motion, we may assume that $\gamma$ is translating in the positive $x$-direction with speed 1. Then the soliton equation reads $\frac1\kappa=\metric{\nu}{\p_x}=\sin\th$. Solving this (see \cite{Drugan-Lee-Wheeler_2016}), we obtain a cycloid equation
	\begin{equation}\label{eq-ex:cycloid_eq}
		\gamma(\th)=\frac14\big(1-\cos2\th,2\th-\sin2\th\big),\qquad \th\in(0,\pi).
	\end{equation}
	The curve $\gamma$ is embedded into the strip region $\Omega=\RR\times(0,\pi/2)$, and its endpoints $\gamma(0)$, $\gamma(\pi)$ lie on $\p\Omega$. By Taylor expansion, near $\gamma(0)$ the curve is the graph of
	\begin{equation}\label{eq-ex:taylor_exp}
		y=\frac{2\sqrt 2}3x^{3/2}+O(x^{5/2}),\qquad x>0.
	\end{equation}
	A similar result holds near $\gamma(\pi)$ by symmetry. By Corollary \ref{cor-obs:bd_orthogonal_2}, the translations of $\gamma$ generate a solution of $\IMCFOO{\Omega}{\p\Omega}$ (which is actually smooth).
	
	In the solution generated, each sub-level set $E_t$ is the unbounded region enclosed by the translation $\gamma+(t,0)$ and the two lines $(-\infty,t)\times\{0\}$, $(-\infty,t)\times\{\pi/2\}$. By the expansion \eqref{eq-ex:taylor_exp}, each $\p E_t$ is globally a $C^{1,1/2}$ curve.
	
	\vspace{-6pt}\TODO{make sure good placement of the figure.}
	
	\begin{figure}[h]
		\setlength{\abovecaptionskip}{-9pt}
		\captionsetup{width=0.9\textwidth}
		\begin{center}
			\includegraphics[scale=0.55]{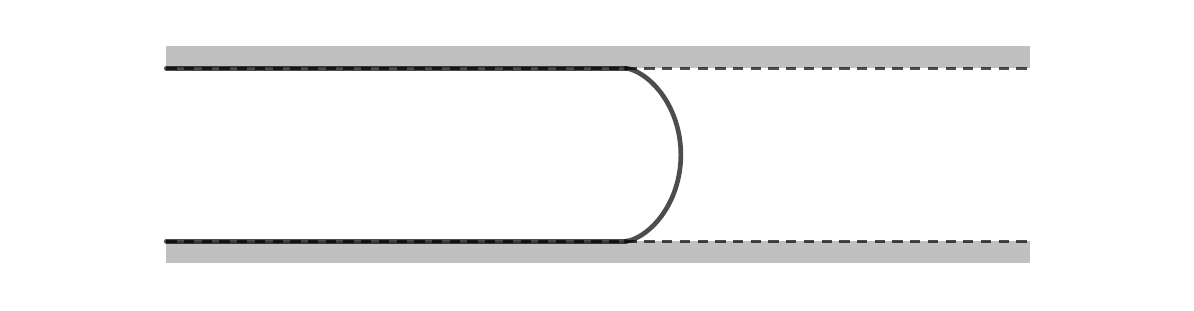}
		\end{center}
		\caption{Translating cycloids in a strip region. The bold curve represents $\p E_t$.}
	\end{figure}
\end{example}

\begin{remark}[homothetic solitons]
	Let $c\in\RR\setminus\{0\}$. If a curve $\gamma$ satisfies
	\begin{equation}\label{eq-ex:homo_sol}
		h\kappa=c^{-1},
	\end{equation}
	then the family of curves $\gamma_t=e^{ct}\gamma$ solves the IMCF. Urbas \cite{Urbas_1999} computed that \eqref{eq-ex:homo_sol} is equivalent to
	\begin{equation}\label{eq-ex:dhdth}
		\frac{d^2h}{d\th^2}=(c-1)h.
	\end{equation}
	For a full classification result based on this equation, see \cite{Castro-Lerma_2017, Drugan-Lee-Wheeler_2016, Urbas_1999}. Furthermore, we refer to \cite{Chang_2020} for a classification of more complex solitons with simultaneous scaling and rotating behaviors.
	
	We are interested in the case $c<1$. Solving \eqref{eq-ex:dhdth}, we obtain up to a normalization $h(\th)=\sin\big(\sqrt{1-c}\,\th\big)$. Since the curve $\gamma$ becomes singular when $h(\th)=0$, due to \eqref{eq-ex:homo_sol}, we restrict the parameter to a single period $\th\in(0,\pi/\sqrt{1-c})$. Then we insert the expression into \eqref{eq-ex:support} and obtain the curve $\gamma(\th)$ in complex form:
	\begin{equation}\label{eq-ex:gen_sol}
		\gamma(\th)=-\sin\big(\sqrt{1-c}\,\th\big)ie^{i\th}+\sqrt{1-c}\cos\big(\sqrt{1-c}\,\th\big)e^{i\th}.
	\end{equation}
	Setting $\mu=1-\sqrt{1-c}$ and $k=\frac{1+\sqrt{1-c}}{1-\sqrt{1-c}}$, we can rewrite \eqref{eq-ex:gen_sol} as
	\begin{equation}\label{eq-ex:gen_sol_2}
		\gamma(\th)=\frac12\mu\big(ke^{i\mu\th}-e^{ik\mu\th}\big).
	\end{equation}
	Comparing with the classical equations \cite[Section 6.3, 6.5]{Lawrence_curves}, we find that \eqref{eq-ex:gen_sol_2} describes a hypocycloid ($c<0$) or epicycloid ($c>0$). This was previously observed in \cite{Castro-Lerma_2017, Drugan-Lee-Wheeler_2016}.
	\TODO{adjust spacing}
\end{remark}

\vspace{-36pt}
			
\begin{figure}[h]
	\setlength{\abovecaptionskip}{-12pt}
	\captionsetup{width=0.8\textwidth}
	\hspace{-196pt}\includegraphics[scale=1.52]{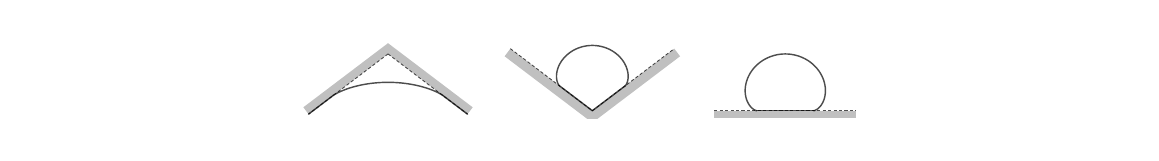}
	\caption{Shrinking hypocycloids and expanding epicycloids, expanding nephroids. These correspond to Examples \ref{ex-ex:hypocycloid}, \ref{ex-ex:epicycloid}, \ref{ex-ex:nephroid} respectively.}\label{fig_ex:solitons}
\end{figure}
			
\begin{example}[$c<0$, shrinking hypocycloids]\label{ex-ex:hypocycloid} {\ }
				
	Assume $c<0$, and set $T=\pi/\sqrt{1-c}$. From \eqref{eq-ex:gen_sol} we have $\gamma(0)=\sqrt{1-c}$ and $\gamma(T)=\sqrt{1-c}\,e^{i\pi/\sqrt{1-c}-i\pi}$. Hence $\gamma$ is supported in a planar cone $\Omega$ with angle $\big(1-\frac1{\sqrt{1-c}}\big)\pi$. Now the family of curves $\gamma_t=e^{ct}\gamma$ forms a smooth solution of IMCF in $\Omega$. By the fact $\lim_{\th\to0}h(\th)=\lim_{\th\to T}h(\th)=0$, the normal vector of $\gamma$ converges to the outer normal vector of $\p\Omega$ when $\th\to0,T$. Then by Corollary \ref{cor-obs:bd_orthogonal_2} and scaling invariance, the solution generated by $\gamma$ respects the obstacle $\p\Omega$. In such a solution, the sub-level set $E_t$ is the unbounded region enclosed by $\gamma_t$ and the two radial lines emanating from $\gamma_t(0)$ and $\gamma_t(T)$. The boundary of $E_t$ is a $C^{1,1/2}$ curve, by similar calculations as above.
				
	This provides examples of solutions of IMCF that approach $+\infty$ near $\p\Omega$.
\end{example}
			
\begin{example}[$c>0$, expanding epicycloids]\label{ex-ex:epicycloid} {\ }
				
	Assume $c>0$. Similarly as above, the family of curves $\gamma_t=e^{ct}\gamma$ form a solution of IMCF in a cone domain $\Omega$ with angle $\big(\frac1{\sqrt{1-c}}-1\big)\pi$, which respects the obstacle $\p\Omega$. The sub-level set $E_t$ is the bounded region enclosed by $\gamma_t$ and the two line segments joining 0 and $\gamma_t(0)$, $\gamma_t\big(\pi/\sqrt{1-c}\big)$. Note that $\gamma$ is embedded in $\RR^2$ only when $c\leq8/9$. For $c>8/9$, we may view $\gamma$ as being embedded in the universal cover of $\RR^2\setminus\{0\}$.
				
	This provides examples of solutions that approach $-\infty$ near $\p\Omega$.
\end{example}
			
\begin{example}[solution in bounded domains]\label{ex-ex:bounded_domain} {\ }
				
	Let $u$ be a solution in Example \ref{ex-ex:epicycloid} generated by an epicycloid. The restriction of $u$ inside $E_t(u)$ is another solution that respects the obstacle $\p E_t(u)$. This provides an example of nontrivial solutions in precompact domains. See Remark \ref{rmk-obs:definition}(vi) and Theorem \ref{thm-obs:max_principle} for related discussions.
\end{example}
			
\begin{example}[$c=3/4$, nephroid]\label{ex-ex:nephroid} {\ }
				
	In the special case $c=\frac34$ in Example \ref{ex-ex:epicycloid}, the supporting domain $\Omega$ becomes a half plane. The corresponding curve $\gamma$ is often called a nephroid. This particular example shows that the condition $\inf_\Omega(u)>-\infty$ in Theorem \ref{thm-halfplane:liouville} cannot be removed.
\end{example}

In \cite[Proposition 7.2]{Huisken-Ilmanen_2001}, it is proved that for a proper weak solution $u$ in $\RR^n$, the rescaled limit at infinity is the unique proper eternal solution in $\RR^n\setminus\{0\}$, which is $u_\infty=(n-1)\log|x|$. It is natural to ask whether similar facts hold under the presence of obstacles. Below we say that $Q\subset\RR^n$ is a conic domain with smooth link, if $\lambda Q=Q$ for all $\lambda>0$ and $\p Q\setminus\{0\}$ is a smooth hypersurface.

\begin{question}\label{qs-ex:asymp_expanding}
	Let $Q\subset\RR^n$ be a locally Lipschitz conic domain with smooth link. Then there is a unique expanding soliton that generates a solution of $\IMCFOOinthm{Q}{\p Q}$. Suppose $\Omega\subset\RR^n$ is smooth and asymptotic to $Q$ at infinity. Consider the maximal solution of $\IVP{\Omega;E_0}$ for some $E_0\Subset\Omega$. Then at spatial infinity, a suitably normalized blow-down sequence converges to the expanding soliton supported in $Q$.
\end{question}

On the other hand, shrinking solitons are related to the asymptotic limit near a rough boundary point, when the boundary converges nicely to a convex tangent cone. In this respect, we raise the following

\begin{question}\label{qs-ex:asymp_shrinking}
	Let $Q\subset\RR^n$ be a convex conic domain with smooth link. Then there is a unique shrinking soliton that generates a solution of $\IMCFOOinthm{Q}{\p Q}$. Let $\Omega\subset\RR^n$ be a bounded domain, with $0\in\p\Omega$, and $\p\Omega\setminus\{0\}$ is smooth, and $\Omega$ is asymptotic to $Q$ near $0$. Consider the maximal solution of $\IVP{\Omega;E_0}$ for some $E_0\Subset\Omega$. Then at $0$, a normalized blow-up sequence converges to the shrinking soliton supported in $Q$.
\end{question}

The solution $u$ in both questions are characterized as the maximal solution. Since the main existence Theorem \ref{thm-intro:existence} requires smoothness of the domain, we did not state that $u$ is a solution of $\IVPOO{\Omega;E_0}{\p\Omega}$, although this is expected.

\section{Background in geometric measure theory}\label{sec:gmt}

We collect relevant results in geometric measure theory, which are used in the main text. We assume that the reader is familiar with the notions of sets with locally finite perimeter and reduced boundary. For a detailed introduction on this topic, we refer to the books \cite{Ambrosio-Fusco-Pallara, Evans-Gariepy, Maggi}. Given a set $E$ with locally finite perimeter, we denote by $\p^*E$ the reduced boundary of $E$, and $\nu_E$ the measure-theoretic outer unit normal, and $|D\chi_E|$ the perimeter measure. De Giorgi's structure theorem implies that $|D\chi_E|=\H^{n-1}\res\p^*E$. For an open set $\Omega$, we denote $\Ps{E;\Omega}=|D\chi_E|(\Omega)$ the perimeter of $E$ inside $U$.

Given a measurable set $E$, and $\alpha\in[0,1]$, we define
\[\begin{aligned}
	& E^{(\alpha)}=\Big\{x\in M: \lim_{r\to0}\frac{|E\cap B(x,r)|}{|B(x,r)|}=\alpha\Big\}.
\end{aligned}\]
The set $E^{(1)}$ is often called the measure-theoretic interior. Federer proved that if $E$ has locally finite perimeter, then we have a disjoint union
\begin{equation}\label{eq-gmt:Federer_decomp}
	M=E^{(1)}\cup E^{(0)}\cup\p^*E\cup Z
\end{equation}
for some $Z$ with $\H^{n-1}(Z)=0$. See \cite[Theorem 16.2]{Maggi}. For sets $E,F$ with locally finite perimeters, such that $\H^{n-1}\big(\p^*E\cap\p^*F\big)=0$, we have the decomposition of perimeters
\begin{equation}\label{eq-gmt:set_operation}
	\begin{aligned}
		\Ps{E\cap F;K}& =\H^{n-1}\big(\p^*E\cap F^{(1)}\cap K\big)+\H^{n-1}\big(\p^*F\cap E^{(1)}\cap K\big), \\
		\Ps{E\cup F;K}& =\H^{n-1}\big(\p^*E\cap F^{(0)}\cap K\big)+\H^{n-1}\big(\p^*F\cap E^{(0)}\cap K\big), \\
		\Ps{E\setminus F;K}& =\H^{n-1}\big(\p^*E\cap F^{(0)}\cap K\big)+\H^{n-1}\big(\p^*F\cap E^{(1)}\cap K\big),
	\end{aligned}
\end{equation}
see \cite[Theorem 16.3]{Maggi}.

Let $\Omega$ be a domain in a Riemannian manifold $(M,g)$, and $\Lambda,r_0>0$ be constants. We call $E$ a $(\Lambda,r_0)$-perimeter minimizer in $(\Omega,g)$, if for all balls $B(x,r)\Subset M$ with $r\leq r_0$, and all comparison sets $F$ with $E\Delta F\Subset B(x,r)\cap\Omega$, we have
\begin{equation}\label{eq-gmt:Lam_r0_def}
	\P{E;B(x,r)}\leq \P{F;B(x,r)}+\Lambda|E\Delta F|.
\end{equation}
This definition follows \cite[Chapter 21]{Maggi} except that we do not assume $\spt(|D\chi_E|)=\p E$. It is well-known \cite{Almgren_1976, Maggi, Tamanini_1983} that the measure-theoretic boundary of a $(\Lambda,r_0)$-perimeter minimizer is a $C^{1,\alpha}$ ($\alpha<1/2$) hypersurface except for a codimension 8 singular set.

The following regularity and convergence results follow by combining (the Riemannian analogue of) Theorem 21.11, 21.14 and 26.6 in \cite{Maggi}, see also \cite{Tamanini_1983, Tamanini_1984}.

\begin{lemma}\label{lemma-gmt:Lam_r0_convergence}
	Let $(M,g)$ be a Riemannian manifold. Suppose the domains $\Omega_i$ converge to $\Omega$ locally, and the metrics $g_i\to g$ locally smoothly. Suppose $E_i\subset\Omega_i$ are $(\Lambda,r_0)$-perimeter minimizers in $(\Omega_i,g_i)$, and $x_i\in\spt(|D\chi_{E_i}|)$, such that $E_i\to E$ in $L^1_{\loc}$ and $x_i\to x$. Then:
	
	(i) $E$ is a $(\Lambda,r_0)$-perimeter minimizer in $(\Omega,g)$, and $x\in\spt(|D\chi_E|)$, and the upper density of $E$ at $x$ is bounded above by a dimensional constant $c<1$.
	
	(ii) If $x\in\p^*E$, then $x_i\in\p^*E_i$ for sufficiently large $i$, and we have $\nu_{E_i}(x_i)\to\nu_E(x)$.
\end{lemma}

Finally, we prove two technical lemmas that are used in the main text. We denote by $\{e_i\}$ the standard basis of $\RR^n$, and for a point $x=(x_1,\cdots,x_{n-1},x_n)\in\RR^n$, we write $x=(x',x_n)$ for short.

\begin{lemma}\label{lemma-gmt:contain_cone}
	Let $\Omega\subset\RR^n$ be a convex domain. Fix $\th\in(0,\pi/2)$. Suppose $E\subset\RR^n$ has locally finite perimeter. Moreover, assume that
	\[\essinf_{\p^*E\cap\Omega}\metric{\nu_E}{e_n}\geq\cos\th.\]
	If $E$ has nonzero lower density at a point $x=(x',x_n)\in\Omega$, then $E$ has density 1 on the set
	\begin{equation}\label{eq-gmt:contains_cone}
		\Omega\cap\Big\{(y',y_n)\in\RR^n: y_n<x_n-|y'-x'|\tan\th\Big\}.
	\end{equation}
\end{lemma}
\begin{proof}
	Let $\rho_\epsilon$ ($\epsilon>0$) be a family of standard mollifiers, and set $f_\epsilon=\chi_E*\rho_\epsilon$. Let $\Omega'\Subset\Omega$ be another convex domain. For any smooth vector field $X$ with $\supp(X)\subset\Omega'$ and any $\epsilon<d(\p\Omega',\p\Omega)$, we compute
	\[\begin{aligned}
		-\int_{\RR^n}\D f_\epsilon\cdot X &= \int_{\RR^n}f_\epsilon\div X
		= \int_E\div(\rho_\epsilon*X)
		= \int_{\p^*E}(\rho_\epsilon*X)\cdot\nu_E\,d\H^{n-1}.
	\end{aligned}\]
	Let $w$ be any constant vector field with $\metric{w}{e_n}\geq\sin\th$ (thus $\metric{\nu_E}{w}\geq0$ a.e.). Setting $X=\varphi w$ for arbitrary $\varphi\in C^\infty_0(\Omega')$, $\varphi\geq0$, we obtain
	\[-\int_{\RR^n}\varphi\frac{\p f_\epsilon}{\p w}=\int_{\p^*E}(\rho_\epsilon*\varphi)\metric{w}{\nu_E}\geq0.\]
	This implies $\frac{\p f_\epsilon}{\p w}\leq0$ in $\Omega'$. Let $\Theta>0$ be the lower density of $E$ at the given point $x$. For all sufficiently small $\epsilon$ we have $f_\epsilon(x)\geq\frac12\Theta$, and it follows by the convexity of $\Omega'$ that $f_\epsilon\geq\frac12\Theta$ in the region
	\[Q:=\Omega'\cap\Big\{(y',y_n): y_n<x_n-|y'-x'|\tan\th\Big\}.\]
	Taking $\epsilon\to0$, we conclude that $E$ has density at least $\frac12\Theta$ almost everywhere in $Q$. Therefore $E$ occupies full measure in $Q$. The result follows by taking $\Omega'\to\Omega$.
\end{proof}

\begin{lemma}\label{lemma-gmt:min_in_halfplane}
	Let $E\subset\{x_n<0\}$ be nonempty and locally perimeter-minimizing in $\RR^n$. Then $E$ must be a half space, i.e. there exists $c>0$ such that $E=\big\{(x',x_n): x_n<-c\big\}$.
\end{lemma}
\begin{proof}
	Choose $x\in\p^*E$. There exists a sequence $\lambda_i\to\infty$, such that the limit $E_\infty=\lim_{i\to\infty}\lambda_i(E-x)$ exists and is a perimeter-minimizing cone. Since $E_\infty\subset\{x_n\leq0\}$ and $0\in\operatorname{spt}\big(|D\chi_{E_\infty}|\big)$, it follows by a strong maximum principle \cite[Theorem 4]{White_2010} that $\p E_\infty=\{x_n=0\}$. Then by the classical monotonicity formula (see \cite[Theorem 28.17]{Maggi} for details), $\p E$ itself is a cone centered at $x$. Thus $\p E$ must be a flat hyperplane.
\end{proof}

\noindent\textit{Department of Mathematics, Duke University, Durham NC, USA.}

\vspace{3pt}

\noindent\textit{Email:} \href{mailto:kx35@math.duke.edu}{kx35@math.duke.edu}
	
\end{document}